\documentclass[12pt,english]{article}
\usepackage[T1]{fontenc}

\usepackage[latin9]{inputenc}
\usepackage{geometry}
\usepackage{amsthm}
\usepackage{amsmath}
\usepackage{amssymb}
\usepackage{graphicx}
\usepackage{setspace}
\usepackage{esint}
\usepackage{xr}
\newtheorem{mainthm}{\protect\theoremname}

\theoremstyle{plain}
\newtheorem{thm}{\protect\theoremname}[section]
  \theoremstyle{plain}
  \newtheorem{lem}[thm]{\protect\lemmaname}
  \theoremstyle{remark}
  \newtheorem*{rem*}{\protect\remarkname}
  \theoremstyle{definition}
  \newtheorem{defn}[thm]{\protect\definitionname}
  \theoremstyle{plain}
  \newtheorem{prop}[thm]{\protect\propositionname}
  \theoremstyle{plain}
  \newtheorem{cor}[thm]{\protect\corollaryname}
  \theoremstyle{remark}
  \newtheorem{claim}[thm]{\protect\claimname}
    \theoremstyle{remark}
  
  \theoremstyle{definition}
  \newtheorem*{example*}{\protect\examplename}

\usepackage{changepage}
\usepackage{caption}
\makeatother

\usepackage{babel}
  \providecommand{\claimname}{Claim}
    \providecommand{\conjname}{Conjecture}
  \providecommand{\corollaryname}{Corollary}
  \providecommand{\definitionname}{Definition}
  \providecommand{\examplename}{Example}
  \providecommand{\lemmaname}{Lemma}
  \providecommand{\propositionname}{Proposition}
  \providecommand{\remarkname}{Remark}
\providecommand{\theoremname}{Theorem}
\usepackage{color}

\begin{document}

\date{\today}

\begin{title}
{Near tangent dynamics in a class of Hamiltonian impact systems}

\end{title}
\author{M. Pnueli, V. Rom-Kedar}
\author{M. Pnueli$^1$, V. Rom-Kedar$^{1,2}$ \\
\normalsize $^1$ Department of Computer Science and Applied Mathematics, \\ \normalsize The Weizmann Institute of Science, Rehovot, Israel \\
\normalsize $^2$ The Estrin Family Chair of Computer Science and Applied Mathematics.
}
\date{\today}
\maketitle

\begin{abstract}
Tangencies correspond to singularities of  impact systems, separating between  impacting and non-impacting trajectory segments.  The closure of their orbits constitute the singularity set, which, even in the simpler billiard limit, is known to have a complex structure.  The properties of this set are studied in a class of near integrable two degrees-of-freedom Hamiltonian impact systems. For this class of systems, in the integrable limit, on iso-energy surfaces, tangency appears  at an isolated torus. We construct a piecewise smooth iso-energy return map for the perturbed flow near such a tangent torus. Away from the singularity set, this map has invariant curves, so, the  singularity set is included in a limiting singularity band. An asymptotic upper bound of this band  width   is found for both non-resonant and resonant tangent tori. In the Diophantine case the band extent to both the impacting and non-impacting regimes is, asymptotically, of order \(\epsilon\). In the resonance case its asymptotic extent to the non-impacting regime is at most of order \(\sqrt{\epsilon}\) whereas its maximal asymptotic extent to the impacting regime is of order \(\epsilon^{2/3}\). Studying numerically a truncated standard form of the return map which is integrable at \(\epsilon=0\)  and has a square-root singularity with a singularity parameter \(\alpha\), reveals that only the strongest resonance obeys this maximal extent resonance scaling. We find that for all but one resonant case the singularity band has connecting orbits from its lower to its upper boundaries. The number of iterates required for an orbit to visit close to both  the lower and upper boundaries of the singularity band diverges as \(\epsilon\alpha^{2}\) is decreased, giving rise to multitude of transient complex phenomena.   On the other hand, large  \(\alpha\) values at a fixed small \(\epsilon\) lead to a significant reduction in this  connection time.
\end{abstract}
% 37J40 (near integrable Ham.), 34A36 (discontinuous odes), 70H09 (Ham perturbations), %37C83   Dynamical systems with singularities (billiards, etc.)

\maketitle
%-Title

%+Contents
%\tableofcontents

\section{Background}

 Hamiltonian impact systems (HIS)  are piecewise smooth dynamical systems. The rules for evolution in time are the  smooth Hamiltonian dynamics inside a phase-space domain and some reflection law at the domains' boundary \cite{kozlov1991billiards,granados2012melnikov}.  For example, billiards are HIS where the Hamiltonian has only a diagonal kinetic energy term, the smooth dynamics is restricted, in the configuration space, to lie within a table \(D\), and  the reflection rule is  of elastic impacts at the tables' boundary. As for billiards, the existence of a  tangent trajectory segment  in an HIS is a singular phenomenon,  separating, locally, between nearby phase-space regions  - impacting and non-impacting trajectory segments. The study of the dynamics at and near all tangent orbits  is usually difficult.

In the context of general, far from integrable, piecewise smooth dynamical systems, the behavior near isolated tangent periodic and homoclinic orbits, called  grazing orbits,  have been extensively investigated, developing a non-smooth bifurcation theory for classes of such systems
  \cite{nordmark1991non,nordmark1992effects,nordmark2001existence,chillingworth2013periodic,di2001grazing,ivanov1994impact,ivanov1996bifurcations,lamba1995chaotic,turaev1998elliptic}. At the other extreme, grazing orbits in a one degree degrees-of-freedom (DOF) impact oscillator system with slowly moving walls, i.e. for a slow-fast piecewise smooth 1.5 DOF system, were shown to destroy adiabatic invariance due to the multiple crossings between impacting and non-impacting  motion \cite{neishtadt2008jump}.
KAM stability and return maps for an  harmonic oscillator subject to a discontinuous force and time-periodic forcing were analyzed in \cite{kunze1997application,kunze2001non}.

Here we study grazing orbits of a class of HIS - near integrable HIS of the mechanical type; For mechanical Hamiltonian impact systems, the smooth motion is governed by a smooth potential \(V\) and the motion  is restricted, as for billiards,  to lie within a table \(D\) in the configuration space, impacting  elastically  from the table's boundary (see, e.g. \cite{kozlov1991billiards,RK2014smooth}, for physical motivation).
For a general potential \(V\) and a general domain \(D\)  such an  HIS may be neither  integrable  nor hyperbolic. Yet, there are certain families of HIS of these two extreme types;  Hyperbolic impact systems appear, for example, at high energy, when the potential \(V\) is uniformly close to zero and the table \(D\) corresponds to  a hyperbolic billiard. Certain multi-particle systems moving in a linear potential were proved to be hyperbolic \cite{Wojtkowski1998,Wojtkowski1999} as well. At the other end,  when the the billiard table respects the continuous symmetries of the potential, one can expect  Liouville integrability \cite{pnueli2019structure}.

Once a class of Liouville integrable mechanical HIS has been identified for certain potentials and domains, it is natural to consider small deformations of either the potential or the table and study the properly defined near-integrable dynamics   \cite{pnueli2018near,pnueli2019structure}. Here we study the near tangent behavior for such near-integrable systems.  In the integrable limit, the system has an isolated tangent torus which separates between impacting and non-impacting branched leaves of \emph{level sets} on an energy surface (see \cite{pnueli2019structure} for a more general  setup).
 By constructing  a Poincar\'{e} section to the tangent torus which is transverse for all trajectories of interest - impacting, tangent and non-impacting alike - we are able to derive and analyze the piecewise smooth return map to this section. Moreover, since the section is also transverse to unperturbed tori that are bounded away  from the tangent torus,  and since under small conservative perturbation most of these tori persist  \cite{pnueli2018near},  a finite neighborhood of the tangent torus remains invariant,  separating between perturbed impacting and non-impacting regimes. We call this region the tangency band. We construct, analyse and simulate the local return map near the tangent torus, thus providing the  description of the dynamics in this band.

 The paper is ordered as follows: setup and notations are presented in section \ref{sec:intro}, along with an outline of the main results of this paper. In section \ref{sec:singulrcurve} we establish the properties of the singularity curve - the curve which separates between impacting and non-impacting returns to the  Poincar\'{e} section. In section \ref{sec:returnmap}, a piecewise smooth return map for near tangent level sets is constructed.  In section \ref{sec:stable} we establish  bounds on the tangency band widths. In section \ref{sec:truncatedmap} we present basic properties and numerical simulations for a truncated version of the return map, demonstrating the stability results of section \ref{sec:stable}, as well as proposing  various complex behaviors and long transient structures associated with near-tangent orbits. A discussion appears in section \ref{sec:discussion}.

\section{Set up and main results}\label{sec:intro}
\subsection{Setup}
Consider a smooth, separable 2 DOF mechanical Hamiltonian with a single straight wall which is perpendicular to one of the symmetry axes:
\begin{equation}
H=H(\cdot;\epsilon,b)=H_{int}(q_{1},p_{1},q_{2},p_{2})+\epsilon V_{c}(q_{1},q_{2})+b\cdot{}V_{b}(q_1;q_1^w)
\label{eq:hgeneral}
\end{equation}
\textit{The first term} $H_{int}$ represents the underlying integrable structure of the phase space:
\begin{equation}
H_{int}=\frac{||p||^{2}}{2}+V_{int}(q_{1},q_{2})=\frac{p_{1}^{2}}{2}+V_{1}(q_{1})+\frac{p_{2}^{2}}{2}+V_{2}(q_{2})
=H_{1}(q_{1},p_{1})+H_{2}(q_{2},p_{2}),
\label{eq:hint}
\end{equation}
and satisfies the S3B conditions:

\begin{defn}\label{def:s3b} \emph{Hamiltonians of the S3B (Separable, r-Smooth, Simple, Bounded level sets) class} are integrable, mechanical Hamiltonians of the form (\ref{eq:hint})
which satisfy the following conditions:
\begin{enumerate}
\item $V_{int}=V_{1}(q_{1})+V_{2}(q_{2})$ is a sum of two   potentials each depending on one variable,
\item $V_1,V_2$ are $C^{r+1}$ smooth,
\item $V_1,V_2$ have only a finite, discrete number of simple extremum points,
\item $V_1,V_2$ are bounded from below and go to infinity as $|q_1|,|q_2|\rightarrow{}\infty$ respectively.
\end{enumerate}
\end{defn}

By property 4,  S3B Hamiltonians have bounded energy surfaces. For simplicity of presentation, since all the analysis here is local near a regular torus of (\ref{eq:hgeneral}), assume hereafter that each of the potentials $V_i(q_i)$, $i=1,2$ has a single, simple minimum, located at $q_{ic}=0$ such that $V_i(q_{ic})=0$ (and so $\min H_{int}=0$), and are convex, so in particular $q_i\cdot V_i'(q_i)> 0$ for $q_i\neq 0$ (for potentials with  multiple number of extremal points, our results apply  to the  behavior near each regular torus of the smooth system belonging to a branch of the Liouville foliation of isoenergy surfaces, see  \cite{pnueli2018near}).
\\

\noindent \textit{The second term} in (\ref{eq:hgeneral}), $\epsilon V_c$, represents a small, regular perturbation by coupling and is assumed to be $ C^{r+1}$ smooth. Thus \(V_{c}(\cdot)\) is bounded on the bounded energy surfaces (see \cite{pnueli2018near}), and the bound generally depends on \(E\), the energy level. Hereafter we assume that \(E=\mathcal{O}(1)\).\\

\noindent \textit{The third term }in  (\ref{eq:hgeneral}),  the singular billiard potential $V_{b}$, corresponds  to a single vertical wall, located at the coordinate $q_1=q_{1}^w$. Motion occurs to the right of the wall, and at the wall the particle reflects elastically, so \(p_1\rightarrow-p_1\) at the wall. The elastic reflection is represented formally    as a singular energy barrier:
\begin{equation}
V_{b}(q_1;q_1^w)=\begin{cases}
0, & q_1>q_1^w\\
1, & q_1< q_1^w
\end{cases}\label{eq:billiardpot}
\end{equation}
where $b$ is either a fixed large number (representing an impassable wall for all energies of interest - can also be taken as infinity) or zero (representing the smooth Hamiltonian system without the impact).
More generally, it is possible to extend the analysis and consider a \(C^{r+1}\) smooth perturbation of the wall as in   \cite{pnueli2018near} and a steep smooth potential instead of a wall as in \cite{RK2014smooth,pnueli2018near}, however, for now, we present the near tangent analysis for the case of a single perpendicular wall and leave such  extensions to future studies.
Summarizing, for \(b>0\) we consider the   $ C^{r}$  -smooth near-integrable Hamiltonian flow  in the half plane \( q_1>q_1^w\) with  elastic impacts at the vertical wall boundary:

\begin{defn}\label{def:wallsys} \emph{The wall system} is a system of the form  (\ref{eq:hgeneral}), where \(H_{int}\) is of the S3B class, in Eq. (\ref{eq:hint}) $V_i(0)=0$ and $q_i\cdot V_i'(q_i)> 0$ for all $q_i\neq 0$, \(V_{1,2,c}\) are  $ C^{r+1}$ smooth with $r>6$, \(V_{b}\) is of the form (\ref{eq:billiardpot}),  the energy is strictly below \(b\) and \(q_1^{wall}<0\).
\end{defn}

We  denote  by \(\Phi^{\epsilon,im}_{t}\)  the wall system flow,  by  \(\mathcal{R}_{1}z^{w}=\mathcal{R}_{1}(q_{1}^{w},p^{w}_1,q^{w}_2,p^{w}_2)=(q_{1}^{w},-p^{w}_1,q^{w}_2,p^{w}_2)\) the elastic reflection from the wall (and, more generally, by \(\mathcal{R}_{i}\) the reflection of the \(i\) th momentum). We also denote by \(\Phi^{\epsilon}_{t}\)  the auxiliary smooth perturbed Hamiltonian flow associated with the wall system,  i.e.,  the smooth flow associated with  (\ref{eq:hgeneral}) when taking     \(b=0\) and  \(V_{i},i=1,2,c\) satisfy the same assumptions as in Definition \ref{def:wallsys}.

The integrable  unperturbed smooth flow, \(\Phi^{0}_{t}\)  admits global action-angle coordinates (e.g. \cite{Arnold2007CelestialMechanics}): $(J,\varphi)=S_{1}(q_{1},p_{1}), \ (I,\theta)=S_{2}(q_{2},p_{2})$, so $H_{int}(J,I)=H_{1}(J)+H_{2}(I)$, and the corresponding dynamics on level sets \((H_{1}(J),H_{2}(I))=(E_{1},E_2)\) are rotations on a torus with frequencies $\omega_1(J)=\frac{2\pi }{T_{1}(E_1)},\omega_2(I)=\frac{2\pi }{T_{2}(E_2)}$.
By the assumptions  on \(V_{i}\), the frequencies   \( \omega_i(\cdot)\) are bounded away from zero and the inverse functions \(J=H_{1}^{-1}(E_{1}), I =H_{2}^{-1}(E_{2}) \) are well defined and \(C^{r}\) smooth for all allowed level sets.
Since the wall is vertical and respects the separability symmetry of the underlying integrable Hamiltonian flow, by the rule of elastic reflection and the symmetry of the kinetic energy term, the partial energies \((E_{1},E_2)\) are preserved upon impact and the motion remains rotational on level sets.  Hence, the unperturbed wall system (\(\epsilon=0\) in   (\ref{eq:hgeneral})), is Liouville integrable \cite{pnueli2018near,pnueli2019structure}. This system has  singular tangent level sets \((E_{1},E_2)\) exactly when  $E_1=V_1(q_1^w)$ (then  $p_1=0$ at the wall), and these level sets correspond to a torus in the allowed region of motion for all  \(E_2>0\)  if and only if  $q_{1}^w<0$. We deal only with this case of a vertical wall in the left half plane (the case of   $q_{1}^w\ge0$ is left for future works). Then, on each energy surface with \(E>V_1(q_1^w)\)   there is a unique tangent level set \((E_{1},E_2)=(V_1(q_1^w),E-V_1(q_1^w))\) with the corresponding  \textit{tangency action}:
\begin{equation}\label{eq:Itan}
I_{tan}(E)=H_{2}^{-1}(E-V_{1}(q_1^w)).
\end{equation}
 of the motion in the \((q_{2},p_2) \) plane. Notably, for the unperturbed wall system, the  tangent torus divides the energy surface to impacting ($I<I_{tan}(E)$ ) and non-impacting ($I>I_{tan}(E)$ ) level sets.

\subsubsection*{The  return map}

The  cross-section $\Sigma=\{(q_1,p_1,q_2,p_2)|p_1=0,\dot{p}_1<0\}$ is transverse to the flow,  and, for any fixed \(\eta\) and sufficiently small \(\epsilon\), the condition \(\dot{p}_1<0\) is satisfied by the wall system for all \(q_{1}>\eta\).  On a given energy surface,  the coordinates $(q_2,p_2)$ on \(\Sigma_{E}=\Sigma|_{H(q,p)=E}\), are symplectic, as are the induced  action-angle coordinates $(\theta,I)=S_2(q_2,p_2)$.  The unperturbed Poincar\'{e} return map, $\mathcal{F}_0:\Sigma_{E}\rightarrow\Sigma_{E}$ on the half cylinder  \(\Sigma_{E}\) is of the form   \cite{pnueli2018near}:
\begin{equation}
\mathcal{F}_0:(\theta,I)\rightarrow(\bar{\theta},\bar{I}),\quad\ \text{ where}\qquad\ \begin{cases}
\bar{I}=I\\
\bar{\theta}=\theta+\Theta(I;E)
\end{cases}\label{eq:int_impact_returnmap}
\end{equation}
and
\begin{equation}
\Theta(I;E)=\omega_{2}(I)\cdot(T_{1}(E-H_{2}(I))-\Delta t_{travel}(E-H_{2}(I)))=\omega_2(I)\cdot\tilde{T}_1(E-H_{2}(I))
\end{equation}
with
\begin{equation}
\Delta t_{travel}(E_{1})=\begin{cases}
2\int_{q_{min}(E_{1})}^{q_1^w}\frac{dq_{1}}{\sqrt{2(E_{1}-V_{1}(q_{1}))}} & \mbox{impact }(E_{1}>V_1(q_1^w))\\
0 & \mbox{no impact, tangency,}
\end{cases}
\label{eq:deltattravel}\end{equation}
\begin{equation}
T_{1}(E_{1})=2\int_{q_{min}(E_{1})}^{q_{max}(E_{1})}\frac{dq_{1}}{\sqrt{2(E_{1}-V_{1}(q_{1}))}}, \quad\ V_1(q_{min,max}(E_{1}))=E_{1}.
\end{equation}
Here  $q_{min/max}(E_{1})$ is the minimal/maximal $q_1$ value of the  $H_1$-level set, so $q_{max}>0$ is on $\Sigma$ and, for impacting trajectories, $q_{min}<q_1^w<0$ is outside the billiard.
Notice that the choice of the cross-section \(\Sigma_{E}\) ensures that at most  a single impact with the wall may occur in between returns to the section    \cite{pnueli2018near}. With no loss of generality, we set the angle coordinates so that \(\phi=0\) on \(\Sigma\), and similarly,  \(\theta=0\) on the section \(p_{2}=0,\dot p_2<0\). Let \(z_{tan}(t,\theta;E) \) denote the unperturbed tangent trajectory with energy \(E\) and phase \(\theta\) at \(\Sigma_{E}\), i.e.  \(z_{tan}(0,\theta;E)\in\Sigma_{E}, \Pi z_{tan}(0,\theta;E)=(\theta ,I_{tan}(E))\) and \(\Pi z\) denotes hereafter the projection of \(z\) to the \((\theta,I)=S_{2}(q_2,p_2)\) plane.

The perturbed return map, $\mathcal{F}_\epsilon$, for sufficiently small \(\epsilon\) and away from tangency, i.e. for  $|I-I_{tan}(E)|>\rho$ for some positive \(\rho\), was constructed and studied in \cite{pnueli2018near}. In particular, it was established that this map is a \(C^{r}\) perturbation of the integrable map  $\mathcal{F}_0$, and that as long as  \( \rho >K\epsilon^{\alpha},\alpha<\frac{1}{2}\),  and the twist condition is satisfied, for sufficiently small \(\epsilon\) it admits KAM tori, namely invariant curves  \((\theta,I^{\epsilon}(\theta)) \) of the perturbed map     $\mathcal{F}_\epsilon$ which are bounded away from being  tangent (see Corollary 3.2  in \cite{pnueli2018near}).
Here, we study the behavior near the singularities of  $\mathcal{F}_\epsilon$, namely for trajectories which are not bounded away from tangency.

\subsection{Main results}

The non-smoothness
of  $\mathcal{F}_0$  near the  singular circle \(I_{tan}(E)\) requires special care in the construction of the perturbed dynamics, distinguishing between the return map for impacting and non-impacting trajectory segments. The dividing line between impacting and non-impacting segments is the singularity curve:
\begin{defn}
The \emph{singularity curve}, \(\sigma^{\epsilon}=\sigma^{\epsilon}(E)=\{(\theta^{\epsilon}_{tan}(s),I^{\epsilon}_{tan}(s))\}_{s\in S}\), is the set of  all initial conditions in $\Sigma_{E} $ which, under the wall system flow,  touch the wall tangentially before their first return to $\Sigma_{E}$. \end{defn}

\begin{defn}\label{def:singularityset}
 \emph{The  singularity set}, $\mathcal{S}^{\epsilon}=\bigcup_{k=-\infty}^\infty\mathcal{F}_\epsilon^k \sigma^{\epsilon}$,
is the invariant set of all backward and forward iterations of the singularity curve. Its closure is denoted by \(\mathcal{\bar S}^{\epsilon}\). \end{defn}
The parametric representation for the singularity curve  on some set \(S\) is introduced for generality of the discussion. Hereafter, we say that a parametric   curve \(\{(\theta(s),I(s))\}_{s\in S^{1}} \) is a dividing circle if \(\theta(s),I(s)\) are continuous and periodic in \(s\), the curve is not self intersecting, and  \(\{\theta(s)\}_{s\in S^{1}}=[-\pi,\pi]\).
For \(\epsilon=0\),  all the circles with a fixed \(I\) are invariant and a tangency occurs if and only if $I=I_{tan}(E)$. Thus the unperturbed  parametric singularity curve is identical to the singularity set and is given by the circle:  \(\sigma^{0}=\mathcal{\bar S}^{0}=\{(\theta_{tan}^0 ,I_{tan}^{0})|I_{tan}^{0}(s)=I_{tan}(E),\theta_{tan}^{0}(s)=s, s\in[-\pi,\pi]\}\), namely, at \(\epsilon=0\), the tangency curve is a graph. In section  \ref{sec:singulrcurve} we establish that for sufficiently small \(\epsilon\) the perturbed singularity curve and its first image are also smooth graphs:

\begin{mainthm}\label{prop:singularcurve}  For   \(E>V_1(q_1^w)\) and sufficiently small $\epsilon$, the singularity curve  \(\sigma^{\epsilon}(E)\) of the wall system is an isolated  dividing circle, dividing  the cylinder  $\Sigma_{E} $ to impacting (below the curve) and non-impacting (above the curve) regions of the first return map,  $\mathcal{F}_\epsilon $. Moreover, this circle is a graph:   $\sigma^{\epsilon}=\{(\theta,I_{tan}^{\epsilon}(\theta)),\theta\in[-\pi,\pi]\}$,  its image under the map is also  a graph,  $\mathcal{F}_\epsilon \sigma^{\epsilon}=\{(\theta,\bar{I}_{tan}^{\epsilon}(\theta)),\theta\in[-\pi,\pi]\}$ and both   \(I_{tan}^{\epsilon}(\theta)\) and \(\bar{I}_{tan}^{\epsilon}(\theta)\) are \(C^{r}\) \(\epsilon-\)close to the constant function \( I_{tan}(E)\). Finally, the singularity curve and its image are  related by the reflection symmetry:   \(\bar{I}_{tan}^{\epsilon}(\theta)=I_{tan}^{\epsilon}(-\theta)\). \end{mainthm}

Next, we study the extent of the singularity set. By Corollary 3.2 in \cite{pnueli2018near}, for any \(\rho>0\), for sufficiently small \(\epsilon\),
 there are KAM curves that remain at a distance \(\rho\) from \(\sigma^{0}\): the family of  KAM curves that are  \(\rho\)-below  \(\sigma^{0}\) consist of quasi-periodic orbits with transverse impacts whereas those KAM tori which are  \(\rho\)-above   \(\sigma^{0}\) consist of KAM curves of the smooth system  that remain bounded away from the wall. We thus  define:

\begin{defn}\label{def:tangencyband}
The \emph{tangency band}, \(\mathcal{B}^{\epsilon}\), is the smallest closed band,  a set enclosed in between two dividing circles, which  includes \(\mathcal{\bar S}^{\epsilon}\).  The lower boundary of \(\mathcal{B}^{\epsilon}\) is the impacting dividing circle, \(\{(I^{\epsilon}_-(s),\theta^{\epsilon}_{-}(s))\}_{s\in S^{1}} \)  and its upper boundary is the non-impacting  dividing circle \(\{(I^{\epsilon}_+(s),\theta^{\epsilon}_{+}(s))\}_{s\in S^{1}} \).

The \textit{maximal width} of the band is  \(W=\max_\theta (I^{\epsilon}_+(s)|_{\theta^{\epsilon}_{+}(s)=\theta}-I^{\epsilon}_-(s)|_{\theta^{\epsilon}_{-}(s)=\theta})\),   and its upper and lower  maximal widths are  \(W^{+}=\max_\theta (I^{\epsilon}_+(s)|_{\theta^{\epsilon}_{+}(s)=\theta}-I^{\epsilon}_{tan}(\theta))\) and  \(W^{-}=\max_\theta (I^{\epsilon}_-(s)|_{\theta^{\epsilon}_{-}(s)=\theta}-I^{\epsilon}_{tan}(\theta))\) respectively so  \(W^{\pm}\leqslant W \leqslant W^++W^-\).      \end{defn}
  Corollary 3.2  in \cite{pnueli2018near} implies that for \(\alpha<\frac{1}{2}\), for sufficiently small \(\epsilon\), there exist a constant \(C\) (depending on the parameters and energy) such that \(\mathcal{B}^{\epsilon}\subset\{(\theta,I )|I\in [I_{tan}(E)-C\epsilon^\alpha,I_{tan}(E)+C\epsilon^\alpha],\theta\in[-\pi,\pi]\}\), namely,  one can take  \(\rho \) to approach zero slower than  \(\sqrt{\epsilon}\) to establish the persistence of bounding  KAM curves which are at a distance \(\rho\) from the singularity line. Thus, there exists \(C>0\) such that
  \(W\leqslant C\epsilon^\alpha\) for all  \(\alpha<\frac{1}{2}\). Here, we construct  and study  $\mathcal{F}_\epsilon$ in the singularity band   and establish tighter bounds on the upper and lower singularity band widths.

First, we establish that the division of     $\Sigma_{E}$  to impacting and non-impacting initial conditions allows to approximate the return map in each part  by employing perturbation methods, leading to the main analytical result of this paper, established in section \ref{sec:returnmap}:

\begin{mainthm}\label{thm:mainthm}
For \(E>V_1(q_1^w)\), for sufficiently small \(\epsilon\),  the return map $\mathcal{F}_\epsilon:\Sigma_E\rightarrow\Sigma_{E} $ of the wall system flow near the singularity set can be brought, by a \(C^{r}\) smooth symplectic change of coordinates, to the symplectic form:
\begin{equation}\label{eq:pert-return}
\mathcal{F}^{loc}_\epsilon:\begin{cases}
\bar{K}=K+\epsilon f(\bar{\phi})+\mathcal{O}_{C^{r-2}}(\epsilon G_{s}(K)) \\
\bar{\phi}=\phi+\Omega+G_{s}(K)+\mathcal{O}_{C^{r-2}}(\epsilon,\epsilon G_{s}(K),K^{2})
\end{cases}
\end{equation}
where $\epsilon f(\bar{\phi})=\bar{I}^{\epsilon}_{tan}(\bar{\phi})-I^{\epsilon}_{tan}(\bar{\phi})=\epsilon\int_0^{T^{tan}_1(E)}\{I,V_c\}\mid_{z_{tan}(t,\theta;E)} dt+\mathcal{O}_{C^{r}}(\epsilon^{2})$
is the difference between the pre-image and the forward image of the tangency curve at \(\bar{\phi}\), $\Omega=\Omega(E)=2\pi\frac{T_1(V_{1}(q_1^w))}{ T_{2}(E-V_{1}(q_1^w))} =\Theta(I_{tan}(E);E)$ is the unperturbed rotation number   at the unperturbed tangent action (\ref{eq:Itan}),  $G_{s}(K)$ is a continuous piecewise smooth function with a square root singularity at the origin:
\begin{equation}
G_{s}(K)=\begin{cases}\tau(I_{tan}(E))\cdot K+\mathcal{O}_{C^{r}}(K^2) &K\geq 0 \\ -\frac{(2\omega_2(I_{tan}(E)))^{3/2}}{|V_1'(q_1^w)|}\cdot\sqrt{-K}+\tau(I_{tan}(E))\cdot K+\mathcal{O}_{C^{r-2}}((-K)^{3/2},K^2) &K<0\end{cases}\label{eq:gsindef}
\end{equation}
and
\begin{equation}\label{eq:twist}
\tau(I)=\frac{d}{dI} \left(\omega_2(I)\cdot T_1(E-H_{2}(I))\right)
\end{equation}
is the twist of the return map of the unperturbed smooth Hamiltonian system.
\end{mainthm}

By \( \mathcal{O}_{C^{r}}(a(\epsilon),b(K),\epsilon G_{s}(K))\) we mean that for sufficiently small \(\epsilon,|K|\) this term may be represented by  a  \(C^{r}\) smooth function of its arguments,  \( g(a(\epsilon),b(K),\epsilon G_{s}(K))\), satisfying \( g(0,0,0)=0\). Namely, this term is small and its  \(j\)-th derivative w.r.t. \(K\) can grow at most as  \(\epsilon^j G_{s}^{(j)}(K)\) (and similar growth estimates apply to the singular terms in \(G_{s}(K)\)).    The coordinates $(K,\phi)$ are related to the  $(I,\theta)$ coordinates by a \(\theta\)-dependent shift of the singularity action, so that for all small \(\epsilon\) the singularity curve in these coordinates is the circle $K=0,\phi\in[-\pi,\pi]$, impacts correspond to negative \(K\),   the region near the singularity line corresponds to small \(|K|\) values, and the upper and lower widths of the singularity band in the \(K\) variable are simply the maximal positive and minimal negative \(K\) values of the singularity set.

In Theorem \ref{thm:stability} we establish that for sufficiently small \(\epsilon\), the map \(\mathcal{F}^{loc}_\epsilon\)   has  invariant KAM  curves that remain, for sufficiently small \(\epsilon\), at a signed distance  $\pm\epsilon^{\gamma_{\pm}(\Omega)}$ away from the singularity line $K=0$; namely, there exist constants \(\gamma_{\pm}(\Omega)>0\) and constants \(k_{\pm}(\Omega,\gamma_{\pm}(\Omega)),\bar k_{\pm}(\Omega,\gamma_{\pm}(\Omega))>0\), such that  \(\mathcal{F}^{loc}_\epsilon\)    has KAM  curves
\( (K_{\pm}(\phi),\phi) \), satisfying \(k_{+}\epsilon^{\gamma_{+}}<K_{+}(\phi)<\bar k_{+}\epsilon^{\gamma_{+}}\) and \(-\bar k_{-}\epsilon^{\gamma_{-}}<K_{-}(\phi)<-k_{-}\epsilon^{\gamma_{-}}\). For Diophantine  $\Omega$ values \(\gamma_{\pm}(\Omega)\lesssim1\) whereas for resonant  $\Omega$ values \(\gamma_{+}(\Omega)\lesssim1/2\)   and  \(\gamma_{-}(\Omega)\lesssim2/3\). These provide a refinement of the results of Corollary 3.2  in \cite{pnueli2018near} which were obtained by analysing the  map \(\mathcal{F}_\epsilon\) away from the tangency zone. This implies:
\begin{mainthm}\label{thm:mainthm3}
For \(E>V_1(q_1^w)\), for sufficiently small \(\epsilon\),  the width of the tangency band for a Diophantine  $\Omega$ value is
of $\mathcal{O}(\epsilon)$ whereas for resonant values of  $\Omega$ the tangency band is asymmetric, with its upper width being at most of order  $\mathcal{O}(\sqrt{\epsilon})$ and its lower width  being at most of   $\mathcal{O}(\epsilon^{2/3})$.
\end{mainthm}

To examine whether these bounds are realized and  to study the dynamics in the tangency band we examine in section \ref{sec:truncatedmap} a  truncated version of the map   (\ref{eq:pert-return}), the truncated tangency map \(F_{TTM}\), in which we set \(g_{1}=g_2=0 \):\begin{equation}\label{eq:truncatedreturn}
F_{TTM}:\begin{cases}
\bar{K}=K+\epsilon f(\bar{\phi}) \\
\bar{\phi}=\phi+\Omega+\tau\cdot K-\alpha\sqrt{-K}\cdot Heav(-K)
\end{cases}
\end{equation}
  \(Heav(K)\) denotes the heaviside function, and by re-scaling \(K \) we take $\tau\in\{1,-1\}$.
After establishing some basic properties of this map (fixed points and twist), we examine numerically, for \(f(\bar{\phi})=\sin(\bar{\phi})\) (i.e. for the first  Fourier mode of \(f)\),  the width and structure of the tangency band of the truncated map.  We observe that the widths scaling with \(\epsilon\) near \(K=0\) for resonant and Diophantine \(\Omega\) are very similar as long as \(\Omega\) is not close to the strongest resonance cases, $\Omega\approx 2\pi \text{ (mod $2\pi$)}$. So we call the non-strong-resonance cases  general-\(\Omega\) cases. For the strong resonance cases, the map  (\ref{eq:truncatedreturn}) has fixed points which are close to the singularity line,  undergoing different types of bifurcations in  the positive shear  ($\tau=1$) and the negative shear ($\tau=-1$) cases. Thus, we divide our numerical observations to three cases: general-\(\Omega\) , positive shear strong resonance and negative shear strong resonance cases:

\begin{claim}\label{conj:bandwidth}
For the general-\(\Omega\) cases  the tangency band width is of $\mathcal{O}(\epsilon)$, whereas  for the strong resonance cases its upper width is of order  $\mathcal{O}(\sqrt{\epsilon})$ and its lower width is of   $\mathcal{O}(\epsilon^{2/3})$.
\end{claim}
\begin{claim}\label{conj:resislands}
The set  $\mathcal{B}^{\epsilon}\setminus \mathcal{\bar S}^{\epsilon}$ may include  two sided islands:  stable periodic orbits surrounded by quasi-periodic orbits that  hop between impacting and non-impacting returns to \(\Sigma_E\). \end{claim}
\begin{claim}\label{conj:nonresconnectingorbit}
For the general-\(\Omega\) cases  and for the positive shear strong resonance case, the tangency band includes a connecting orbit which visits arbitrarily close the upper and lower boundaries of $\mathcal{B}^{\epsilon}$, yet, the connecting time, \(N_{c}(\epsilon,\alpha,\Omega,\tau)\), diverges as \(\epsilon\) and/or \(\alpha\) are decreased.\end{claim}

 The implication of these large transient connection times is that for very long periods (e.g., for \(\epsilon=0.01\) and \(\alpha=1\), for more than  \(10^7\) iterates) the density of orbits is non-uniform. During these transient times orbits   cover essentially densely  overlapping yet different parts of the same ergodic component. Interestingly, for the negative shear strong resonance case, at \(\Omega\) values which are slightly below \(2\pi k\), there are $\alpha$ values for which no such connecting orbits are found, and,  some orbits belonging to the singular set do not overlap even after \(10^{10}\) iterates. We conclude that even if there is a connecting orbit,  \(N_{c}(\epsilon,\alpha,2\pi,-1)\)  is much larger than all other cases. Whether such a connecting orbit exists, or, alternatively, the singularity set is divided to several invariant subsets for this case, remains an open question.

\section{The singularity curve\label{sec:singulrcurve}}
Recall that in the unperturbed setting, the iso-energy singularity curve, \(\sigma^{0}(E)\subset \Sigma_{E}\), is  the invariant circle
 $\sigma^{0}(E)=\{(\theta,I_{tan}(E))$, $\theta\in[-\pi,\pi]\}$, and it separates \( \Sigma_{E}\) to impacting ($I<I_{tan}(E)$) and non-impacting ($I>I_{tan}(E)$) circles. Here we prove Theorem \ref{prop:singularcurve} which asserts that  for sufficiently small \(\epsilon>0\) the singularity curve \(\sigma^{\epsilon}(E)\) (which is generically not invariant) has similar dividing properties and is a smooth graph. We first characterize the curve which divides between impacting and non-impacting points at  \(\Sigma_{E}^* \), a properly defined section near the wall,  and then integrate this wall tangent curve backward and forward to \(\Sigma_{E}\) to establish the properties of the singularity curve and its image.

\subsection{Characterization of tangency at a section near the wall}

Define an auxiliary iso-energy cross-section \(\Sigma_{E}^* \) near the wall (see Figure \ref{fig:unperttangency}): \begin{defn}
The iso-energy  \textit{star-section} $\Sigma_{E}^*$ is the  union of the inward impact wall section: $\Sigma_{E}^{w}=\{(q,p)|H(q,p)=E, q_1=q_1^w,p_1\leqslant0\}$ and the by-pass section    $\Sigma_{E}^{>}=\{(q,p)|H(q,p)=E,p_{1}=0,\dot{p}_{1}>0,q_1\geq q_1^w\}$, namely $\Sigma_{E}^*=\Sigma_{E}^{>}\cup\Sigma_{E}^w$.
The \textit{auxiliary section} (which may be crossed only by the auxiliary smooth flow) is the section    $\Sigma_{E}^{<}=\{(q,p)|H(q,p)=E,p_{1}=0,\dot{p}_{1}>0,q_1<q_1^w\}$.\end{defn}

\begin{figure}[ht]
\begin{centering}
\includegraphics[scale=0.2]{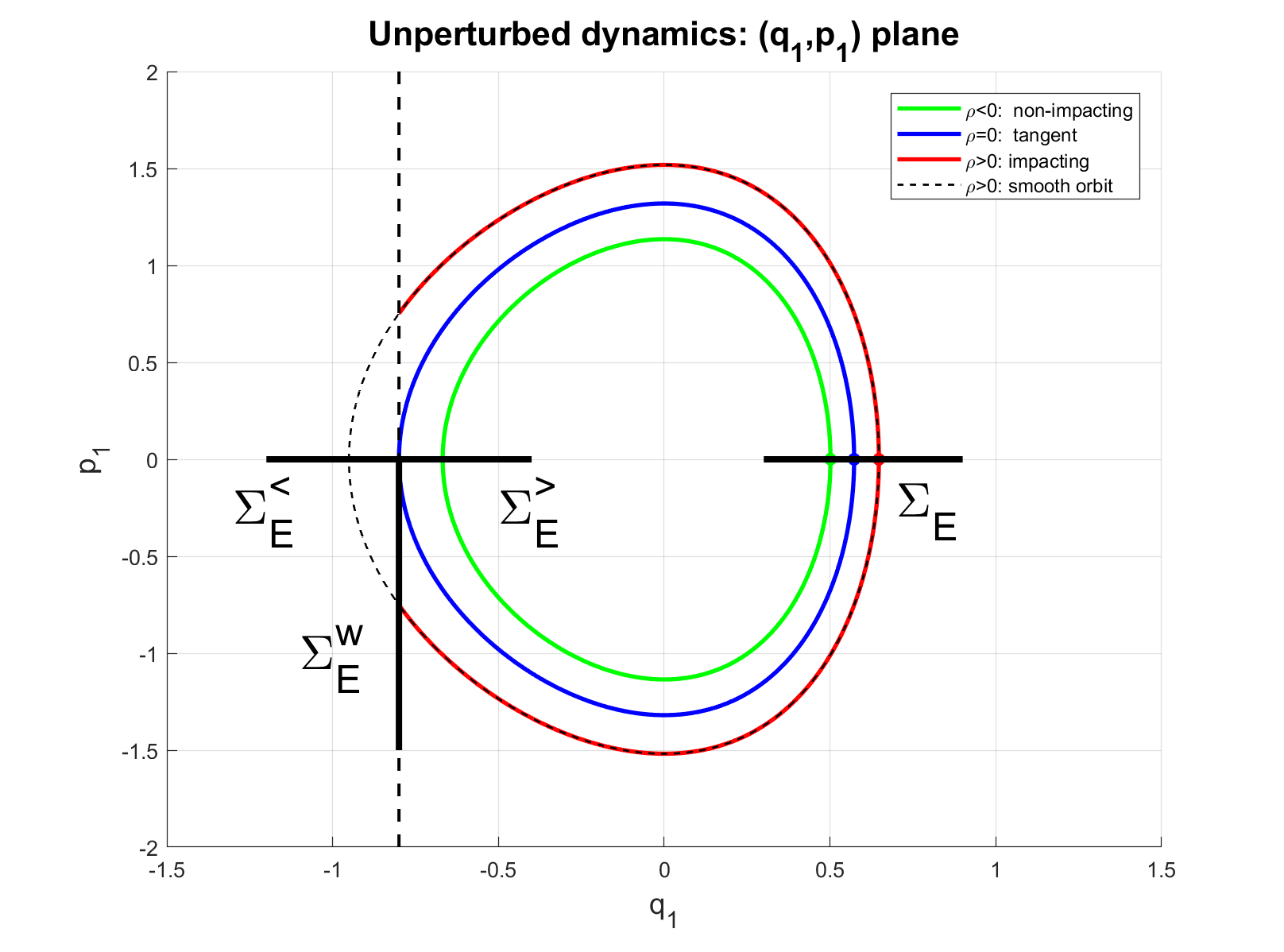}
\includegraphics[scale=0.2]{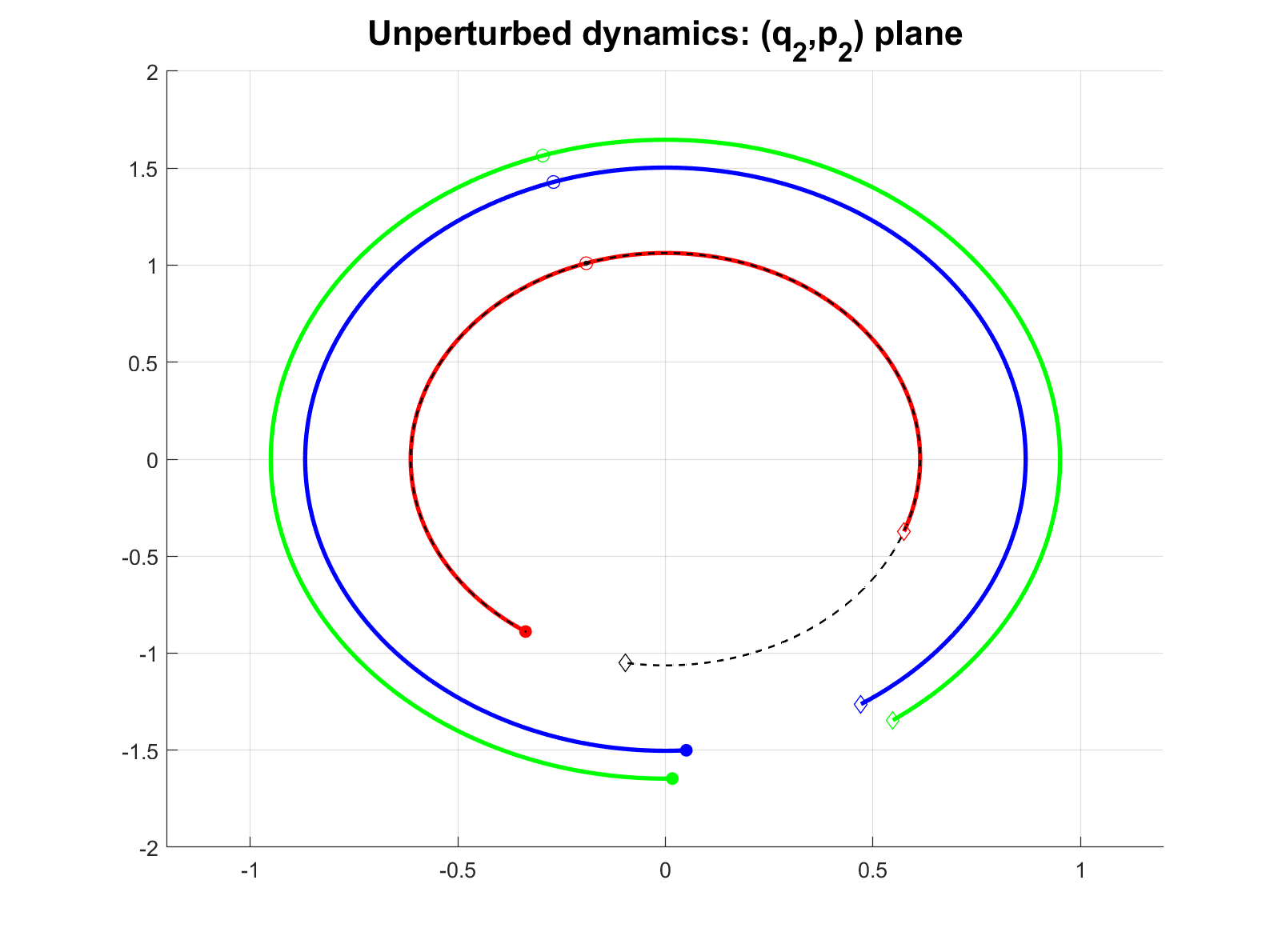}
\par\end{centering}
\protect\caption{\label{fig:unperttangency}Near tangent dynamics in the $(q_1,p_1)$ plane for a single iteration of the unperturbed return map $\Sigma_{E}\rightarrow\Sigma_{E}$. Filled circles - indicate i.c., circles - position at the wall, diamonds - end points, i.e. the image of the return map. The phase \(\theta\) gained by the impact flow (red diamond) is smaller than the phase gained by the smooth flow (black diamond).    Here  $q_1^w=-0.8$,  $V_1(q_1)=2q_1^2+q_1^3+\frac{1}{4}q_1^4$, $V_2(q_2)=\frac{3}{2}q_2^2, H=2$ and $\epsilon=0$.}
\end{figure}
For all \(E>0\), the piecewise smooth two dimensional star section \(\Sigma_{E}^*\) is well defined and is non-empty. For    \( E<V_1(q_1^w) \) and sufficiently small \(\epsilon\), the wall section, \(\Sigma_{E}^{w}\) is the empty set and the star section consists of the by-pass section alone:  for such energies all segments by-pass the wall without an impact.  For     \( E> V_1(q_1^w) \) both \(\Sigma_{E}^{>},\Sigma_{E}^w\) are non-empty and
include the wall iso-energy tangent curve: \begin{lem}\label{lem:tangentcurve}
The collection of all tangential wall positions with energy \(E>V_1(q_1^w)\) is given by:
\begin{equation}\label{eq:tancurvepert}
\Xi_{wall}^{\epsilon}(E)=\{(q_{1}=q_1^{w},p^{w}_1=0,q_2,p_2)|E-V_1(q_1^w)=\frac{p_2^2}{2}+V_2(q_2)+\epsilon\cdot V_c(q_1^w,q_2)\}.
\end{equation}
 For sufficiently small \(\epsilon\) the projection \(\Pi\) of this curve to the \((q_{2},p_2)\) plane, the iso-energy wall tangency curve,  \(\sigma_{wall}^{\epsilon}(E)=\Pi \Xi_{wall}^{\epsilon}(E)\), is a circle which can be represented as a graph in the unperturbed action angle coordinates \((I,\theta)\):
\begin{equation}\label{eq:tancurvepertitheta}
\sigma_{wall}^{\epsilon}(E)=\{(\theta ,I_{tan,\Sigma^*_E}^{\epsilon}(\theta))|\theta\in[-\pi,\pi]\}
\end{equation} where \(I_{tan,\Sigma^*_E}^{\epsilon}(\theta)=I_{tan}(E)+O_{C^{r}}(\epsilon)\) is  a \(C^{r}\)-smooth function.
\end{lem}

\begin{proof}
Since the wall is the vertical line \(q_{1}=q_1^w\), a tangent wall position occurs  if and only if \(q_{1}=q_1^{w},p^{w}_1=0\), and (\ref{eq:tancurvepert}) follows from plugging these conditions in the perturbed Hamiltonian (\ref{eq:hgeneral}). The obtained curve is a circle:
from the S3B property, the potential $V_2(q_2)$ has bounded level sets, so, for all  \( H>V_1(q_1^w),  \) at \(\epsilon=0\), Eq. (\ref{eq:tancurvepert}) defines a circle in the \((q_2,p_2)\) plane:\begin{equation}\label{eq:p2taneps0}
\{(q_{2},p_{2})|p_{2}=p_2^{tan,\pm,0}(q_2;H)=\pm\sqrt{2\left(H-V_1(q_1^w)-V_2(q_2)\right)},q_2\in[ q_{2min}^{0}(E),q_{2max}^{0}(E)]\}.
\end{equation} Defining   \((I_{tan,\Sigma^*_E}^{0}(q_{2},p_{2}),\theta^{0}_{tan,\Sigma^*_E}(q_{2},p_{2}))=S_{2}(q_{2},p_2^{tan,\pm,0}(q_2;E)),\) the unperturbed singularity curve at the wall representation in the action angle coordinates  is \(I_{tan,\Sigma^*_E}^{0}(q_{2},p_{2})=I_{tan}(E)\) and \(\theta^{0}_{tan,\Sigma^*_E}\in[-\pi,\pi]\). In particular, the constant function \(I_{tan,\Sigma^*_E}^{0}(q_{2},p_{2})=I^{0}_{tan,\Sigma^*_E}(\theta^{0}_{tan,\Sigma^*_E}(q_{2},p_2^{tan,\pm,0}(q_2;E)))\) is a graph of     \(\theta^{0}_{tan,\Sigma^*_E}\), satisfying
 \(\frac{dI^{0}_{tan,\Sigma^*_E}(\theta)}{d\theta}=0.\)

For all   \( H>V_1(q_1^w) \), for $\epsilon$ sufficiently small, this property is preserved by $V_2(q_2)+\epsilon V_c(q_1^w,q_2)$; away from  \( q_{2min,2max}^{0}(E)\) the solution to  (\ref{eq:tancurvepert})  depends \(C^{r+1}\) - smoothly on \(\epsilon\) and is of the form
\((q_{2},p_2^{tan,\pm,\epsilon}(q_2;E),q_2)\):\begin{equation}
p_2^{tan,\pm,\epsilon}(q_2;E)=\pm\sqrt{2(E-V_1(q_1^w)-V_2(q_2)-\epsilon\cdot V_c(q_1^w,q_2))}.\label{eq:p2tanpm}
\end{equation} Close to the turning points, since  \(V_2'(q_{2min/max}^0(E))\neq0\), the solution to  (\ref{eq:tancurvepert})  depends \(C^{r+1}\) - smoothly on \(\epsilon\) and is of the form
\((q_2^{tan,\pm,\epsilon}(p_2;E),p_{2})\): in particular, there exist $q_{2min/max}^{\epsilon}(E)=q_2^{tan,\pm,\epsilon}(0;E)\approx q_{2min/max}^0(E)+O_{C^{r}}(\epsilon)$ such that
\(\Xi_{wall}^{\epsilon}(E)=\{(q_{1}^{w},0,q_{2},p_{2})|  p_{2}=p_2^{tan,\pm,\epsilon}(q_2;E),q_2\in[ q_{2min}^{\epsilon}(E),q_{2max}^{\epsilon}(E)]\} \),
so indeed we have a circle (both \(\sigma_{wall}^{\epsilon}(E)\) in the \((q_{2},p_2)\) plane and  \(\Xi_{wall}^{\epsilon}(E)\) in the full phase space).
Defining \((\theta_{tan,\Sigma^*_E}^{\epsilon}(q_{2}),I_{tan,\Sigma^*_E}^{\epsilon}(q_{2}))=S_{2}(q_{2},p_2^{tan,\pm,\epsilon}(q_2;E))\) in the chart of  \(q_2\) values bounded away from the turning points and \((,\theta_{tan,\Sigma^*_E}^{\epsilon}(p_{2}),I_{tan,\Sigma^*_E}^{\epsilon}(p_{2}))
=S_{2}(q_2^{tan,\pm,\epsilon}(p_2;E),p_{2})\) at the chart near the turning points establishes the existence of a \(C^{r}\) smooth action angle presentation of the tangent circle (\ref{eq:tancurvepertitheta}) for  all \(E>V_1(q_1^w)\) and sufficiently small \(\epsilon\), and  its \(C^{r}\) closeness to the circle  \((\theta ,I_{tan,\Sigma^*_E}^{0}(\theta))\). In particular, in the first chart   \(\frac{dI_{tan,\Sigma^*_E}^{\epsilon}(q_{2})}{d\theta_{tan,\Sigma^*_E}^{\epsilon}(q_{2})}=\frac{dI_{tan,\Sigma^*_E}^{\epsilon}(q_{2})/dq_{2}}{d\theta_{tan,\Sigma^*_E}^{\epsilon}(q_{2})/dq_{2}}=\frac{O_{C^{r-1}}(\epsilon)}{d\theta_{tan,\Sigma^*_E}^{0}(q_{2})/dq_{2}+O_{C^{r-1}}(\epsilon)}=O_{C^{r-1}}(\epsilon)\) (since away from the turning points \(d\theta_{tan,\Sigma^*_E}^{0}(q_{2})/dq_{2}\) is bounded away from \(0\)), and similar calculation applies to the second chart near the turning points. It thus follows that the circle can be represented as a graph in the unperturbed action angle coordinates representation.

\end{proof}

 Next, we establish that the wall tangency circle divides the plane $\Pi\Sigma_{E}^*$ to impacting and non-impacting trajectory segments:
\begin{lem}\label{lem:divideimpwall}
For sufficiently small  $\epsilon$ the $(q_2,p_2)$ coordinates  parametrize $\Sigma^*_E$: to each point $(q_2,p_2)\in\Pi\Sigma^*_E$ there exists a unique point \(z\in\Sigma^*_E\). The wall tangency circle, \(\sigma_{wall}^{\epsilon}(E)\),     divides \(\Pi\Sigma^*_E\) between  iso-energy impacting (interior points to \(\sigma_{wall}^{\epsilon}(E)\) corresponding to $z\in\Sigma^w_E$) and non-impacting points  (exterior points to \(\sigma_{wall}^{\epsilon}(E)\), corresponding to $z\in\Sigma^>_E$)  on $\Sigma_E^*$.
\end{lem}

\begin{proof}
A trajectory impacts if and only if it has sufficient kinetic energy in the $q_1$ direction to reach the wall $q_1=q_1^w$. In other words, for a given energy $E$, for a given $(q_2,p_2)$, impact happens iff:
\begin{equation}
\frac{p_1^2}{2}\overset{at\ wall}{=}E-V_1(q_1^w)-\frac{p_2^2}{2}-V_2(q_2)-\epsilon V_c(q_1^w,q_2)\geq 0
\end{equation}
Or equivalently
\begin{equation}\label{eq:insidecurveq2p2}
\frac{p_2^2}{2}+V_2(q_2)+\epsilon V_c(q_1^w,q_2)\leq E-V_1(q_1^w)
\end{equation}
In this case, $z^w=(q_1^w,p_{1}^{w}=-p^{\epsilon}_1(q_{2},p_{2},E)=-\sqrt{2(E-V_1(q_1^w)-\frac{p_2^2}{2}-V_2(q_2)-\epsilon V_c(q_1^w,q_2))},q_2,p_2)\in\Sigma_{E}^{w}$  is the unique intersection point with $\Sigma_E^*$, and the case of equality in (\ref{eq:insidecurveq2p2}) sets $p_1=0$, so for this case only, this point also belongs to $\Sigma_{E}^{>}$. It follows from (\ref{eq:insidecurveq2p2}) that in the $(q_2,p_2)$ plane such points are inside (and respectively, on) the iso-energy tangency curve.

On the other hand,  consider a $(q_2,p_2)$ point which is  outside the tangency curve, namely, there exists some \(\eta>0\) such that:
\begin{equation}\label{eq:q2p2eta}
 \frac{p_2^2}{2}+V_2(q_2)+\epsilon V_c(q_1^w,q_2)=E-V_1(q_1^w)+\eta .
\end{equation}We need to show that for such a point there exist a unique \(q_{1,min}^{\epsilon,\eta}\in(q_{1}^{w},0)\) such that  $(q_{1,min}^{\epsilon,\eta},p_1=0,q_2,p_2)\in\Sigma_{E}^{>}$. By the S3B assumption, for any \(\eta\in(0,V_{1}(q_1^{w}))\)
there exists   \(q_{1,min}^{0,\eta}\in( q_{1}^{w},0)\)  such that
\(
E-V_1(q_{1,min}^{0,\eta})=E-V_1(q_1^w)+\eta
\), i.e. \(V_1(q_{1,min}^{0,\eta})=V_1(q_1^w)-\eta\)
and  \(V_1'(q_{1,min}^{0,\eta})<0\). Thus, by the inverse function theorem, for $\epsilon$ sufficiently small, there exists a unique $q_{1,min}^{\epsilon,\eta}=q_{1,min}^{0,\eta}+O(\epsilon)$ for which:
\begin{equation}\label{eq:q1etaeps}
\frac{p_2^2}{2}+V_2(q_2)+\epsilon V_c(q_{1,min}^{\epsilon,\eta},q_2)=E-V_1(q_{1,min}^{\epsilon,\eta}).
\end{equation}For any fixed \(\eta>0 \), taking \(\epsilon=o(\eta)\) sufficiently small, we conclude that    \(q_{1,min}^{\epsilon,\eta}\in(q_{1}^{w},0)\)  and indeed $(q_{1,min}^{\epsilon,\eta},p_1=0,q_2,p_2)\in\Sigma_{E}^{>}$.  Now, we need to show that for \((q_{2},p_{2)})\) values which satisfy (\ref{eq:q2p2eta}) with arbitrary small positive  \(\eta\) the same conclusion applies, namely that the solution to (\ref{eq:q1etaeps}) satisfies \(q_{1,min}^{\epsilon,\eta}\in(q_{1}^{w},0)\). Indeed, subtracting the two equations we obtain \begin{equation}
 \eta(q_{1,min}^{\epsilon,\eta},q_{2})=V_1(q_1^w)-V_1(q_{1,min}^{\epsilon,\eta})+\epsilon( V_c(q_1^w,q_2)- V_c(q_{1,min}^{\epsilon,\eta},q_2))
\end{equation}
so \(\eta\) can vanish only if \(V_{1}'(q_1^w)=O(\epsilon)\)
which is false by the S3B assumption and the assumption that \(q_{1}^w<0\). Since for small \(\eta=\eta(\epsilon)\),   \(q_{1,min}^{0,\eta}=q_1^w+\frac{\eta }{|V_1'(q_1^w)|}+O(\eta^2)\), we conclude that  \( q_{1,min}^{\epsilon,\eta}=
q_1^w+\bar q_{1}\) where \( \bar q_{1}=O(\epsilon,\eta)\). Plugging in this expression in  (\ref{eq:q1etaeps}) and using (\ref{eq:q2p2eta}), we find that for arbitrary small positive \(\eta\), the next order term,  \( \bar q_{1}\), must be positive: \begin{equation}
\bar q_{1}=\frac{-\eta }{V_1'(q_1^w)}+O_{C^{r-1}}(\epsilon\bar q_{1}^{2})+O_{C^{r-1}}(\bar q_{1}^{2})+O_{C^{r}}(\epsilon\eta)>0
\end{equation}
%\begin{equation}\begin{split}
% \frac{p_2^2}{2}+V_2(q_2)+\epsilon V_c(q_1^w,q_2)+\epsilon\frac{\partial %V_{r}(q_1^w,q_{2})}{\partial q_{1}} \bar q_{1}+O_{C^{r-1}}(\epsilon\bar %q_{1}^{2})&=E-V_1(q_1^w)-V_1'(q_1^w)\bar q_{1}+O_{C^{r-1}}(\bar q_{1}^{2})\\E-V_1(q_1^w)+\eta+\epsilon\frac{\partial %V_{r}(q_1^w,q_{2})}{\partial q_{1}} \bar q_{1}+O_{C^{r-1}}(\epsilon\bar %q_{1}^{2})&=E-V_1(q_1^w)-V_1'(q_1^w)\bar q_{1}+O_{C^{r-1}}(\bar q_{1}^{2})\\\bar %q_{1}&=\frac{-\eta }{\epsilon\frac{\partial V_{r}(q_1^w,q_{2})}{\partial %q_{1}} +V_1'(q_1^w)}+O_{C^{r-1}}(\epsilon\bar q_{1}^{2})+O_{C^{r-1}}(\bar %q_{1}^{2})
%\end{split}\end{equation}
so indeed, the point $z=(q_{1,min}^{\epsilon,\eta},p_1=0,q_2,p_2)\in\Sigma_{E}^{>}$.
\end{proof}
\subsection{Forward and backward flow from the near wall section.}

 We establish next that   \(\Sigma_{E}^* \) is crossed exactly once before returning to   \(\Sigma_{E}\) and that the crossing times are well defined and have a continuous limit as \(\epsilon\rightarrow0\):
   \begin{lem}\label{lem:sigstarcross}
For all \(E>0\) and sufficiently small $\epsilon$, the wall flow of an initial condition  \(z_{0}\in\Sigma_{E}\)   crosses the star section exactly once before returning to  $\Sigma_{E}$: for any \(z\in\Sigma_{E}\) there exists a unique star-crossing  time,  \(t^{-,\epsilon}(z_{0})\), such that \(\Phi^{\epsilon,im}_{t^{-,\epsilon}(z_0)}z_{0}\in\Sigma_{E}^* \) and a unique return time,     \(\tilde{T}_1^{\epsilon }(z_0)=t^{-,\epsilon}(z_0)+t^{+,\epsilon}(z_0)\) such that
\(\Phi^{\epsilon,im}_{\tilde{T}_1^{\epsilon}(z_0)}z_{0}\in\Sigma_E\). Moreover, \(\lim _{\epsilon\rightarrow0}t^{\pm,\epsilon}(z_0)=\frac{1}{2}\tilde{T}_1(J(I(q_{20},p_{20}),E))\).  \end{lem}

\begin{proof} Denote by $z=(q_{10},p_{10}=0,q_{20},p_{20})\in\Sigma_{E}$  where  \(q_{10}>0\). Then, since \(\dot q_{1}=p_1, \  \dot p_{1}|_{\Sigma_{E}}=-V_1'(q_{1max}(E))+O(\epsilon)<0\), \(q_1\) decreases until either \(p_1\) changes its sign (so   it crosses \(\Sigma^{>}\), possibly at \(q_1^w\)) or it reaches (possibly tangentially) the wall, crossing \(\Sigma^w\). Namely, there exists a unique first crossing time of the star section, \(t^{-,\epsilon}(z_0)\).  In the first case the trajectory continues without impact whereas in the second case it impacts the wall, so \(p_{1}(t^{-,\epsilon}(z_0))\rightarrow-p_1(t^{-,\epsilon}(z_0))\geq0\). In both cases, just after crossing \(\Sigma^*_E\), the horizontal velocity strictly increases and is/becomes positive (since either   \(\dot p_{1}|_{\Sigma^{>}_{E}}=-V_1'(q_{1min}(E))+O(\epsilon)>0\) or    \(\dot p_{1}|_{\Sigma^{w}_{E}}=-V_1'(q^{wall}_{1})+O(\epsilon)>0\)). This trend reverses once \(q_{1}\) becomes positive, leading to the return  to  \(\Sigma_{E}\) at a unique return time  \(\tilde{T}_1^{\epsilon}(z_0)=t^{-,\epsilon}(z_0)+t^{+,\epsilon}(z_0)\). The limit of the return time follows from the smooth dependence of the flow on \(\epsilon\) in between impacts and the symmetry in \(p_{1}\) of the unperturbed wall flow.\end{proof}

 The flow restricted to i.c. which cross through the by-pass section, and in particular, to the i.c. belonging to the singularity curve, is identical to the smooth perturbed flow which is \(C^{r}\) \(\epsilon-\)close to the unperturbed flow. On the other hand, the dependence of the return times on initial conditions is non-smooth near tangency for impacting initial conditions, i.e. for i.c. that cross \(\Sigma^w\). In section \ref{subsec:pert-return} we derive the expansion for the non-smooth part.
The tangency curve, by definition, belongs to the common boundary of   \(\Sigma_{E}^{>},\Sigma^w_{E}\), namely, \(\Xi_{wall}^{\epsilon}(E)\subset(\Sigma_{E}^{>}\cap\Sigma_{E}^w)\subset\Sigma_{E}^*\) and thus: \begin{lem}\label{lem:tancurvetwist}
For  \(E>V_1(q_1^w)\) and for sufficiently small $\epsilon$, the  iso-energy singularity curve of the wall system can be represented as a \(C^{r }\) smooth graph in the unperturbed action-angle coordinates: $\sigma^{\epsilon}(E)=\{(\theta,I_{tan}^{\epsilon}(\theta;E)),\theta\in[-\pi,\pi]\}$.\end{lem}

\begin{proof}
First, notice that the singularity curve  \(\sigma^{\epsilon}(E)\) is the projection to the \((q_{2},p_2)\) plane of the backwards flow of   \(\Xi_{wall}^{\epsilon}(E)\) to \(\Sigma_{E}\).  Indeed, denote by $z_{0}=(q_{10},p_{10}=0,q_{20},p_{20})\in\Xi^{\epsilon}(E)\subset\Sigma_{E}$ an initial condition with \((q_{20},p_{20})\) belonging to \(\sigma^{\epsilon}(E)\) (since \(\Sigma_{E}\) is a transverse section \(q_{10}=q_{10}(q_{20},p_{20},E)\) is uniquely defined). Then,  by lemma \ref{lem:sigstarcross}, there exists a unique crossing time of   $t^{-,\epsilon}(z_0)$ of \(\Sigma^*_E\), and, by definition of the singularity curve the crossing of  \(\Sigma^*_E\) must be tangent to the wall, so indeed   \(z(t^{-,\epsilon}(z_0))\in\Xi_{wall}^{\epsilon}(E)\). Conversely, let  \(z^{w}= (q_1^{w},p^{w}_1=0,q_2,p_2=p_2^{tan,\pm,\epsilon}(q_2;H))\in\Xi_{wall}^{\epsilon}(E)\). Then, integrating backwards the smooth flow from this point, and noticing, as in lemma  \ref{lem:sigstarcross}, that     \(\dot p_{1}|_{\Sigma^{w}_{E}}=-V_1'(q^{wall}_{1})+O(\epsilon)>0\) (since \(q_{1}^{wall}<0\)), we obtain that there exist a unique first backward crossing time of the section  \(\Sigma_{E}\). By uniqueness of solutions we conclude  that for every \(z^{w}\in\Xi_{wall}^{\epsilon}(E)\) there exists a unique  \(z_{0}(z_w)\in\Xi ^{\epsilon}(E)\) such that  \( z^{w}=\Phi^\epsilon_{ t^{-,\epsilon}(z_0)}z_{0}\)  and thus \( z_{0}=\Phi^\epsilon_{-t^{-,\epsilon}(z_0)}z^{w}\). Since by lemma \ref{lem:tangentcurve} the tangency curve \(\Xi_{wall}^{\epsilon}(E)\) is a smooth circle, and the backward crossing time  depends smoothly on   i.c. belonging to \(\Xi_{wall}^{\epsilon}(E)\), its pre-image on   \(\Sigma_{E}\) by the smooth diffeomorphism is also a smooth circle.

Using the unperturbed transformation to action angle coordinates, the backward evolution of \(\Xi_{wall}^{\epsilon}(E)\)  from  \(\Sigma_{E}^{>}\) to the section $\Sigma_{E}$ is represented, in the \((q_{2},p_2)\) plane by the map  $\{(\theta ,I_{tan,\Sigma^*_E}^{\epsilon}(\theta)),\theta\in[-\pi,\pi]\}\rightarrow\{(\theta'(\theta;E),I'(\theta;E))|\theta\in[-\pi,\pi]\}=\sigma^{\epsilon}(E)\subset\Sigma_{E}$, where:
\begin{equation}\label{eq:tanbackwardsmap}
\begin{cases}
I'(\theta;E)-I_{tan,\Sigma^*_E}^{\epsilon}(\theta)& =\epsilon f^{-}(I_{tan,\Sigma^*_E}^{\epsilon}(\theta),\theta;\epsilon) \\
\theta'(\theta;E)-\theta&=-\frac{1}{2}\cdot T_1(E,I_{tan,\Sigma^*_E}^{\epsilon}(\theta))\cdot \omega_2(I_{tan,\Sigma^*_E}^{\epsilon}(\theta))+\epsilon g^{-}(I_{tan,\Sigma^*_E}^{\epsilon}(\theta),\theta;\epsilon).
\end{cases}
\end{equation}
The right hand sides correspond  to the backward integrals of the perturbed smooth vector field \((\frac{dI}{d(-t)} =-\epsilon\{I,V_{c}\},\frac{d\theta}{d(-t)}=-\{\theta,H\}=-\omega_2(I)-\epsilon\{\theta,V_{c}\})\) along the smooth perturbed trajectory. Since tangent trajectories do not impact,  their evolution under the perturbed smooth flow is \(C^{r}\)-close to the  integrable motion, so the right hand side in (\ref{eq:tanbackwardsmap}) and the  $f^{-},g^{-}$ perturbation terms are smooth in $\theta$ and \(\epsilon \) and \(f^{-}\) may be found to leading order by a Melnikov-type calculation (see section \ref{sec:returnmap}).
Thus, the singularity curve is given parametrically by \(\sigma^{\epsilon}(E)=\{(\theta'(\theta;E),I'(\theta;E))|\theta\in[-\pi,\pi]\}\). To complete the proof we need to show that it is a graph in the  angle coordinate, \(\theta'\), namely, we need to show that    \(\frac{d\theta'}{d\theta}\neq0\). Now,
\begin{equation}\label{eq:dthetpdthet1}\begin{split}
\frac{d\theta'}{d\theta}&=1-\frac{1}{2}\frac{d}{d\theta}\left(T_1(I_{tan,\Sigma^*_E}^{\epsilon}(\theta))\cdot \omega_2(I_{tan,\Sigma_1}^{\epsilon}(\theta))\right)+\epsilon\frac{dg^{-}}{d\theta}(I_{tan,\Sigma^*_E}^{\epsilon}(\theta),\theta;\epsilon) \\
&=1-\frac{1}{2}\tau(I_{tan,\Sigma^*_E}^{\epsilon}(\theta))\frac{dI_{tan,\Sigma^*_E}^{\epsilon}(\theta)}{d\theta}+\epsilon\frac{dg^{-}}{d\theta}(I_{tan,\Sigma^*_E}^{\epsilon}(\theta),\theta;\epsilon)
\end{split}
\end{equation}
 where $\tau(I)$ is the twist (see Eq. (\ref{eq:twist})) of the return map for the smooth Hamiltonian system. By Lemma \ref{lem:tangentcurve} (and in particular, as it implies that \(\frac{dI_{tan,\Sigma^*_E}^{\epsilon}(\theta)}{d\theta}=O_{C^{r-1}}(\epsilon)\) ), and since the unperturbed twist, \(\tau(I)\), is bounded, for sufficiently small \(\epsilon\), \(\frac{d\theta'}{d\theta}=1+O_{C^{r-1}}(\epsilon)\),  so the parametric curve \((\theta'(\theta;E),I'(\theta;E))  \) can be represented as the graph \((\theta',I_{tan}^{\epsilon}(\theta';E))\) where:\begin{equation}\begin{split}
I_{tan}^{\epsilon}(\theta';E)&=I_{tan,\Sigma^*_E}^{\epsilon}(\theta(\theta',E))+\epsilon f^{-}(I_{tan}(E),\theta(\theta',E);\epsilon)+O(\epsilon^{2}),\\ \theta&(\theta',E)=\theta'-\frac{1}{2}\cdot T_1(E,I_{tan}(E))\cdot \omega_2(I_{tan}(E)).\end{split}
\end{equation}
%% see 23.2 version for computations of dI/dtheta
\end{proof}

Recall that by lemma \ref{lem:divideimpwall}, on $\Sigma^*_E$, the $(q_2,p_2)$ values inside \(\sigma_{wall}^{\epsilon}(E)\) represent impacting trajectories which reach the wall and are on $\Sigma_{E}^{w}$, whereas $(q_2,p_2)$ values outside the iso-energy tangency curve represent non-impacting trajectories crossing $\Sigma_{E}^{>}$ at $q_1$ value which is greater than $q_1^w$. The tangency is the singular curve dividing between the two different types of trajectories.
Using the unperturbed action-angle coordinates to parametrize $\Sigma^*_E$,  by (\ref{eq:insidecurveq2p2}), for impacting trajectories
\begin{equation}
H_2(I)+\epsilon V_c(q_1^w,q_2(I,\theta))\leq H-V_1(q_1^w)
\end{equation}
Hence on $\Sigma^*_E$ non-impacting trajectories are given by $(I,\theta)$ such that $I>I_{tan,\Sigma^*_E}^{\epsilon}(\theta)$ and impacting trajectories by $I<I_{tan,\Sigma^*_E}^{\epsilon}(\theta)$.
Since the backward evolution of the tangency curve \(\sigma_{wall}^{\epsilon}(E)\) under the backwards perturbed flow from $\Sigma_1$ to $\Sigma$ is the singularity curve \(\sigma^{\epsilon}(E)\) we conclude:

\begin{lem}\label{lem:tancurveseparates}
The singularity curve separates between impacting and non-impacting initial conditions on $\Sigma_{E}$.
\end{lem}

\begin{proof}
Consider separately the backward evolution of iso-energy trajectories from $\Sigma_{E}^{w}$ and from $\Sigma_{E}^{>}$ to $\Sigma_{E}$. The backward Hamiltonian flow is orientation preserving, hence the backward image of trajectories in $\Sigma_{E}^w$ (respectively $\Sigma_{E}^{>}$) are mapped to the interior (respectively exterior) of the pre-image of their boundary, i.e. the tangency curve.
\end{proof}

% twist calcs
  Let $\rho(z)=I_{tan}^{\epsilon}(\theta)-I$ denote the action-signed-distance of an i.c.   \(z^{\theta,\rho,\epsilon}\in\Sigma_{E}\) from the singularity curve at the angle \(\theta\). By Lemma \ref{lem:tancurvetwist}, since \(I_{tan}^{\epsilon}(\theta)\) is smooth, \(\rho(z)=\rho(\theta,E)\) is well defined and is \(C^{r}\). Hereafter, we parameterize \(z^{\theta,\rho,\epsilon}\in \Sigma_E\) by \((\theta,\rho)\):\begin{equation}\label{eq:notationzerho}
z^{\theta,\rho,\epsilon}:=(q_{1}=q_{1}^{\epsilon}(I_{tan}^{\epsilon}(\theta)-\theta,\rho,E),p_{1}=0,I=I_{tan}^{\epsilon}(\theta)-\theta,\rho),
\end{equation}
so \(H(z^{\theta,\rho,\epsilon})=E\) and \(q_{1}^{\epsilon}(I,\theta,E)=q_{1,max}(E-H_2(I))+O(\epsilon)\) is found by the IFT from Eq. (\ref{eq:hgeneral}). By definition, \( \Pi z^{\theta,0,\epsilon}=(I_{tan}^{\epsilon}(\theta),\theta)\in\sigma^{\epsilon}(E)\) and the unperturbed tangent trajectory is  \(z_{tan}(t,\theta;E)=\Phi_{t}^0 z^{\theta,0,0}\) . By Lemma \ref{lem:tancurveseparates} , for all \(\epsilon\in[0,\epsilon_{1}]\), \(\rho>0\) corresponds to impacting i.c. of \(\mathcal{F}_\epsilon\)  and \(\rho\leqslant0\) correspond to non-impacting i.c. of this map. Notice that  $z^{\theta,\rho,0}=(q_{1,max}(E-H_2(I)),0,I_{tan}(E)-\theta,\rho)$, so, by lemma  \ref{lem:tancurvetwist}, \(z^{\theta,\rho,\epsilon}=z^{\theta,\rho,0}+\mathcal{O}_{C^{r}}(\epsilon)\).

\subsection{The image  of the singularity curve}
\begin{prop}\label{prop:perttan}
For sufficiently small $\epsilon$, the return map to $\Sigma_{E}$ of an initial condition on the  singularity curve,  $(\theta ,I^{\epsilon}_{tan}(\theta))$, is of the form:\begin{equation}\label{eq:tangimageform}\begin{split}
\bar{I}_{tan}^\epsilon(\bar{\theta}_{tan})&=I_{tan}^{\epsilon}(\theta)+\epsilon f(\theta;\epsilon)\equiv\mathfrak{f}_\epsilon(\theta)\\
\bar{\theta}_{tan}&=\theta+\omega_2(I_{tan}(E))\cdot T_1(I_{tan}(E);H)+\epsilon g(I_{tan}^{\epsilon}(\theta),\theta;\epsilon)
\end{split}\end{equation}
where $\bar{I}_{tan}^\epsilon(\cdot)$ is the forward image of the tangency curve on $\Sigma_{E}$, $\bar{\theta}_{tan}$ is the forward image of $\theta$,  \(f,g \) are \(C^{r}\) smooth in \(\theta,\epsilon\) and \begin{equation}\label{eq:deffoftheta}
f(\theta;\epsilon)=\int_0^{{T}_1^{\epsilon}(z^{\theta,0,\epsilon})}\{I,V_c\}\mid_{\Phi_{t}^\epsilon z^{\theta,0,\epsilon}}dt.
\end{equation}

\end{prop}

\begin{proof}
Let $z^{\theta,0,\epsilon}=(q_{10},p_{10}=0,q_{20},p_{20})\in\Xi^{\epsilon}(E)\subset\Sigma_{E}$, where $ S_{2}(q_{20},p_{20})=(\theta ,I_{tan}^\epsilon(\theta))$. As a tangent trajectory is not affected by the impact, integration along the smooth perturbed vector field yields:
\begin{equation}\label{eq:tangentitheta}
\begin{cases}
\bar{I}_{tan}=I+\int_0^{{T}_1^{\epsilon}(z^{\theta,0,\epsilon})}\{I,H_{int}+\epsilon V_c\}\mid_{\Phi_{t}^\epsilon z^{\theta,0,\epsilon}}dt=I_{tan}^{\epsilon}(\theta)+\epsilon f(\theta;\epsilon) \\
\bar \theta_{tan}=\theta+\int_0^{{T}_1^{\epsilon}(z^{\theta,0,\epsilon})}\{\theta,H_{int}+\epsilon V_c\}\mid_{\Phi_{t}^\epsilon z^{\theta,0,\epsilon}}dt=\theta+\omega_2(I_{tan}(E))\cdot T_1^{tan}+\epsilon g(\theta ,I_{tan}^{\epsilon}(\theta);\epsilon)
\end{cases}
\end{equation}
where $f,g$ are $C^r$ smooth, $2\pi$-periodic functions of $\theta$ (see section \ref{sec:returnmap} for the leading order calculation of \(f\)). ${T}_1^{\epsilon}(z^{\theta,0,\epsilon})$ denotes the return time to $\Sigma_{E}$ of the perturbed tangent trajectory  and $T_1^{tan}={T}_1^{0}(z^{\theta,0,0})=T_1(I_{tan}(E);E)$ is the return time of the unperturbed tangent trajectory.
By definition, $\bar{I}_{tan}=\mathfrak{f}_\epsilon(\theta)\equiv I_{tan}^{\epsilon}(\theta)+\epsilon f(I_{tan}^{\epsilon}(\theta),\theta;\epsilon)$, lies on the forward image of the tangency curve at the angle $\bar{\theta}_{tan}$, and by Eq. (\ref{eq:tangentitheta}) \(\frac{d\bar \theta_{tan}}{d\theta}=1+O_{C^{r}}(\epsilon  )\), hence $\bar{I}_{tan}=\mathfrak{f}_\epsilon(\theta)=\bar{I}_{tan}^{\epsilon}(\bar{\theta}_{tan})$.
\end{proof}

\subsection{A time-reversal symmetry of the singularity curve}
The singularity line and its image are related by a symmetry (see  \cite{lamba1995chaotic} for a similar symmetry observed  in  a 1 d.o.f harmonic oscillator impacting with a harmonic oscillating wall):

\begin{lem}\label{lem:symmetry}
The image of  the singularity curve $\mathcal{\sigma}^\epsilon$ under $\mathcal{F}_{\epsilon}$ is identical to its reflection about the $q_2$ axis, namely  $\bar{\sigma}^{\epsilon}:=\mathcal{F}_{\epsilon}\mathcal{\sigma}^\epsilon
=\{(\theta,I_{tan}^\epsilon(-\theta)),\theta\in[-\pi,\pi]\}$.
\end{lem}

\begin{proof}
This is a result of a pointwise symmetry on the singularity curve which follows from the mechanical form of the Hamiltonian. Let  $z_{0}=(q_{10},p_{10}=0,q_{20},p_{20})\in\Xi^{\epsilon}(E)\subset\Sigma_{E},  $  let \(\mathcal{R}_{2}z_{i}=(q_{1i},p_{1i}=0,q_{2i},-p_{2i})\) denote a reflection about the \(q_{2}\) axis and let  $S_2(q_{2i},p_{2i})=S_{2}\Pi z_{i}=(I_i,\theta_i)\). Since we consider i.c. on the singularity curve, by definition, if $(I_0,\theta_0)=S_2(q_{20},p_{20})$ then $ I_0=I_{tan}^{\epsilon}(\theta_0)$.
Let   \( z^w=(q_1^w,0,q_2^w,p_2^w)=\Phi^\epsilon_{ t^{-,\epsilon}(z_0)}z_{0}\)     and \(z_{1}=(q_{11},p_{11}=0,q_{21},p_{21})=\Phi^\epsilon_{ t^{+,\epsilon}(z_0)}z^{w}=\Phi^\epsilon_{ t^{\epsilon}(z_0)}z_{0} \) denote the image of \(z_{0}\) at the star section   and under the return map respectively, so \((q_{21},p_{21})=\mathcal{F}_{\epsilon}(I_{tan}^{\epsilon}(\theta_0),\theta_0)\).
Then, from time reversal symmetry:
\begin{equation}\begin{split}
\Phi^{\epsilon}_{t^{+,\epsilon}(z_0)}(q_1^w,0,q_2^w,p_2^w)&=(q_{11},p_{11}=0,q_{21},p_{21}) \\
\Rightarrow \Phi^{\epsilon}_{-t^{+,\epsilon}(z_0)}(q_1^w,0,q_2^w,-p_2^w)&=(q_{11},-p_{11}=0,q_{21},-p_{21})=\mathcal{R}_{2}z_1
\\
\Rightarrow \Phi^{\epsilon}_{t^{+,\epsilon}(z_0)}\mathcal{R}_{2}z_1&=(q_1^w,0,q_2^w,-p_2^w)\\
\end{split}
\end{equation}
Thus, \( \mathcal{R}_{2}z_1\) also belongs to the singularity curve and \(t^{-,\epsilon}(\mathcal{R}_{2}z_1)=t^{+,\epsilon}(z_0)\). Similarly,
\begin{equation}\begin{split}
\Phi^{\epsilon}_{-t^{-,\epsilon}(z_0)}(q_1^w,0,q_2^w,p_2^w)&=(q_{10},p_{10}=0,q_{20},p_{20}) \\
\Rightarrow \Phi^{\epsilon}_{t^{-,\epsilon}(z_0)}(q_1^w,0,q_2^w,-p_2^w)&=(q_{10},-p_{10}=0,q_{20},-p_{20})=\mathcal{R}_{2}z_0
\end{split}
\end{equation}
Thus\begin{equation}
\Phi^{\epsilon}_{{t^{-,\epsilon}(z_0)+t^{+,\epsilon}(z_0)}}\mathcal{R}_{2}z_1=\mathcal{R}_{2}z_0
\end{equation}
\(\)
We  deduce that \( \Pi\mathcal{ R}_{2}z_1\in\sigma^\epsilon(E)\) and that \(\mathcal{F}_{\epsilon} \Pi\mathcal{R}_{2}z_1=\Pi\mathcal{R}_{2}z_0\in \bar \sigma^\epsilon(E)\). Since by the definition of action-angle coordinates for mechanical Hamiltonian (see appendix) $S_{2}(\Pi\mathcal{R}_{2}z_i)=S_2(q_{2i},-p_{2i})=(I_i,-\theta_i)$  we conclude that \(\Pi\mathcal{ R}_{2}z_1\in\sigma^\epsilon(E)\) implies that $ I_{1}=\bar I_{tan}^\epsilon(\theta_1)=I_{tan}^\epsilon(-\theta_1)$, and,  equivalently,
 \(\Pi\mathcal{R}_{2}z_0\in \bar \sigma^\epsilon(E)\) implies that \(I_{tan}^{\epsilon}(\theta_0)=\bar I_{tan}^{\epsilon}(-\theta_0)\), as claimed, see Figure  \ref{fig:symmetry}.

\end{proof}

\begin{figure}[ht]
\begin{centering}
\includegraphics[scale=0.2]{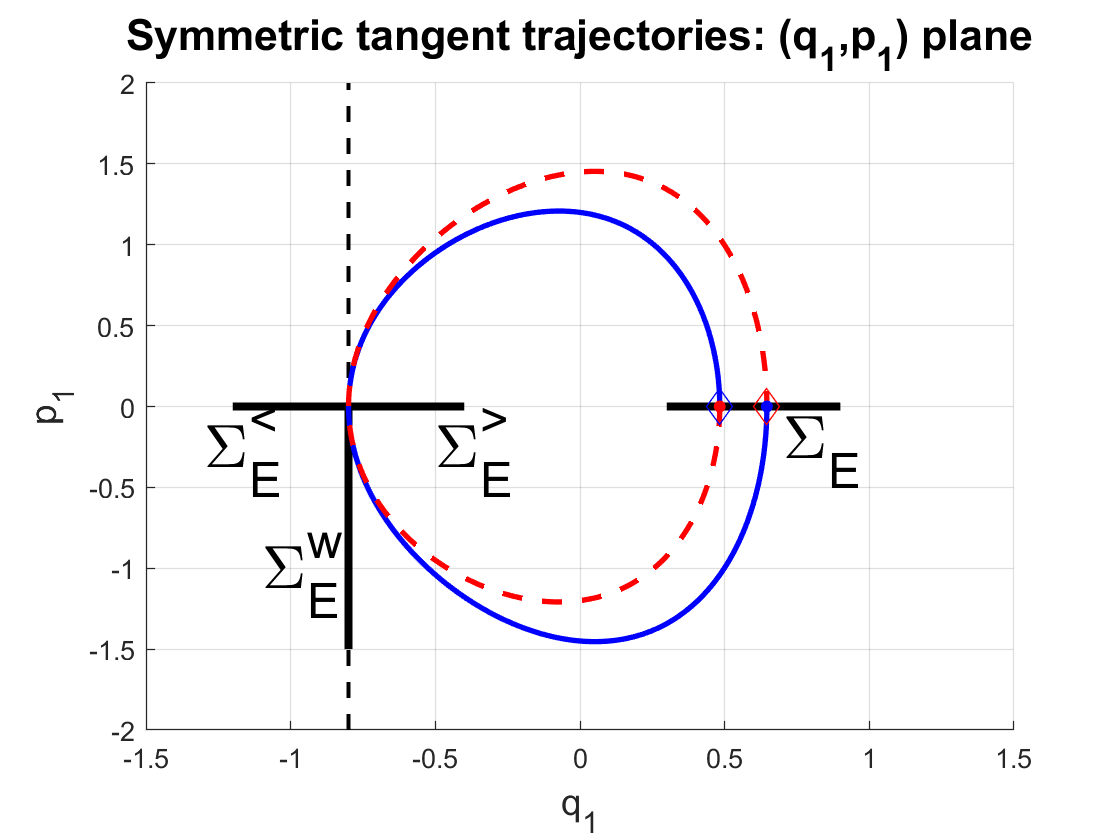}
\includegraphics[scale=0.2]{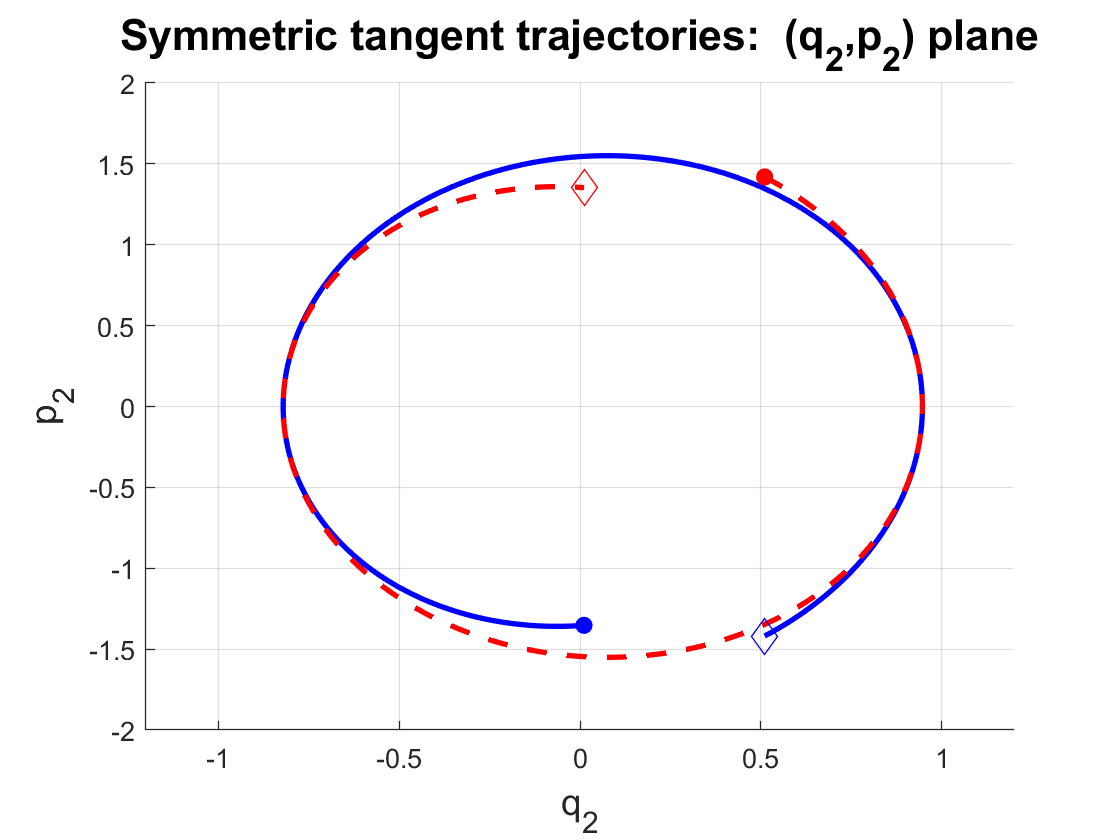}
\par\end{centering}
\protect\caption{\label{fig:symmetry}Two symmetric tangent trajectories projected to the $(q_1,p_1)$ (left) and $(q_2,p_2)$ (right) planes. In the  $(q_1,p_1)$ plane the starting point (solid circles) of a trajectory on $\Sigma_{E}$ is the ending point (diamond) of the other, whereas in the $(q_2,p_2)$ plane the starting and ending points are related by the time-reversal  symmetry in $\theta$. Here  $q_1^w=-0.8$,  $V_1(q_1)=2q_1^2+q_1^3+\frac{1}{4}q_1^4$, $V_2(q_2)=\frac{3}{2}q_2^2, H=2$, $V_c(q_1,q_2)=q_1\cdot q_2$ and $\epsilon=0.3$.}
\end{figure}
Combining this symmetry and lemma \ref{lem:tancurvetwist} we conclude:

\begin{cor}\label{cor:forwardtan}
For sufficiently small $\epsilon$, the forward image of the singularity curve can be represented as a graph in the unperturbed action-angle coordinates: $\bar \sigma^{\epsilon}(E)=\{(\theta,\bar{I}_{tan}^{\epsilon}(\theta)),\theta\in[-\pi,\pi]\}$.
\end{cor}

\subsection{Proof of Theorem \ref{prop:singularcurve} }
 Combining the above lemmas we establish Theorem \ref{prop:singularcurve}: by
lemma \ref{lem:tancurveseparates}, for   \(E>V_1(q_1^w)\) and sufficiently small $\epsilon$, the singularity curve  \(\sigma^{\epsilon}(E)\) is an isolated  circle, dividing  the cylinder  $\Sigma_{E} $ to impacting (below the curve) and non-impacting (above the curve) regions. By lemma \ref{lem:tancurvetwist} for sufficiently small $\epsilon$, this circle is a graph:   $\sigma^{\epsilon}(E)=\{(\theta,I_{tan}^{\epsilon}(\theta)),\theta\in[-\pi,\pi]\}$,  by corollary \ref{cor:forwardtan} its image under the map is also  a graph,  $\bar \sigma^{\epsilon}(E)=\mathcal{F}_\epsilon \sigma^{\epsilon}=\{(\theta,\bar{I}_{tan}^{\epsilon}(\theta)=I_{tan}^{\epsilon}(2\pi-\theta)),\theta\in[-\pi,\pi]\}$, and by  lemma \ref{lem:tangentcurve}, \ref{lem:tancurvetwist} and corollary \ref{cor:forwardtan} both   \(I_{tan}^{\epsilon}(\theta)\) and \(\bar{I}_{tan}^{\epsilon}(\theta)\) are \(C^{r}\) \(\epsilon-\)close to the constant function \( I_{tan}(E)\) and   \(I_{tan}^{\epsilon}(\theta)=\bar{I}_{tan}^{\epsilon}(-\theta)\).

\section{The near tangent return map}\label{sec:returnmap}
Near tangency, the unperturbed map $\mathcal{F}_0$  is symplectic and  $C^{0}$  but is not smooth due to the non-smooth dependence of the rotation number on \(I\), see Eqs. (\ref{eq:int_impact_returnmap}-\ref{eq:deltattravel}). This singularity leads to non-trivial behavior under perturbation. We first construct the perturbed near tangent return map  by separating the approximations of the perturbed near tangent trajectories to impacting and non-impacting, according to their position with respect to the singularity curve on $\Sigma_{E}$.
We then construct a change of coordinates bringing the near tangent map  to the form (\ref{eq:pert-return}).

\subsection{The perturbed near tangent return map}\label{subsec:pert-return}

\begin{prop}\label{lem:pertreturnmap}
For sufficiently small $\epsilon$,
the near-tangent local return map to the cross-section $\Sigma_{E}$, $\mathcal{F}_\epsilon^{loc}:(\theta,I)\rightarrow(\bar \theta,\bar I)$ is  of the form :
\begin{subequations}\label{eq:neartanItheta}
\begin{align}
I\geq I^\epsilon_{tan}(\theta): &\begin{cases}
\bar{I}=I+\bar{I}^{\epsilon}_{tan}(\bar{\theta})-I^\epsilon_{tan}(\theta)+\mathcal{O}_{C^{r-1}}(\epsilon(I-I^\epsilon_{tan}(\theta)))\\
\begin{aligned}\bar{\theta}=\theta+\omega_2(I_{tan}(E))\cdot T_1^{tan}+&\tau(I_{tan}(E))\cdot(I-I^\epsilon_{tan}(\theta))+\mathcal{O}_{C^{r-1}}(\epsilon,(I-I^\epsilon_{tan}(\theta))^2)\end{aligned}
\end{cases}\\
I< I^\epsilon_{tan}(\theta): &\begin{cases}
\bar{I}=I+\bar{I}^{\epsilon}_{tan}(\bar{\theta})-I^\epsilon_{tan}(\theta)+\mathcal{O}_{C^{r-2}}(\epsilon\sqrt{I^\epsilon_{tan}(\theta)-I}\ ) \\
\begin{aligned}\bar{\theta}=\theta+\omega_2&(I_{tan}(E))\cdot T_1^{tan}(E)+ G_{s}(I-I^\epsilon_{tan}(\theta))+\mathcal{O}_{C^{r-2}}(\epsilon,(I-I^\epsilon_{tan}(\theta))^{2},\epsilon\sqrt{I^\epsilon_{tan}(\theta)-I})\end{aligned}
\end{cases}
\end{align}
\end{subequations}
where $\bar{I}^{\epsilon}_{tan}(\bar{\theta})$ is the image of $I_{tan}^{\epsilon}(\theta)$ under the map parametrized by $\bar{\theta}$ (see Proposition \ref{prop:perttan}),  \(\bar{I}^{\epsilon}_{tan}(\bar{\theta})-I^\epsilon_{tan}(\theta)=\epsilon\int_0^{T^{tan}_1(E)}\{I,V_c\}\mid_{ z_{tan}(t,\theta;E)} dt+\mathcal{O}_{C^{r}}(\epsilon^{2})\) and \(G_{s}(\cdot)\) is the function  of Eq. (\ref{eq:gsindef}) (continuous with a square root singularity at the origin for negative arguments).
\end{prop}

\begin{figure}[ht]
\begin{centering}
\includegraphics[scale=0.2]{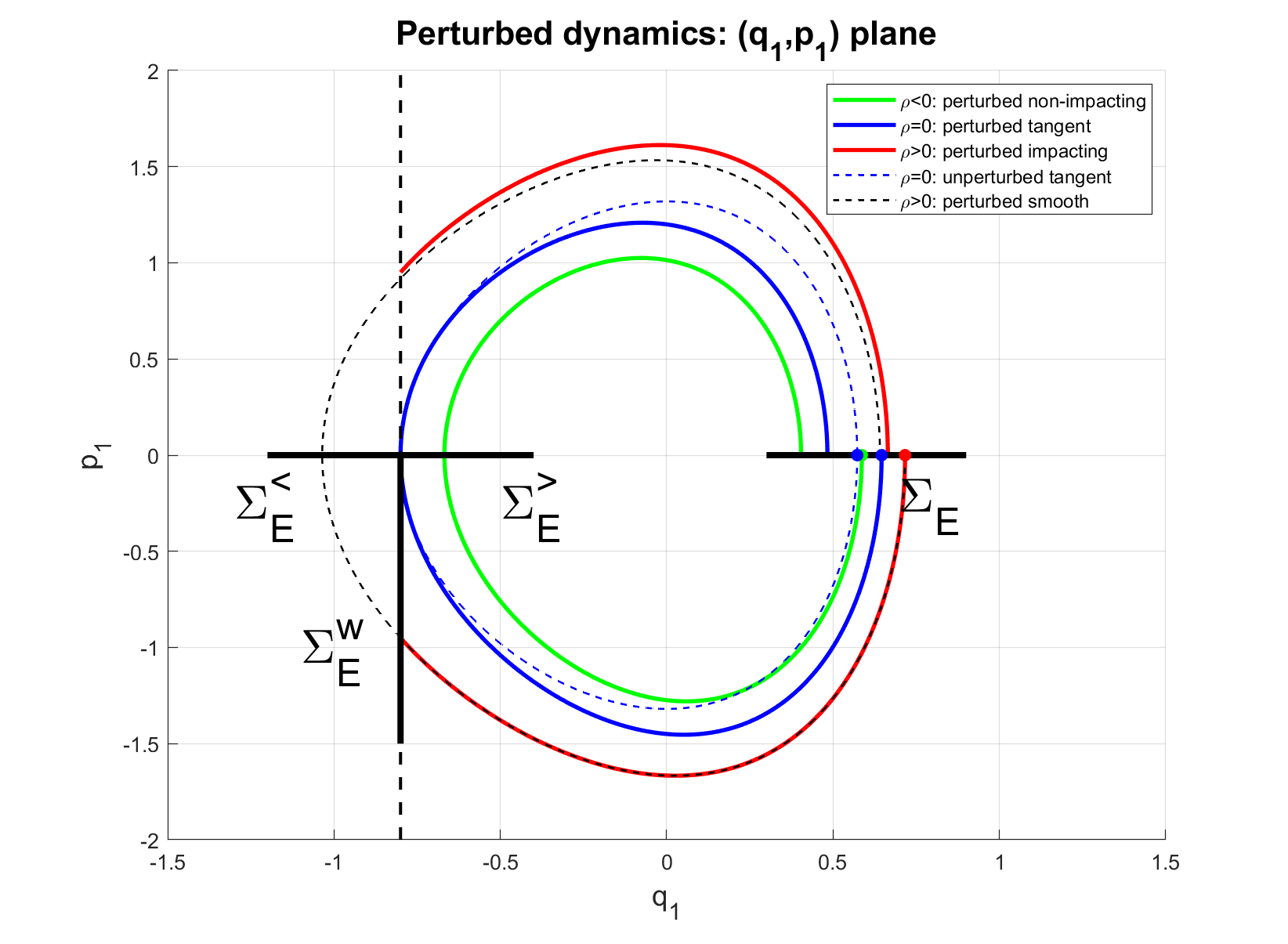}
\includegraphics[scale=0.2]{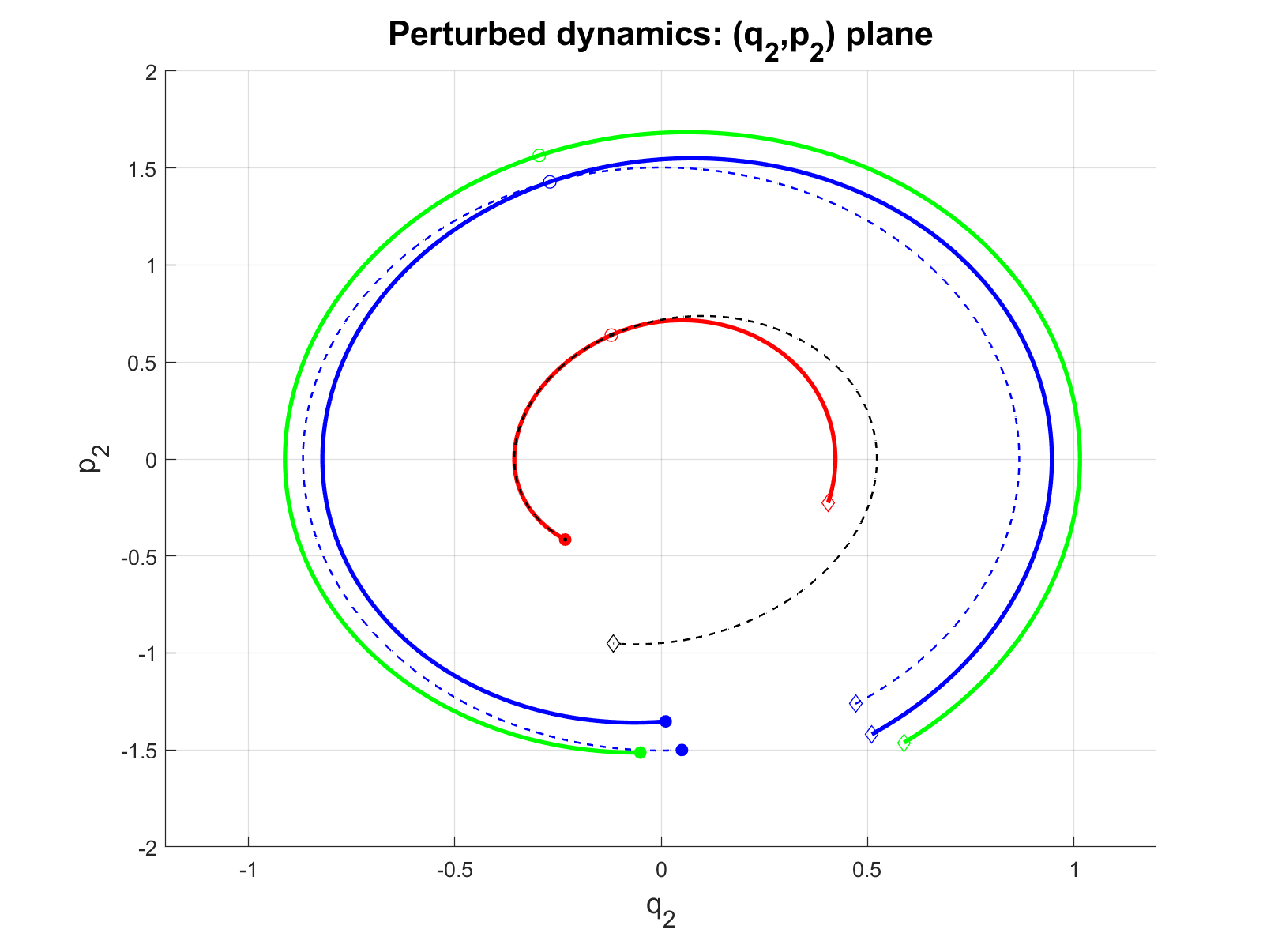}
\par\end{centering}
\protect\caption{\label{fig:perttangency}Iso-energy impacting, tangent and non-impacting  trajectories segments in-between crossings of the return section \(\Sigma_{E}\). (a) Projection to the $(q_1,p_1)$ plane. (b) Projection  to the  $(q_2,p_2)$ plane.    Here, all segments meet \(\Sigma_{E}^*\) with the same angle \(\theta^{w}\) (open circles in (b)), yet,  each initial condition (filled circles), \(z^{\theta,\rho,\epsilon}\), has a different phase \(\theta\)  and a corresponding distinct action distance \(\rho\). Parameters are as in Figure \ref{fig:symmetry}. }
\end{figure}

\begin{proof} Let \(\epsilon_{1}>0\) be sufficiently small so that Theorem  \ref{prop:singularcurve} holds, and consider \(\epsilon\in[0,\epsilon_{1}]\).  Then, i.c. on  \(\Sigma_{E}\) can be parameterized by \(\theta,\rho\) (recall Eq. (\ref{eq:notationzerho})) and:  \begin{equation}\label{eq:impactflow}
\mathcal{F}_\epsilon^{loc}\Pi z^{\theta,\rho,\epsilon}=\Pi\Phi^{\epsilon,im}_{t^\epsilon(z)}z, \quad\Phi^{\epsilon,im}_{t^\epsilon(z)}z =\begin{cases}\Phi^{\epsilon}_{t^\epsilon(z)}z & \text{no impact, }\rho\leqslant 0 \\
\Phi^{\epsilon}_{t^{+,\epsilon}(z)}\circ \mathcal{R}_{1}\circ\Phi^{\epsilon}_{t^{-,\epsilon}(z)}z & \text{with impact, }\rho> 0. \\
\end{cases}
\end{equation}where, for shorthand notation, we omit the labeling of \(z\) when it is inessential. By Theorem \ref{prop:singularcurve}, the division to these two cases occurs exactly at the singularity curve \(\sigma^{\epsilon}(E)\), and by lemma \ref{lem:sigstarcross}, \(t^{-,\epsilon}(z)\),  the crossing time of \(\Sigma_{E}^*\),   and  \(t^{+,\epsilon}(z)\), the travel time from \(\Sigma_{E}^*\) to \(\Sigma_{E}\) are uniquely defined by \(z\) and limit to \(\frac{1}{2}T_1(I_{tan}(E))\) as \(\epsilon\rightarrow0\), thus they are finite. Next, we calculate the integrals of \(\{I,H\}\) and \(\{\theta,H\}\) along the perturbed trajectory segments on the finite time intervals \([0,t^{-,\epsilon}(z)) \) and \((t^{-,\epsilon}(z),\tilde{T}_{1}^{\epsilon}(z)=t^{-,\epsilon}(z)+t^{+,\epsilon}(z)]\), approximate these integrals along some properly defined unperturbed approximations (see below and Figure \ref{fig:unperttangency}), and express the results in terms of the image of the singularity curve.

Denote by   \(z_{sm}^{\theta,\rho,\epsilon}(t)\) the trajectory of  the smooth Hamiltonian system ($b=0$) emanating from \(z^{\theta,\rho,\epsilon}\) (so   $z^{\theta,\rho,\epsilon}_{sm}(0)=z^{\theta,\rho,\epsilon},z_{sm}^{\theta,\rho,\epsilon}(t)=\Phi_{t}^\epsilon z^{\theta,\rho,\epsilon} $).
This solution depends smoothly on i.c. and on \(\epsilon\),  so, for small \(\epsilon,\rho\) and for all $0\leq t \leq T_1^{tan}(E)+\eta, \eta\in[0,\frac{1}{4}T_1^{tan}]$
:

\begin{equation}\label{eq:zsmootht}
z_{sm}^{\theta,\rho,\epsilon}(t)=z_{sm}^{\theta,\rho,0}(t)+\mathcal{O}_{C^r}(\epsilon)=z_{sm}^{\theta,0,0}(t)+\mathcal{O}_{C^{r}}(\epsilon,\rho)=z_{tan}(t,\theta;E)+\mathcal{O}_{C^{r}}(\epsilon,\rho).
\end{equation}
Let \({T}_1^{\theta,\rho,\epsilon}={T}_1^\epsilon(z^{\theta,\rho,\epsilon})\) denote the return time of \(z_{sm}^{\theta,\rho,\epsilon}(t)\) to the cross-section $\Sigma_E$,
\(t^{h,\epsilon}(z^{\theta,\rho,\epsilon})\)  the travel time of  \(z_{sm}^{\theta,\rho,\epsilon}(t)\) from \(\Sigma_{E}\) to the mid section    $\Sigma_{E}^{<}\cup\Sigma_{E}^{>}$ and by   \(t^{\star}(z^{\theta,\rho,\epsilon})\)  its travel time from \(z^{w}(z^{\theta,\rho,\epsilon})=z_{sm}^{\theta,\rho,\epsilon}(t^{-,\epsilon}(z^{\theta,\rho,\epsilon}))\in\Sigma^{w}_{E}\)  to  $\Sigma_{E}^<$, so \(t^{\star}=t^{h,\epsilon}(z^{\theta,\rho,\epsilon})-t^{-,\epsilon}(z^{\theta,\rho,\epsilon})\) and   \(t^{\star}(z^{\theta,\rho,0})=\frac{1}{2}\Delta t_{travel}(J(I_{tan}(E)-\rho ,E))\).
By smooth dependence on i.c. and parameters (Eq. (\ref{eq:zsmootht})):
\begin{equation}
T_1^{\theta,\rho,\epsilon}=T_1^{\theta,\rho,0}+\mathcal{O}_{C^{r}}(\epsilon)=T_1^{tan}(E)-\rho\cdot\frac{d}{dI}T_{1}(J(I))|_{I_{tan}(E)}+\mathcal{O}_{C^{r-1}}(\epsilon,\rho^{2})
\end{equation}
and
\begin{equation}\label{eq:halftexp}
t^{h,\epsilon}(z^{\theta,\rho,\epsilon})=\frac{1}{2}T_1^{\theta,\rho,0}+\mathcal{O}_{C^{r}}(\epsilon)=\frac{1}{2}T_1^{tan}(E)-\frac{1}{2}\rho\cdot\frac{d}{dI}T_{1}(J(I))|_{I_{tan}(E)}+\mathcal{O}_{C^{r-1}}(\epsilon,\rho^{2}).
\end{equation}
 In particular, $T_1^{\theta,0,\epsilon}=T_1^\epsilon(I^{\epsilon}_{tan}(\theta),\theta;E)$ is the return time to $\Sigma_{E}$ of the perturbed tangent trajectory $z_{sm}^{\theta,0,\epsilon}(t)$, and $T_1^{\theta,0,0}=T_1^{tan}(E)=T_1(J(I_{tan}(E),E))$, the return time of unperturbed tangent trajectory $z_{sm}^{\theta,0,0}(t)$ as well $T_1^{\theta,\rho,0}=T_1^{\rho}=T_1(J(I_{tan}(E)-\rho,E))$, the return time of  unperturbed smooth trajectory with actions   \((J,I)=(J(I_{tan}(E)-\rho,E),I_{tan}(E)-\rho)\) are independent of \(\theta\).

\paragraph*{The return map for non-impacting i.c., Eq.  \ref{eq:neartanItheta}a:} For non-positive \(\rho\) (green and blue trajectory segments in Figures \ref{fig:unperttangency} and \ref{fig:perttangency}), the trajectory segment between returns does not impact, so   \(z_{im}^{\theta,\rho,\epsilon}(t)|_{\rho\leqslant0}=z_{sm}^{\theta,\rho,\epsilon}(t)\),   and  smooth standard perturbation theory applies: \begin{equation}\label{eq:ibarcalc1}
\begin{split}
\bar{I}=I(\mathcal{F}_{\epsilon} \Pi z^{\theta,\rho,\epsilon})&=I(\Pi z^{\theta,\rho,\epsilon})+\int_0^{{T}_1^{\theta,\rho,\epsilon}}\{I,H_{int}+\epsilon V_c\}\mid_{z_{sm}^{\theta,\rho,\epsilon}(t)} dt \\
&=I_{tan}^{\epsilon}(\theta)-\rho+\epsilon\int_0^{T_1^{\theta,0,\epsilon}+\mathcal{O}_{C^{r}}(\rho)}\{I,V_c\}\mid_{z_{sm}^{\theta,0,\epsilon}(t)+\mathcal{O}_{C^{r}}(\rho)} dt\\
&=\mathfrak{f}_\epsilon(\theta)-\rho+\mathcal{O}_{C^{r}}(\epsilon\rho) \\
&=I_{tan}^{\epsilon}(\theta)-\rho+\epsilon\int_0^{T_1^{\theta,0,0}+\mathcal{O}_{C^{r}}(\epsilon)}\{I,V_c\}\mid_{z_{sm}^{\theta,0,0}(t)+\mathcal{O}_{C^{r}}(\rho,\epsilon)} dt+\mathcal{O}_{C^{r}}(\epsilon\rho) \\
&=I_{tan}^{\epsilon}(\theta)-\rho+\epsilon\int_0^{T_1^{tan}(E)}\{I,V_c\}\mid_{z_{tan}(t,\theta;E)(t)} dt+\mathcal{O}_{C^{r}}(\epsilon\rho,\epsilon^{2})
\end{split}
\end{equation}
where, as in Proposition \ref{prop:perttan} (i.e. Eq. (\ref{eq:tangentitheta})),
\begin{equation}
\begin{split}
\mathfrak{f}_\epsilon(\theta):=I(\mathcal{F}_{\epsilon} \Pi z^{\theta,0,\epsilon})=\bar I_{tan}^{\epsilon}(\bar \theta_{tan})
&=I_{tan}^{\epsilon}(\theta)+\epsilon\int_0^{{T}_1^{\theta,0,\epsilon}}\{I,V_c\}\mid_{z_{sm}^{\theta,0,\epsilon}(t)} dt
\\&=I_{tan}^{\epsilon}(\theta)+\epsilon f(I_{tan}^{\epsilon}(\theta),\theta;\epsilon) \\
&=I_{tan}^{\epsilon}(\theta)+\epsilon\int_0^{T^{tan}_1(E)}\{I,V_c\}\mid_{z_{sm}^{\theta,0,0}(t)} dt+\mathcal{O}_{C^{r}}(\epsilon^{2})
\end{split}
\end{equation}
is the action of the image of   \(z^{\theta,0,\epsilon}\in\sigma^\epsilon\), namely, \(\mathcal{F}_{\epsilon} \Pi z^{\theta,0,\epsilon}=(\mathfrak{f}_\epsilon(\theta)=\bar{I}_{tan}^{\epsilon}(\bar{\theta}_{tan}),\bar{\theta}_{tan})\in\bar \sigma^\epsilon\). In particular, the leading order term of the action of the singularity curve image is:\begin{equation}
 f(I_{tan}^{\epsilon}(\theta),\theta;0)=\int_0^{T^{tan}_1(E)}\{I,V_c\}\mid_{z_{tan}(t,\theta;E)} dt.
\end{equation}
The leading order dependence of \(\bar \theta \) on \(\rho\) is similarly  found:\begin{equation}\label{eq:thebarnoimpact}\begin{split}
\bar{\theta}=\theta(\mathcal{F}_{\epsilon} \Pi z^{\theta,\rho,\epsilon})&=\theta+\int_0^{T_1^{\theta,\rho,\epsilon}}\{\theta,H_{int}+\epsilon V_c\}\mid_{z_{sm}^{\theta,\rho,\epsilon}(t)} dt\\
&=\theta+\int_0^{T_1^{\theta,\rho,0}+\mathcal{O}_{C^{r}}(\epsilon)}(\{\theta,H_{int}\}+\epsilon\{\theta, V_{c}\})\mid_{z_{sm}^{\theta,\rho,0}(t)+\mathcal{O}_{C^{r}}(\epsilon)}dt \\
&=\theta+\omega_2(I_{tan}(E)-\rho)\cdot T_1^{\rho}+\mathcal{O}_{C^{r}}(\epsilon) \\&=\theta+\omega_2(I_{tan}(E))\cdot T_1^{\rho}-\rho\cdot\omega_2'(I_{tan}(E))\cdot T_1^{\rho}+\mathcal{O}_{C^{r-1}}(\epsilon,\rho^2) \\&=\theta+\omega_2(I_{tan}(E))\cdot T_1^{tan}(E)-\rho\cdot\tau(I_{tan}(E))+\mathcal{O}_{C^{r-1}}(\epsilon,\rho^2)
\end{split}
\end{equation}
where $\tau( I)$ is the twist of the iso-energy return map of the smooth Hamiltonian system (see Eq. (\ref{eq:twist})), so the second line of  Eq.  \ref{eq:neartanItheta}a is established.

Finally, we show that Eq. (\ref{eq:ibarcalc1}) implies the first line of Eq.  \ref{eq:neartanItheta}a. By Proposition \ref{prop:perttan}, on the singularity curve:
\begin{equation}\begin{split}
\bar{\theta}_{tan}=\bar{\theta}( z^{\theta,0,\epsilon})=\theta(\mathcal{F}_{\epsilon} \Pi z^{\theta,0,\epsilon})&=\theta+\int_0^{T_1^{\theta,0,\epsilon}}\omega_2(I(z_{sm}^{\theta,0,\epsilon}(t))+\epsilon \{\theta,V_c\}\mid_{z_{sm}^{\theta,0,\epsilon}(t)} dt\\
&=\theta+\omega_2(I_{tan}(E))\cdot  T_1^{tan}(E)+\epsilon g(I_{tan}^{\epsilon}(\theta),\theta;\epsilon),
\end{split}
\end{equation}
where the \(C^{r }\) smooth function \( g(I_{tan}^{\epsilon}(\theta),\theta;\epsilon)\)  includes integration along the perturbed tangent segment (and does not have a simple approximation by the unperturbed dynamics).
For small negative \(\rho\) we obtain (see first line of Eq. (\ref{eq:thebarnoimpact})) that $$\bar{\theta}( z^{\theta,\rho,\epsilon})=\theta+\int_0^{T_1^{\theta,0,\epsilon}+\mathcal{O}_{C^{r}}(\rho)}\{\theta,H_{int}+\epsilon V_c\}\mid_{z_{sm}^{\theta,0,\epsilon}(t)+\mathcal{O}_{C^{r}}(\rho)} dt
=\bar{\theta}_{tan}+\mathcal{O}_{C^{r}}(\rho),$$so, as \(\bar{I}_{tan}^{0}(\cdot)=I_{tan}(E)\) and  is independent of its argument,
\(
\bar{I}_{tan}^{\epsilon}(\bar{\theta}_{tan})=\bar{I}_{tan}^{\epsilon}(\bar{\theta})+\mathcal{O}_{C^{r-1}}(\epsilon\rho)\), and thus, by Eq. (\ref{eq:ibarcalc1}):
\begin{equation}
\bar{I}( z^{\theta,\rho,\epsilon})=\bar{I}_{tan}^{\epsilon}(\bar{\theta})-\rho+\mathcal{O}_{C^{r-1}}(\epsilon\rho)=I+\bar{I}^{\epsilon}_{tan}(\bar{\theta})-I^\epsilon_{tan}(\theta)+\mathcal{O}_{C^{r-1}}(\epsilon(I_{tan}^{\epsilon}(\theta)-I))
\end{equation}
 thus Eq.  \ref{eq:neartanItheta}a is established.

\paragraph*{The return map for impacting i.c. (Eq.  \ref{eq:neartanItheta}b):}

 For negative \(\rho\) (the red trajectory segment in Figures \ref{fig:unperttangency} and \ref{fig:perttangency}),  we  divide \(z_{im}^{\theta,\rho,\epsilon}(t)\) to two smooth segments, before and after the impact. We  employ regular perturbation theory for estimating the change in \((\theta,I)\) along these segments (see also \cite{pnueli2018near}), showing that these depend smoothly on the travel times to and from the impact. We then study the singular dependence of the travel times and establish  (\ref{eq:neartanItheta}b).
To this aim we introduce the following  notations:

\begin{itemize}

\item Let $z_{im}^{\theta,\rho,\epsilon}(t):=\Phi^{\epsilon,im}_{t}z^{\theta,\rho,\epsilon}$ denote the  wall trajectory segment with initial condition $z^{\theta,\rho,\epsilon}_{im}(0)=z^{\theta,\rho,\epsilon} $ and \(t\in[0,\tilde{T}_1^{\theta,\rho,\epsilon}]\), where \(\tilde{T}_1^{\theta,\rho,\epsilon}=t^\epsilon(z^{\theta,\rho,\epsilon})\) is  the return time of the wall trajectory $z_{im}^{\theta,\rho,\epsilon}(t)$ to the cross-section $\Sigma_E$ and \(t^{-,\epsilon}(z^{\theta,\rho,\epsilon}),t^{+,\epsilon}(z^{\theta,\rho,\epsilon})\) denote the corresponding travel times between \(\Sigma_{E}\) to  \(\Sigma^*_{E}\) and back. Thus,  \(\mathcal{F}_\epsilon^{loc}\Pi z^{\theta,\rho,\epsilon}=\Pi z_{im}^{\theta,\rho,\epsilon}(\tilde{T}_1^{\theta,\rho,\epsilon})\). \item Let \(z^{w}=\Phi^{\epsilon}_{t^{-,\epsilon}(z)}z^{\theta,\rho,\epsilon}=(q_{1}^w,p_{1}^{w},\theta^{w},I^w)\) and let \(\rho ^{w}\) denote the action-distance from tangency at the wall    \({}\rho ^{w}=I_{tan,\Sigma^*_E}^{\epsilon}(\theta^{w})-I^w\), where \(I_{tan,\Sigma^*_E}^{\epsilon}(\theta^{w})\) is the tangency circle at the wall.   \item Let $\hat z_{sm}^{\theta,\rho,\epsilon}(t):=\Phi^{\epsilon}_{t-\hat T_1^{\theta,\rho,\epsilon}}
\circ\Phi^{\epsilon}_{t^{+,\epsilon}(z)}\circ\mathcal{ R}_{1}\circ\Phi^{\epsilon}_{t^{-,\epsilon}(z)}z^{\theta,\rho,\epsilon}$ denote  the  smooth flow from the backward smooth pre-image of  \(z_{im}^{\theta,\rho,\epsilon}(\tilde{T}_1^{\theta,\rho,\epsilon})=\hat z_{sm}^{\theta,\rho,\epsilon}(\hat T_1^{\theta,\rho,\epsilon})\), so
 \(\hat T_1^{\theta,\rho,\epsilon}=T_1^\epsilon(\hat z_{sm}^{\theta,\rho,\epsilon}(0))\). Equivalently, there exists \(\hat \rho= \hat \rho(\theta,\rho,\epsilon), \hat \theta=\hat \theta(\theta,\rho,\epsilon)\) s.t \(\hat z_{sm}^{\theta,\rho,\epsilon}(t)= z_{sm}^{\hat \theta,\hat\rho,\epsilon}(t)\) and \(\hat I= I_{tan}^{\epsilon}(\hat \theta)-\hat \rho\) (the dependence of  \((\hat \theta,\hat\rho)\) is non-smooth near tangency, i.e. near \(\rho=0\), see below). By construction, this is the smooth trajectory segment which coincides with the impacting trajectory after the impact: \(z_{sm}^{\hat \theta,\hat\rho,\epsilon}(t)=z_{im}^{\theta,\rho,\epsilon}(t), \,t\in[t^{+,\epsilon}(z^{\theta,\rho,\epsilon}),\tilde{T}_1^{\theta,\rho,\epsilon}]\).  \item  Let \(t^{\star}\) and \(\hat t^{\star}\) denote the forward and backward flight times from the wall section and its reflection to the auxiliary section  \(\Sigma^{<}_{E}\), so    \(z_{min}:=\Phi^{\epsilon}_{t^{\star}}z^w\) and \(\hat z_{min}:=\Phi^{\epsilon}_{-\hat t^{\star}}\mathcal{R}_1z^w\) are in \(\Sigma^{<}_{E}\),   \(t^{\star}( z^{\epsilon,\rho, \theta})=t^{h,\epsilon}( z^{\theta,\rho,\epsilon})-t^{-,\epsilon}(z^{\theta,\rho,\epsilon})\)  and  \(\hat t^{\star}( z^{\theta,\rho,\epsilon})=T_1^{\hat \theta,\hat\rho,\epsilon} -t^{h,\epsilon}( z^{\hat \theta,\hat\rho,\epsilon})-t^{+,\epsilon}(z^{\hat \theta,\hat\rho,\epsilon})\)).\end{itemize}

When\  $\epsilon=0$ the return times of the  wall and smooth trajectories (i.e. of $z_{im}^{\theta,\rho,0}(t)$ and $z_{sm}^{\theta,\rho,0}(t)$)  do not depend on $\theta$, so  $T_1^{\theta,\rho,0}=T_1^{\rho}=T_1(J(I_{tan}(E)-\rho,E))$, similarly  $\tilde{T}_1^{\theta,\rho,0}=\tilde{T}_1^{\rho}=\tilde{T}_1(J(I_{tan}(E)-\rho,E))$,   \(\hat \rho(\theta,\rho,0)=\rho, \hat \theta(\theta,\rho,0)=\theta-2\Delta t_{travel}(J(I_{tan}(E)-\rho ,E))\omega_{2}(I_{tan}(E)-\rho))\), and \(t^{\star}( z^{\theta,\rho,0})=\hat t^{\star}( z^{\theta,\rho,0})=\Delta t_{travel}(J(I_{tan}(E)-\rho ,E))\).

Next, by integrating the brackets \(\{I,H\},\{\theta,H\}\) along properly chosen segments of \(z_{sm}^{\theta,\rho,\epsilon}(t)\) and \( z_{sm}^{\hat \theta,\hat\rho,\epsilon}(t)\) we establish the leading order behavior of the return map and  show that the error terms are small and depend smoothly on \(t^{\star}\) and \(\hat t^{\star}\). We then establish that these flight times depend smoothly on   \(\sqrt{\rho^{w}}\).  We complete the proof by showing that  functions that are small and depend smoothly on \(\sqrt{\rho^{w}}\) also  depend smoothly on \(\sqrt{\rho}\), so that the error terms are indeed of the required form.

\begin{lem}\label{lem:ibarthetbarimp}
The functions \(\bar{I}=I(\mathcal{F}_{\epsilon} \Pi z^{\theta,\rho,\epsilon}),\bar{\theta}=\theta(\mathcal{F}_{\epsilon} \Pi z^{\theta,\rho,\epsilon})\) and \(I^{w}=I(\Phi^{\epsilon}_{t^{-,\epsilon}(z^{\theta,\rho,\epsilon})}z^{\theta,\rho,\epsilon}),\) \(\theta^w=\theta(\Phi^{\epsilon}_{t^{-,\epsilon}(z^{\theta,\rho,\epsilon})}z^{\theta,\rho,\epsilon})\) are smooth functions of \(t^{\star},\hat t^{\star},\hat \rho,\hat \theta,\theta,\rho,\epsilon\) satisfying:\begin{equation}\label{eq:leadingorder}
\begin{split}
\bar{I}&=I+(\bar{I}_{tan}^{\epsilon}(\bar\theta)-I_{tan}^{\epsilon}(\theta))+\mathcal{O}_{C^{r}}(\rho\epsilon,\hat \rho\epsilon,t^{\star}\epsilon,\hat t^{\star}\epsilon, (\hat \theta-\theta)\epsilon,\rho^{2}) \\\bar{\theta}&=\theta+\omega_2(I_{tan}(E))\cdot T_1^{tan}(E)-\rho\cdot\tau(I_{tan}(E))-\omega_2(I_{tan}(E))\cdot(t^{\star}+\hat t^{\star})\\&
\qquad \qquad +
 \mathcal{O}_{C^{r-1}}(\epsilon,\rho^2,\rho-\hat \rho,(\rho-\hat \rho)\hat  t^{\star},\epsilon t^{\star},\epsilon\hat  t^{\star}).
\end{split}
\end{equation} and \begin{equation}\label{eq:iwtheaw}
\begin{split}
I^w&=I+\epsilon\int_0^{\frac{1}{2}T_1^{tan}}\{I,V_c\}\mid_{z_{sm}^{\theta,0,0}(t)} dt+\mathcal{O}_{C^{r}}(\rho\epsilon,t^{\star}\epsilon,\epsilon^{2})\\
&=\hat I+\epsilon\int_0^{\frac{1}{2}T_1^{tan}}\{I,V_c\}\mid_{z_{sm}^{\hat\theta,0,0}(t)} dt+\mathcal{O}_{C^{r}}(\hat\rho\epsilon,\hat t^{\star}\epsilon,\epsilon^{2}) \\\theta^w&=\theta+\frac{1}{2}\omega_2(I_{tan}(E))\cdot T_1^{tan}(E)-\frac{1}{2}\rho\cdot\tau(I_{tan}(E))-\omega_2(I_{tan}(E))\cdot t^{\star} +
 \mathcal{O}_{C^{r-1}}(\epsilon,\rho^2,\epsilon t^{\star})
\end{split}
\end{equation}\end{lem} \begin{proof}

The leading order change in the action is calculated by integrating the \(\{I,H\}\) bracket along \(z_{sm}^{\theta,\rho,\epsilon}(t)\)  before the impact and along
 \(\hat z_{sm}^{\theta,\rho,\epsilon}(t)=z_{sm}^{\hat\theta,\hat\rho,\epsilon}(t)\) after
the impact, and approximating each segment separately by the corresponding smooth tangent trajectory:\begin{equation}\label{eq:ibarcalc2}
\begin{split}
\bar{I}=I(\mathcal{F}_{\epsilon} \Pi z^{\theta,\rho,\epsilon})&=I(\Pi z^{\theta,\rho,\epsilon})+\int_0^{t^{-,\epsilon}(z^{\theta,\rho,\epsilon})+t^{+,\epsilon}(z^{\theta,\rho,\epsilon})}\{I,H_{int}+\epsilon V_c\}\mid_{z_{im}^{\theta,\rho,\epsilon}(t)} dt \\
& =I+\epsilon\int_0^{t^{-,\epsilon}(z^{\theta,\rho,\epsilon})}\{I,V_c\}\mid_{z_{sm}^{\theta,\rho,\epsilon}(t)} dt+\\
& \qquad \qquad \ +\epsilon\int_{\hat T_1^{\theta,\rho,\epsilon}-t^{+,\epsilon}(z^{\theta,\rho,\epsilon})}^{\hat T_1^{\theta,\rho,\epsilon}}\{I,V_c\}\mid_{ z_{sm}^{\hat\theta,\hat\rho,\epsilon}(t)} dt\\
& =I+\epsilon\int_0^{t^{h,\epsilon}(z^{\theta,\rho,\epsilon})}\{I,V_c\}\mid_{z_{sm}^{\theta,\rho,\epsilon}(t)} dt-\epsilon\int_{t^{h,\epsilon}(z^{\theta,\rho,\epsilon})}^{t^{h,\epsilon}(z^{\theta,\rho,\epsilon})+t^{\star}}\{I,V_c\}\mid_{z_{sm}^{\theta,\rho,\epsilon}(t)}dt\\
& \qquad  \ +\epsilon\int_{ T_1^{\hat\theta,\hat\rho,\epsilon}-t^{h,\epsilon}( z^{\hat\theta,\hat\rho,\epsilon}) }^{ T_1^{\hat\theta,\hat\rho,\epsilon}}\{I,V_c\}\mid_{ z_{sm}^{\hat\theta,\hat\rho,\epsilon}(t)} dt-\epsilon\int_{T_1^{\hat\theta,\hat\rho,\epsilon}-t^{h,\epsilon}( z^{\hat\theta,\hat\rho,\epsilon})}^{T_1^{\hat\theta,\hat\rho,\epsilon}-t^{h,\epsilon}( z^{\hat\theta,\hat\rho,\epsilon})+\hat t^{\star}}\{I,V_c\}\mid_{z_{sm}^{\hat\theta,\hat\rho,\epsilon}(t)}\\
&\\
& =I+\epsilon\int_0^{t^{h,\epsilon}(z^{\theta,0,\epsilon})}\{I,V_c\}\mid_{z_{sm}^{\theta,0,\epsilon}(t)} dt+\mathcal{O}_{C^{r}}(\rho\epsilon)-\epsilon\int_{t^{h,\epsilon}(z^{\theta,\rho,\epsilon})}^{t^{h,\epsilon}(z^{\theta,\rho,\epsilon})+t^{\star}}\{I,V_c\}\mid_{z_{sm}^{\theta,\rho,\epsilon}(t)}dt\\
& \qquad  \ +\epsilon\int_{ T_1^{\hat\theta,0,\epsilon}-t^{h,\epsilon}( z^{\hat\theta,0,\epsilon}) }^{ T_1^{\hat\theta,0,\epsilon}}\{I,V_c\}\mid_{ z_{sm}^{\hat\theta,0,\epsilon}(t)} dt+\mathcal{O}_{C^{r}}(\hat \rho\epsilon)-\epsilon\int_{T_1^{\hat\theta,\hat\rho,\epsilon}-t^{h,\epsilon}( z^{\hat\theta,\hat\rho,\epsilon})}^{T_1^{\hat\theta,\hat\rho,\epsilon}-t^{h,\epsilon}( z^{\hat\theta,\hat\rho,\epsilon})+\hat t^{\star}}\{I,V_c\}\mid_{z_{sm}^{\hat\theta,\hat\rho,\epsilon}(t)}
\\&=I+(\bar{I}_{tan}^{\epsilon}(\bar\theta_{tan})-I_{tan}^{\epsilon}(\theta))+\mathcal{O}_{C^{r}}(\rho\epsilon,\hat \rho\epsilon,t^{\star}\epsilon,\hat t^{\star}\epsilon,\epsilon(\hat \theta- \theta))
\end{split}
\end{equation}
where,  we used the smooth dependence of \(t^{h,\epsilon}(z^{\theta,\rho,\epsilon})\) on parameters and i.c.,  that  all the integrands are smooth and that the integration periods are finite. It follows that the non-smooth behavior arises only from the singular dependence of the lengths of the intervals of integration (\(t^{\star}\) and \(\hat t^{\star}\)) and from the non-smooth dependence of \(\hat \rho\) and \(\hat \theta\) on \(\theta,\rho\).  So, up to the different form of the error terms we have the same expression as in Eq. (\ref{eq:ibarcalc1}), and, to establish the leading order behavior we again only need to show that \(
 \bar{I}_{tan}^{\epsilon}(\bar\theta_{tan}(\theta))\) and \(\bar{I}_{tan}^{\epsilon}(\bar{\theta})
\)
are close in the same sense. Since \(\bar{I}_{tan}^{\epsilon}\) is a smooth function and its dependence on its argument is of order \(\epsilon\), we establish this by showing that \((\bar{\theta}-\bar\theta_{tan})=\mathcal{O}_{C^{r}}(\hat \theta - \theta ,t^{\star},\hat t^{\star},\rho,\hat\rho)\).
Indeed, the image of the angle is:\begin{equation}\label{eq:thetbar1}\begin{split}
\bar{\theta}(z^{\theta,\rho,\epsilon})&=\theta+\int_0^{t^{-,\epsilon}(z^{\theta,\rho,\epsilon})}\{\theta,H\}\mid_{z_{im}^{\theta,\rho,\epsilon}(t)}dt+\int_{t^{+,\epsilon}(z^{\theta,\rho,\epsilon})}^{\tilde{T}_1^{\theta,\rho,\epsilon}}\{\theta,H\}\mid_{z_{im}^{\theta,\rho,\epsilon}(t)}dt  \\
 &=\theta+\int_0^{t^{-,\epsilon}(z^{\theta,\rho,\epsilon})}+\int_{t^{+,\epsilon}(z^{\theta,\rho,\epsilon})}^{\tilde{T}_1^{\theta,\rho,\epsilon}}[\omega_2(I(\cdot))+\epsilon\{\theta,V_{c}\}]\mid_{z_{im}^{\theta,\rho,\epsilon}(t)}dt  \\ &=\theta+\int_0^{t^{h,\epsilon}(z^{\theta,\rho,\epsilon})}
 -\int_{t^{-,\epsilon}(z^{\theta,\rho,\epsilon})}^{t^{h,\epsilon}(z^{\theta,\rho,\epsilon})}[]\mid_{z_{sm}^{\theta,\rho,\epsilon}(t)}
 \\&\qquad\ +
 \int_{T_1^{\epsilon,\hat\rho,\hat \theta}-t^{h,\epsilon}(z^{\epsilon,\hat\rho,\hat \theta})}^{ T_1^{\epsilon,\hat\rho,\hat \theta}}
-\int^{T_1^{\epsilon,\hat\rho,\hat \theta}-t^{+,\epsilon}(z^{\theta,\rho,\epsilon})}_{ T_1^{\epsilon,\hat\rho,\hat \theta}-t^{h,\epsilon}(z^{\epsilon,\hat\rho,\hat \theta})}[]\mid_{z_{sm}^{\epsilon,\hat\rho,\hat \theta}(t)}dt
\end{split}
\end{equation}
Hence, it is close to the tangent angle image, \(\bar\theta_{tan}(\theta)=\bar{\theta}(z^{\theta,0,\epsilon})\):
\begin{equation}\label{eq:thetbar2}\begin{split}
\bar{\theta}(z^{\theta,\rho,\epsilon})&=\theta+\int_0^{t^{h,\epsilon}(z^{\theta,\rho,\epsilon})}
 []\mid_{z_{sm}^{\theta,\rho,\epsilon}(t)} +
 \int_{T_1^{\epsilon,\hat\rho,\hat \theta}-t^{h,\epsilon}(z^{\epsilon,\hat\rho,\hat \theta})}^{ T_1^{\epsilon,\hat\rho,\hat \theta}}
[]\mid_{z_{sm}^{\epsilon,\hat\rho,\hat \theta}(t)}dt+\mathcal{O}_{C^{r}}(t^{\star},\hat t^{\star})
 \\&=\theta+\int_0^{t^{h,\epsilon}(z^{\theta,0,\epsilon})}
 []\mid_{z_{sm}^{\theta,0,\epsilon}(t)} +
 \int_{T_1^{\epsilon,0,\hat \theta}-t^{h,\epsilon}(z^{\epsilon,0,\hat \theta})}^{ T_1^{\epsilon,0,\hat \theta}}
[]\mid_{z_{sm}^{\epsilon,0,\hat \theta}(t)}dt+\mathcal{O}_{C^{r}}(t^{\star},\hat t^{\star},\rho,\hat\rho)
 \\&=\bar{\theta}(z^{\theta,0,\epsilon})+\mathcal{O}_{C^{r}}(\hat \theta - \theta ,t^{\star},\hat t^{\star},\rho,\hat\rho).
\end{split}
\end{equation}Thus, we established the claimed result for \(\bar I\).

The form of \(I^{w}\) follows by the same approximations as in (\ref{eq:ibarcalc2}):
\begin{equation}\begin{split}
I^{w}(z^{\theta,\rho,\epsilon})&=I(\Pi z^{\theta,\rho,\epsilon})+\int_0^{t^{-,\epsilon}(z^{\theta,\rho,\epsilon})}\{I,H_{int}+\epsilon V_c\}\mid_{z_{im}^{\theta,\rho,\epsilon}(t)} dt \\
&=I+\epsilon\int_0^{\frac{1}{2}T_1^{tan}}\{I,V_c\}\mid_{z_{sm}^{\theta,0,0}(t)} dt+\mathcal{O}_{C^{r}}(\rho\epsilon,t^{\star}\epsilon,\epsilon^{2}).
\end{split}
\end{equation}
and similarly\begin{equation}\begin{split}
I^{w}(z^{\theta,\rho,\epsilon})&=I(\Pi z^{\hat\theta,\hat\rho,\epsilon})+\int_0^{ T_1^{\hat\theta,\hat\rho,\epsilon}-t^{+,\epsilon}(z^{\theta,\rho,\epsilon})}\{I,H_{int}+\epsilon V_c\}\mid_{z_{sm}^{\hat\theta,\hat\rho,\epsilon}(t)} dt \\
&=\hat I+\int_0^{ t^{h,\epsilon}(z^{\hat\theta,\hat\rho,\epsilon})+\hat t^{\star}}\{I,H_{int}+\epsilon V_c\}\mid_{z_{sm}^{\hat\theta,\hat\rho,\epsilon}(t)} dt \\
&=\hat I+\epsilon\int_0^{\frac{1}{2}T_1^{tan}}\{I,V_c\}\mid_{z_{sm}^{\hat \theta,0,0}(t)} dt+\mathcal{O}_{C^{r}}(\hat \rho\epsilon,\hat t^{\star}\epsilon,\epsilon^{2}).
\end{split}
\end{equation}

Finally, notice that to leading order in \(\epsilon\), Eq. (\ref{eq:thetbar1}) becomes\begin{equation}
\begin{split}
 \bar{\theta}(z^{\theta,\rho,\epsilon})&=\theta+\int_0^{t^{h,0}(z^{\theta,\rho,0})}-\int_{t^{h,\epsilon}(z^{\theta,\rho,\epsilon})-t^{\star}}^{t^{h,\epsilon}(z^{\theta,\rho,\epsilon})}\omega_2(I(\cdot))\mid_{z_{sm}^{\theta,\rho,0}(t)}dt+\mathcal{O}_{C^{r}}(\epsilon ,\epsilon t^{\star})
\\& \qquad +
 \int_{T_1^{\hat\theta,\hat\rho,0}-t^{h,0}(z^{\hat\theta,\hat\rho,0})}^{ T_1^{\hat\theta,\hat\rho,0}}
-\int^{T_1^{\hat\theta,\hat\rho,\epsilon}-t^{h,\epsilon}(z^{\hat\theta,\hat\rho,\epsilon})+\hat t^{\star}}_{ T_1^{\hat\theta,\hat\rho,\epsilon}-t^{h,\epsilon}(z^{\hat\theta,\hat\rho,\epsilon })}\omega_2(I(\cdot))\mid_{z_{sm}^{\hat\theta,\hat\rho,0}(t)} +\mathcal{O}_{C^{r}}(\epsilon ,\epsilon\hat  t^{\star})
 \\ &\\
&=\theta+\omega_2(I_{tan}(E)-\rho)\cdot(\frac{1}{2}T_1^{\rho}-t^{\star})+\mathcal{O}_{C^{r}}(\epsilon ,\epsilon t^{\star})
\\& \qquad \qquad +\omega_2(I_{tan}(E)-\hat \rho)\cdot( \frac{1}{2}T_1^{\hat \rho} -\hat t^{\star})+\mathcal{O}_{C^{r}}(\epsilon ,\epsilon\hat  t^{\star})\\
&=\theta+\omega_2(I_{tan}(E)-\rho)\cdot (T_1^{\rho}-(t^{\star}+\hat t^{\star}))+\mathcal{O}_{C^{r}}(\epsilon ,\epsilon t^{\star},\epsilon\hat  t^{\star},\rho-\hat \rho,(\rho-\hat \rho)\hat  t^{\star})
\end{split}
\end{equation}

 So we established (compare with Eq. (\ref{eq:thebarnoimpact})):\begin{equation}\label{eq:thetbartstar}\begin{split}
\bar{\theta}(z^{\theta,\rho,\epsilon})&=\theta+\omega_2(I_{tan}(E))\cdot T_1^{tan}(E)-\rho\cdot\tau(I_{tan}(E))-\omega_2(I_{tan}(E))\cdot(t^{\star}+\hat t^{\star})\\&
\qquad \qquad +
 \mathcal{O}_{C^{r-1}}(\epsilon,\rho^2,\rho-\hat \rho,(\rho-\hat \rho)\hat  t^{\star},\epsilon t^{\star},\epsilon\hat  t^{\star}).
\end{split}
\end{equation}
Also,
by the same calculations, the angle at the wall may be approximated by:\begin{equation}\begin{split}
\theta^w(z^{\theta,\rho,\epsilon})&=\theta+\omega_2(I_{tan}(E)-\rho)\cdot(\frac{1}{2} T_1^{\rho}-t^{\star})+
 \mathcal{O}_{C^{r}}(\epsilon,\epsilon t^{\star})\\
&=\theta+\frac{1}{2}\omega_2(I_{tan}(E))\cdot T_1^{tan}(E)-\frac{1}{2}\rho\cdot\tau(I_{tan}(E))-\omega_2(I_{tan}(E))\cdot t^{\star}\\&
\qquad \qquad +
 \mathcal{O}_{C^{r}}(\epsilon,\rho^2,\epsilon t^{\star},\rho t^{\star})
\end{split}
\end{equation}  Integrating backwards from \(\bar{\theta}(z^{\theta,\rho,\epsilon})\) we obtain, by the same arguments, and then by using (\ref{eq:thetbartstar}):\begin{equation}\begin{split}
\theta^w(z^{\hat\theta,\hat\rho,\epsilon})&=\bar{\theta}(z^{\theta,\rho,\epsilon})-\omega_2(I_{tan}(E)-\hat \rho)\cdot(\frac{1}{2} T_1^{\hat \rho}-\hat t^{\star})+
 \mathcal{O}_{C^{r}}(\epsilon,\epsilon\hat  t^{\star})\\
&=\bar{\theta}(z^{\theta,\rho,\epsilon})-\frac{1}{2}\omega_2(I_{tan}(E))\cdot T_1^{tan}(E)+\frac{1}{2}\hat \rho\cdot\tau(I_{tan}(E))+\omega_2(I_{tan}(E))\cdot \hat t^{\star}\\&
\qquad \qquad +
 \mathcal{O}_{C^{r}}(\epsilon,\hat \rho^2,\epsilon\hat  t^{\star},\hat \rho\hat  t^{\star})\\
 &=\theta^w(z^{\theta,\rho,\epsilon})+
 \mathcal{O}_{C^{r-1}}(\epsilon,\rho^2,\rho-\hat \rho,(\rho-\hat \rho)\hat  t^{\star},\epsilon t^{\star},\epsilon\hat  t^{\star},\hat \rho\hat  t^{\star})
\end{split}
\end{equation}

So we indeed established the required forms. \end{proof}
 Next, we find the leading order behavior of \((t^{\star}+\hat t^{\star})\) and establish that all the error terms in Eqs.  (\ref{eq:ibarcalc2}) and (\ref{eq:thetbartstar})  depend smoothly on \(\sqrt{\rho^{w}}\), the square root of the action-distance from tangency at the wall. \begin{lem}\label{lem:tstar}For sufficiently small  $\epsilon$, for any \(z^{w}\in\Sigma^{w}_{E}\),  with \(S_{2}(q_{2}^{w},p^{w}_{2})=(\theta^{w},I^{w}=I_{tan,\Sigma^*_E}^{\epsilon}(\theta^{w})-\rho^{w}) \) and small \(\rho^{w}>0 \), there exist unique flight times, \(t^{\star}\) and \(\hat t^{\star}\) such that   \(z_{min}:=\Phi^{\epsilon}_{t^{\star}}z^w\) and \(\hat z_{min}:=\Phi^{\epsilon}_{-\hat t^{\star}}\mathcal{R}_1z^w\) are in \(\Sigma^{<}_{E}\). The flight times are smooth functions of \(\sqrt{\rho^{w}},\theta^{w}\),\(E\) and \(\epsilon\) and are of the form:\begin{equation}\label{eq:tstarlemma}
t^{\star} ,\hat t^{\star} =
\frac{\sqrt{2\rho^{w}\omega_{2}(I_{tan})}}{-V'_1(q_{1}^{w})}(1+O_{C^{r-1}}(\epsilon,\sqrt{\rho^{w}})),
\end{equation} namely, \( \hat t^{\star} =t^{\star}(1+O_{C^{r-1}}(\epsilon,\sqrt{\rho^{w}})) \) (smoothness here  is w.r.t. \(z^{w}\)).     \end{lem}
\begin{proof}
      Since \(z^{w}=(q_{1}^w,p^{w}_1,q^{w}_{2},p^{w}_2)\in\Sigma^{w}_{E}\), the iso-energy momentum, \(p^{w}_1=- p^{\epsilon}_1(\rho^{w},\theta^{w})\), is given by (the definition of \(\rho^{w}\) implies that  \(p^{\epsilon}_1(0,\theta^{w})=0\)): \begin{displaymath}
\begin{split}
p^{\epsilon}_1(\rho^{w},\theta^{w})&=\sqrt{2(E-V_1(q_1^w)-\frac{(p^{w}_2)^2}{2}-V_2(q^{w}_2)-\epsilon V_c(q_1^w,q_2^{w}(\rho^{w},\theta^{w})))} \\&=\sqrt{2(E-V_1(q_1^w)-H_{2}(I_{tan,\Sigma^*_E}^{\epsilon}(\theta^{w})-\rho^{w})-\epsilon V_c(q_1^w,q^{w}_2(I_{tan,\Sigma^*_E}^{\epsilon}(\theta^{w})-\rho^{w},\theta^{w}))}
\\&=\sqrt{\rho^{w}}\sqrt{2(H_{2}'(I_{tan,\Sigma^*_E}^{\epsilon}(\theta^{w}))+\epsilon \frac{\partial V_c(q_1^w,q^{w}_2(I_{tan,\Sigma^*_E}^{\epsilon}(\theta^{w}),\theta^{w})))}{\partial I})+O_{C^{r-1}}(\rho^{w})}
\\&=\sqrt{\rho^{w}}\sqrt{2\omega_{2}(I_{tan}(E))+O_{C^{r-1}}(\rho^{w},\epsilon)}
\\&
\end{split}\end{displaymath}
 We see that \(p^{\epsilon}_1(\rho^{w},\theta^{w},E)\)  is, to leading order, independent of  \(\theta^{w}\), it is    \(C^{r-1}\) smooth in \(\theta^{w}\) and \(\frac{p^{\epsilon}_1(\rho^{w},\theta^{w},E)}{\sqrt{\rho^{w}}}\) is   \(C^{r-1}\) smooth in \(\rho^{w},\theta^{w}\), namely,  \(p^{\epsilon}_1(\rho^{w},\theta^{w})=
 \mathcal{O}_{C^{r-1}}(\sqrt{\rho^{w}},\epsilon)\).

As \(\dot p_{1}=-V'_1(q_1^w)+ O_{C^{r-1}}(\epsilon,z-z^w )>0\) and  \(p_{1}\) is monotone increasing near \(z^{w}\),  for sufficiently small \(-p_{1}^w\) the first crossing time of \(\Sigma_{E}^<\) (the section \(p_{1}=0,\dot p_{1}>0,q_1<q_1^w\)),  \(t^{\star}\), is uniquely defined and depends smoothly on i.c., namely on \(z^{w}\). Since \(-p_{1}^w= \int_0^{t^{\star}} (-V'_1(q_1^w)+ O_{C^{r-1}}(\epsilon,\Phi^{\epsilon}_{t}z^{w}-z^w ))dt=-t^{\star}V'_1(q_1^w)(1+ O_{C^{r}}(\epsilon, t^{\star}))\), we conclude that\begin{equation}\label{eq:p1epatstar}
t^{\star}=\frac{p^{w}_1}{V'_1(q_1^w)}(1+O_{C^{r}}(\epsilon ,p_{1}^{w}))=\frac{\sqrt{2\rho^{w}\omega_{2}(I_{tan})}}{-V'_1(q_1^w)}(1+\mathcal{O}_{C^{r-1}}(\sqrt{\rho^{w}},\epsilon))
\end{equation} \(\) \

Repeating the same computation for negative time for the initial condition \(\mathcal{R}_1z^w\) provides exactly the same results for \(\hat t^\star\), so, by Eq. (\ref{eq:p1epatstar}),  \(\hat t^\star =t^{\star}(1+\mathcal{O}_{C^{r-1}}(\sqrt{\rho^{w}},\epsilon)) \).       \end{proof}
By uniqueness of solutions, \(t^\star(z^{w})=t^\star(z^{\theta,\rho,\epsilon})\) and \(\hat t^\star(z^{w})=\hat t^\star(z^{\theta,\rho,\epsilon}) \). Hence,
to complete the proof we need to show that to leading order we can  replace in formula (\ref{eq:tstarlemma})  \(\rho^{w}\) by \(\rho\) and that the error terms are small in \(\sqrt{\rho}\).
To relate the smoothness w.r.t. \(z^w\) to smoothness w.r.t. to \(\theta,\rho,\epsilon\) we establish:

\begin{lem}\label{lem:rhorhow}For  sufficiently small  $\epsilon>0$, there exist \(C^{r-2}\) functions \(G_\star, G_{\star \star}\) such that \(\sqrt{\rho^{w}}(\sqrt{\rho},\theta;E,\epsilon)=\sqrt{\rho}(1+\epsilon G_\star(\sqrt{\rho},\theta,E,\epsilon))\) and   \(\hat \rho-\rho=\epsilon\rho^{w} G_{\star \star}(\sqrt{\rho^{w}},\theta,E,\epsilon)=O_{C^{r-2}}(\epsilon(\sqrt{\rho})^2)\).         \end{lem}
    \begin{proof} Recall that  \(z^{w}=\Phi^{\epsilon}_{t^{-,\epsilon}(z)}z^{\theta,\rho,\epsilon}=
    \mathcal{R}_1\Phi^{\epsilon}_{T_1^{\hat \theta,\hat\rho,\epsilon}-t^{+,\epsilon}(z)}z^{\hat \theta,\hat\rho,\epsilon}\) where \(\Pi z^{w}=(\theta^{w},I^{w}=I_{tan,\Sigma^*_E}^{\epsilon}(\theta^{w})-\rho^{w})\).
Let  \(\Pi z_{tan}^{w}=(I_{tan,\Sigma^*_E}^{\epsilon}(\theta^{w}),\theta^{w})\) denote the wall point on the tangency curve with the same angle, \(\theta^{w}\) as of \(z^{w}\). Since the  tangency curve at the wall depends smoothly on \(\theta^{w}\), as is the flight time between \(\Sigma_E\) and  \(\Sigma_{E}^{>}\cup\Sigma_{E}^<\),  the singularity-line backward angle is a  \(C^{r}\) function of  \(\theta^{w}\) :  \(\theta_{tan}(\theta^{w},\epsilon)=\theta(\Phi^{\epsilon}_{-t^{h,\epsilon}(z^{\theta_{tan},0,\epsilon})} z_{tan}^{w})\). By definition\begin{displaymath}
I_{tan,\Sigma^*_E}^{\epsilon}(\theta^{w})=I_{tan}^{\epsilon}(\theta_{tan})+\epsilon\int_0^{t^{h,\epsilon}(z^{\theta_{tan},0,\epsilon})}\{I,V_c\}\mid_{z_{sm}^{\theta_{tan},0,\epsilon}(t)} dt, \quad\theta^{w}=\theta(\Phi^{\epsilon}_{t^{h,\epsilon}(z^{\theta_{tan},0,\epsilon})}z^{\theta_{tan},0,\epsilon}).
\end{displaymath}
% and  \(\frac{dI_{tan,\Sigma^*_E}^{\epsilon}(\theta^{w})}{d\theta^{w}}=\frac{d\theta^{w}}{d\theta_{0}}
%\frac{dI_{tan,\Sigma^*_E}^{\epsilon}(\theta^{w})}{d\theta_{0}}=\mathcal{O}_{C^{r}}(\epsilon)\) %and \(\frac{d\theta^{w}}{d\theta_{0}}=1+xxx\).
Since \(z^{w}=\Phi^{\epsilon}_{t^{-,\epsilon}(z)}z^{\theta,\rho,\epsilon}\) we also obtain: \begin{equation}
{I^w}=I_{tan}^{\epsilon}(\theta)-\rho+\epsilon\int_0^{t^{-,\epsilon}(z^{\theta,\rho,\epsilon})}\{I,V_c\}\mid_{z_{sm}^{\theta,\rho,\epsilon}(t)} dt
\end{equation} so, as \(\rho^{w}=I_{tan,\Sigma^*_E}^{\epsilon}(\theta^{w})-I^{w}\), we get:
\begin{equation}\label{eq:rominusrowexact}\begin{split}
\rho^{w}&=\rho +I_{tan}^{\epsilon}(\theta_{tan})-I_{tan}^{\epsilon}(\theta)\\&\qquad+\epsilon\int_0^{t^{h,\epsilon}(z^{\theta_{tan},0,\epsilon})}\{I,V_c\}\mid_{z_{sm}^{\theta_{tan},0,\epsilon}(t)} dt-\epsilon\int_0^{t^{-,\epsilon}(z^{\theta,\rho,\epsilon})}\{I,V_c\}\mid_{z_{sm}^{\theta,\rho,\epsilon}(t)} dt, \end{split}
\end{equation}
  Recall that \(z_{sm}^{\theta,\rho,\epsilon}(t)=z_{sm}^{\theta,0,\epsilon}(t)+\mathcal{O}_{C^{r}}(\rho)
=z_{sm}^{\theta_{ tan},0,\epsilon}(t)+\mathcal{O}_{C^{r}}(\rho,\theta-\theta_{tan})\), and that its travel time to the transverse section \(\Sigma_E^<\), \(t^{h,\epsilon}(z^{\theta,\rho,\epsilon})\), also depends smoothly on \(\theta,\rho\).   Hence,
  \begin{equation}\label{eq:ibarcalc3}
\begin{split}
{\rho^{w}-\rho}&=I_{tan}^{\epsilon}(\theta_{tan})-I_{tan}^{\epsilon}(\theta)\\&+\epsilon\int_0^{t^{h,\epsilon}(z^{\theta_{tan},0,\epsilon})}\{I,V_c\}\mid_{z_{sm}^{\theta_{tan},0,\epsilon}(t)} dt-\epsilon\int_0^{t^{h,\epsilon}(z^{\theta,0,\epsilon})}\{I,V_c\}\mid_{z_{sm}^{\theta,0,\epsilon}(t)} dt+O_{C^{r}}(\epsilon\rho)\\
& \qquad \qquad+\epsilon \int_{0}^{ t^{\star}(\epsilon,\sqrt{\rho^{w}},\theta^{w})}\{I,V_c\}\mid_{\Phi^{\epsilon}_{t}(z^{w})} dt\\
& =\epsilon \int_{0}^{ t^{\star}(\epsilon,\sqrt{\rho^{w}},\theta^{w})}\{I,V_c\}\mid_{\Phi^{\epsilon}_{t}(z^{w})} dt+\epsilon(\theta-\theta_{tan}(\theta^{w}))\tilde g(\theta,\theta^{w},\epsilon)+O_{C^{r}}(\epsilon\rho)\\
& =\epsilon t^{\star}g_{1}(t^{\star},z^{w})+\epsilon (\theta-\theta_{tan}(\theta^{w}))g_{2}(\theta,\theta^{w},\epsilon)+O_{C^{r}}(\epsilon\rho)
\end{split}
\end{equation}
where \(\tilde g,g_{1},g_2\) are  \(C^{r-1}\) functions of their arguments. Notice that:\begin{equation}\label{eq:thetatan1}\begin{split}
\theta_{tan}(\theta^{w},\epsilon)&=\theta^{w}-\int_0^{t^{h,\epsilon}(z^{\theta_{tan},0,\epsilon})}[\omega_2(I(\cdot))+\epsilon\{\theta,V_{c}\}]\mid_{z_{sm}^{\epsilon,0,\theta_{tan}}(t)}dt) \\
& =\theta^{w}-\frac{1}{2}\omega_2(I_{tan}(E))T_1^{tan}(E)+O_{C^{r}}(\epsilon),
\end{split}
\end{equation}
Hence, \(\theta_{tan}'(\theta^{w},\epsilon)=1+O_{C^{r}}(\epsilon) \), and,
\begin{equation}\label{eq:thetaminusthetatan}\begin{split}
\theta({\rho^{w}},\theta^{w},\epsilon)-\theta_{tan}&(\theta^{w},\epsilon)=\theta^{w}-\int_0^{t^{-,\epsilon}(z^{\theta,\rho,\epsilon})}[\omega_2(I(\cdot))+\epsilon\{\theta,V_{c}\}]\mid_{z_{sm}^{\theta,\rho,\epsilon}(t)}dt\\
&\qquad\qquad-(\theta^{w}-\int_0^{t^{h,\epsilon}(z^{\theta_{tan},0,\epsilon})}[]\mid_{z_{sm}^{\theta_{tan},0,\epsilon}(t)}dt) \\
&= \int_0^{t^{h,\epsilon}(z^{\theta_{tan},0,\epsilon})}[]\mid_{z_{sm}^{\theta_{tan},0,\epsilon}(t)}dt-\int_0^{t^{h,\epsilon}(z^{\theta,\rho,\epsilon})-t^{\star}}[]\mid_{z_{sm}^{\theta,\rho,\epsilon}(t)}dt
\\
&= \int_0^{t^{h,\epsilon}(z^{\theta_{tan},0,\epsilon})}[]\mid_{z_{sm}^{\theta_{tan},0,\epsilon}(t)}dt-
\int_0^{t^{h,\epsilon}(z^{\theta,\rho,\epsilon})}[]\mid_{z_{sm}^{\theta,\rho,\epsilon}(t)}dt\\
&\qquad\qquad+\int_0^{t^{\star}}[]\mid_{z_{sm}^{\theta,\rho,\epsilon}(t^{h,\epsilon}(z^{\theta,\rho,\epsilon})-t^{\star}+t)}dt
\\
&= \int_0^{t^{h,0}(z^{\theta_{tan},0,0})}\omega_2(I(\cdot))\mid_{z_{sm}^{0,0,\theta_{tan}}(t)}dt-\int_0^{t^{h,0}(z^{\theta,\rho,0})}\omega_2(I(\cdot))\mid_{z_{sm}^{\theta,\rho,0}(t)}dt+O_{C^{r}}(\epsilon)\\
&\qquad\qquad+\int_0^{t^{\star}}[]\mid_{z_{sm}^{\theta,\rho,\epsilon}(t^{h,\epsilon}(z^{\theta,\rho,\epsilon})-t^{\star}+t)}dt \\&=-\frac{1}{2}\rho\cdot\tau(I_{tan}(E))+\int_0^{t^{\star}}\omega_2(I(\cdot))\mid_{z_{sm}^{\theta,\rho,\epsilon}(t^{h,\epsilon}(z^{\theta,\rho,\epsilon})-t^{\star}+t)}dt+\mathcal{O}_{C^{r}}(\epsilon,\rho^{2},\epsilon t^{\star})\\
 &=-\frac{1}{2}\rho\cdot\tau(I_{tan}(E))+\omega_2(I_{tan}(E))\cdot t^{\star} +\mathcal{O}_{C^{r}}(\epsilon,\epsilon t^{\star},\rho t^{\star},\rho^{2}).
\end{split}
\end{equation}  so (\ref{eq:ibarcalc3}) and (\ref{eq:thetaminusthetatan}) become:
  \begin{equation}\label{eq:ibarcalc6}
\begin{split}
{\rho}&+O_{C^{r}}(\epsilon\rho)=\rho^{w}-\epsilon t^{\star}g_{1}(t^{\star},z^{w})-\epsilon (\theta-\theta_{tan}(\theta^{w}))g_{2}(\theta,\theta^{w},\epsilon)
\\
\theta&=\theta_{tan}(\theta^{w})+\omega_2(I_{tan}(E))\cdot t^{\star} -\frac{1}{2}\rho\cdot\tau(I_{tan}(E))+\mathcal{O}_{C^{r}}(\epsilon,\epsilon t^{\star},\rho t^{\star},\rho^{2}).\end{split}
\end{equation}
By proposition \ref{prop:singularcurve}, the solution to this set of equations corresponds to the map which maps uniquely  an iso-energy impacting i.c. (\((\theta,\rho), \rho>0)\)  to an impacting phase-point on the wall  (\((\rho^{w},\theta^{w}), \rho^{w}>0)\) and a tangent i.c. (\((\theta,\rho)=(\theta_{tan}(\theta^{w}),0) \) to a phase point on the wall which belongs to the tangency curve -  the point \((\rho^{w}=0,\theta^{w})\). At \(\epsilon=0\), we  have \({\rho}=\rho^{w},\theta=\theta_{tan}(\theta^{w})+\omega_2(I_{tan}(E))\cdot t^{\star}-\frac{1}{2}\rho\cdot\tau(I_{tan}(E))+\mathcal{O}_{C^{r}}(\rho^{2})  \). By lemma \ref{lem:tstar}, \(t^{\star}=\frac{\sqrt{2\rho^{w}\omega_{2}(I_{tan})}}{-V'_1(q_{1}^{w})}(1+\mathcal{O}_{C^{r-1}}(\epsilon,\sqrt{\rho^{w}}))\), so \(\frac{\partial t^{\star}}{\partial\theta^{w}}=\mathcal{O}_{C^{r-2}}(\epsilon\sqrt{\rho^{w}},\rho^{w})\) hence, for small \((\epsilon,\sqrt{\rho^{w}})\), by the Implicit function Theorem for (\ref{eq:ibarcalc6}), viewed as a set of \(C^{r-1}\) smooth implicit equations of the form \(F(\theta,\rho,\sqrt{\rho^w},\theta^w;\epsilon)=0\),  and noticing that \(\left\vert \frac{\partial F}{\partial(\theta,\rho^{w})}\right\vert=\det\begin{pmatrix}1+O_{C^{r}}(\epsilon) & \mathcal{O}_{C^{r-1}}(\epsilon,\epsilon\sqrt{\rho^{w}}) \\
\frac{1}{2}\tau(I_{tan}(E))+\mathcal{O}_{C^{r-1}}(\sqrt{\rho^{w}},\rho) & -\theta_{tan}'(\theta^{w},\epsilon)+\mathcal{O}_{C^{r-1}}(\epsilon\sqrt{\rho^{w}}) \\
\end{pmatrix}=-1+\mathcal{O}_{C^{r-1}}(\epsilon)\)), there exist a smooth \(C^{r-1}\) function
 \(\tilde g_{\star}\) such that
  \begin{equation}\label{eq:ibarcalc7}
\begin{split}
{\rho}(\sqrt{\rho^{w}},\theta,E,\epsilon)&=\rho^{w}+\epsilon\sqrt{\rho^{w}}\tilde g_{\star}(\sqrt{\rho^{w}},\theta;E,\epsilon)\\&=\epsilon\sqrt{\rho^{w}}
\tilde g_{\star}(0,\theta;E,\epsilon)+\rho^{w}(1+\epsilon\frac{\partial\tilde g_{\star}(0,\theta;E,\epsilon)}{\partial\sqrt{\rho^{w}}} )+\mathcal{O}_{C^{r-3}}(\epsilon(\sqrt{\rho^{w}})^{3})
\\
\theta^{w}(\sqrt{\rho^{w}},\theta,E,\epsilon)&=\theta+\frac{1}{2}\omega_2(I_{tan}(E))T_1^{tan}(E)-\omega_2(I_{tan}(E))\cdot \frac{\sqrt{2\rho^{w}\omega_{2}(I_{tan})}}{-V'_1(q_{1}^{w})}\\&\qquad+\mathcal{O}_{C^{r-1}}(\epsilon\sqrt{\rho^{w}},(\rho^{w})^{2})
\end{split}
\end{equation}
and indeed \(\rho\) and \(\theta^{w}\) depend smoothly on \(\sqrt{\rho^{w}}\) (but not on \(\rho^{w}\)). Since the solution must map the upper half plane to the upper half plane, and since for very small \(\sqrt{\rho^{w}}\)  the  term \(\epsilon\sqrt{\rho^{w}}
\)  dominates \({\rho}(\sqrt{\rho^{w}},\theta,E,\epsilon)\), we conclude that for all \(\theta\) and sufficiently small \(\epsilon\), \(\tilde g_{\star}(0,\theta;E,\epsilon)\geqslant0 \). Thus, by area preservation of the flow this term must vanish. Hence,  \({\rho}(\sqrt{\rho^{w}},\theta,E,\epsilon)=\rho^{w}+\epsilon\rho^{w} g_{\star}(\sqrt{\rho^{w}},\theta;E,\epsilon)\) where  \(g_{\star}(\sqrt{\rho^{w}},\theta;E,\epsilon)\) is  \(C^{r-2}\), namely,  \(\sqrt{\rho}(\sqrt{\rho^{w}},\theta;E,\epsilon)=\sqrt{\rho^{w}}(1+\frac{1}{2}\epsilon g_{\star}(\sqrt{\rho^{w}},\theta;E,\epsilon)+\mathcal{O}_{C^{r-2}}(\epsilon^{2}))\) and  \(\frac{\partial (\sqrt{\rho})}{\partial \sqrt{\rho^{w}}}\)  is close to one. It follows that there exists an inverse function of the form:\begin{equation}
\sqrt{\rho^{w}}(\sqrt{\rho},\theta;E,\epsilon)=\sqrt{\rho}(1+\epsilon G_\star(\sqrt{\rho},\theta,E,\epsilon)).
\end{equation}
% \(F(\sqrt{\rho^{w}},\rho)=0, \frac{\partial F}{\partial \sqrt{\rho^{w}}}=1+O(\epsilon)+\frac{\frac{1}{2}\epsilon ^{2} g_{\star}(\sqrt{\rho^{w}},\theta;E,\epsilon)g_{\star}'}{\sqrt{\frac{1}{4}\epsilon ^{2} g_{\star}(\sqrt{\rho^{w}},\theta;E,\epsilon)^{2}+\rho}}=1+O(\epsilon)+O(min(\frac{\epsilon ^{2} }{\sqrt{\rho}},\frac{\epsilon  g_{\star}(\sqrt{\rho^{w}},\theta;E,\epsilon)g_{\star}'}{\sqrt{\frac{1}{4}g_{\star}(\sqrt{\rho^{w}},\theta;E,\epsilon)^{2}+\frac{\rho}{\epsilon ^{2} }}} )\)

Similarly,  integrating \(z^{\epsilon,\hat \rho,\hat \theta}(t)\) up to \(\mathcal{R}_1z^w\) (recall that \((I(\mathcal{R}_1z^w),\theta(\mathcal{R}_1z^w))=(I^w,\theta^{w})\) since  \(\mathcal{R}_1\) denotes a reflection in the \((q_{1},p_1)\) plain), we obtain: \begin{equation}
{I^w}=I_{tan}^{\epsilon}(\hat \theta)-\hat \rho+\epsilon\int_0^{t^{h,\epsilon}
(z^{\hat \theta,\hat \rho,\epsilon})+\hat t_{\star}}\{I,V_c\}\mid_{z_{sm}^{\hat \theta,\hat \rho,\epsilon}(t)} dt
\end{equation} so (\ref{eq:ibarcalc3}) becomes:
  \begin{equation}\label{eq:ibarcalc3hat}
\begin{split}
\rho^{w}-\hat \rho &=I_{tan}^{\epsilon}(\theta_{tan})-I_{tan}^{\epsilon}(\hat \theta)\\&+\epsilon\int_0^{t^{h,\epsilon}(z^{\epsilon,0,\theta_{tan}})}\{I,V_c\}\mid_{z_{sm}^{\epsilon,0,\theta_{tan}}(t)} dt-\epsilon\int_0^{t^{h,\epsilon}(z^{\epsilon,0,\hat \theta})}\{I,V_c\}\mid_{z_{sm}^{\epsilon,0,\hat \theta }(t)} dt+\mathcal{O}_{C^{r}}(\epsilon\hat \rho)\\
& \qquad \qquad-\epsilon \int_{0}^{ \hat t^{\star}(\epsilon,\sqrt{\rho^{w}},\theta^{w})}\{I,V_c\}\mid_{\Phi^{\epsilon}_{-t}((\mathcal{R}_1z^{w})} dt\\
& =-\epsilon \int_{0}^{ \hat t^{\star}(\epsilon,\sqrt{\rho^{w}},\theta^{w})}\{I,V_c\}\mid_{\Phi^{\epsilon}_{-t}(\mathcal{R}_1z^{w})} dt+\epsilon(\hat \theta-\theta_{tan}(\theta^{w}))\hat g(\hat \theta,\theta^{w},\epsilon)+\mathcal{O}_{C^{r}}(\epsilon\hat \rho)\\
& =\epsilon \hat t^{\star}\hat g_{1}(\hat t^{\star},z^{w})+\epsilon(\hat \theta-\theta_{tan}(\theta^{w}))\hat g_{2}(\hat \theta,\theta^{w},\epsilon)+\mathcal{O}_{C^{r}}(\epsilon\hat \rho)
\end{split}
\end{equation}and   (\ref{eq:thetaminusthetatan})  becomes
\begin{equation}\label{eq:thetaminusthetatanhat}\begin{split}
\hat \theta-\theta_{tan}&(\theta^{w},\epsilon)=\theta^{w}-\int_0^{t^{h,\epsilon}(z^{\hat \theta,\hat \rho,\epsilon})+\hat t_{\star}}[\omega_2(I(\cdot))+\epsilon\{\theta,V_{c}\}]\mid_{z_{sm}^{\hat \theta,\hat \rho,\epsilon}(t)}dt\\
&\qquad\qquad-(\theta^{w}-\int_0^{t^{h,\epsilon}(z^{\theta_{tan},0,\epsilon})}[]\mid_{z_{sm}^{\theta_{tan},0,\epsilon}(t)}dt) \\
&= \int_0^{t^{h,\epsilon}(z^{\theta_{tan},0,\epsilon})}[]\mid_{z_{sm}^{\theta_{tan},0,\epsilon}(t)}dt-\int_0^{t^{h,\epsilon}(z^{\hat \theta,\hat \rho,\epsilon})+\hat t^{\star}}[]\mid_{z_{sm}^{\hat \theta,\hat \rho,\epsilon}(t)}dt
\\
&
\\
&= \int_0^{t^{h,0}(z^{\theta_{tan},0,0})}\omega_2(I(\cdot))\mid_{z_{sm}^{\theta_{tan},0,0}(t)}dt-\int_0^{t^{h,0}(z^{\hat \theta,\hat \rho,0})}\omega_2(I(\cdot))\mid_{z_{sm}^{\hat \theta,\hat \rho,0}(t)}dt\\
&\qquad\qquad-\int_0^{\hat t^{\star}}[]\mid_{\Phi^{\epsilon}_{-t}(\mathcal{R}_1z^{w}) }dt+\mathcal{O}_{C^{r}}(\epsilon) \\&=-\frac{1}{2}\hat \rho\cdot\tau(I_{tan}(E))-\int_0^{\hat t^{\star}}\omega_2(I(\cdot))\mid_{\Phi^{\epsilon}_{-t}(\mathcal{R}_1z^{w}) }dt+\mathcal{O}_{C^{r}}(\epsilon,\hat \rho^{2},\epsilon\hat  t^{\star})\\
 &=-\frac{1}{2}\hat \rho\cdot\tau(I_{tan}(E))-\omega_2(I_{tan}(E))\cdot\hat  t^{\star} +\mathcal{O}_{C^{r}}(\epsilon,\epsilon\hat  t^{\star},\hat \rho \hat  t^{\star},\hat \rho^{2}).
\end{split}
\end{equation} So, we obtain again implicit equations of the form (\ref{eq:ibarcalc6}) for \((\hat \rho,\hat \theta,\sqrt{\rho^{w}},\theta^{w})\) and by the IFT (w.r.t to \((\hat \rho,\hat \theta)\), we conclude that there exist unique smooth solutions \((\hat \rho(\sqrt{\rho^{w}},\theta^{w},\epsilon) ,\hat \theta(\sqrt{\rho^{w}},\theta^{w},\epsilon))\)  and a function \( \tilde {\hat g}_{\star}(\sqrt{\rho^{w}},\theta^{w};E,\epsilon)\) such that \(
\hat {\rho}(\sqrt{\rho^{w}},\theta^{w},E,\epsilon)=\rho^{w}+\epsilon\sqrt{\rho^{w}}\tilde {\hat g}_{\star}(\sqrt{\rho^{w}},\theta^{w};E,\epsilon)
\). Area preservation again implies that %[[if \(\rho ^{w}=\epsilon^{k}s\) then \(\rho=\epsilon^{k}s-a\epsilon^{1+k/2}s\) so if \(k<2\) int. is small... ]]
 \(\tilde {\hat g}_{\star}(0,\theta^{w};E,\epsilon)\) must vanish for sufficiently small \(\epsilon\), so  \(\hat {\rho}(\sqrt{\rho^{w}},\theta^{w},E,\epsilon)=\rho^{w}(1+\epsilon \hat g_{\star}(\sqrt{\rho^{w}},\theta^{w};E,\epsilon))\) for some \(C^{r-2}\) function \( \hat g_{\star}\). Hence, using the above established solution of  \((\theta,\rho^{w})\)  as a function of \((\sqrt{\rho^{w}},\theta)\), we conclude that   \(\hat \rho-\rho=\epsilon\rho^{w}( \hat g_{\star}(\sqrt{\rho^{w}},\theta^{w}(\sqrt{\rho^{w}},\theta,E,\epsilon);E,\epsilon)-g_{\star}(\sqrt{\rho^{w}},\theta;E,\epsilon)):= \epsilon\rho^{w} G_{\star \star}(\sqrt{\rho^{w}},\theta,E,\epsilon)=O_{C^{r-2}}(\epsilon(\sqrt{\rho})^2).\)   \end{proof}

Now we complete the proof of Proposition \ref{lem:pertreturnmap}: By lemmas \ref{lem:tstar} and  \ref{lem:rhorhow},  \(t^{\star}+\hat t^{\star}=\frac{2\sqrt{2\rho^{w}\omega_{2}(I_{tan})}}{-V'_1(q_{1}^{w})}(1+O_{C^{r-1}}(\epsilon,\sqrt{\rho^{w}}))=\frac{2\sqrt{2\rho\omega_{2}(I_{tan})}}{|V'_1(q_{1}^{w})|}(1+O_{C^{r-2}}(\epsilon,\sqrt{\rho^{}}))\).
  Substituting this expression for \(t^{\star}+\hat t^{\star}\)  in (\ref{eq:leadingorder}) the leading order behavior of (\ref{eq:neartanItheta}b) is established. From the equations of motion and lemma \ref{lem:rhorhow}
 \begin{equation}\label{eq:thetaminusthetahat}\begin{split}
\theta-\hat \theta
 &=-\int_0^{t^{h,\epsilon}(z^{\theta,\rho,\epsilon})-t^{\star}}[\omega_2(I(\cdot))+\epsilon\{\theta,V_{c}\}]\mid_{\Phi^{\epsilon}_{-t}(z^{w})}dt+\int_0^{t^{h,\epsilon}(z^{\theta,\rho,\epsilon})+\hat t^{\star}}[]\mid_{\Phi^{\epsilon}_{-t}((\mathcal{R}_1z^{w})}dt\\&=\mathcal{O}_{C^{r}}(t^{\star},  \hat  t^{\star},|\Phi^{\epsilon}_{t^{\star}}z^{w}-\Phi^{\epsilon}_{-\hat t^{\star}}(\mathcal{R}_1z^{w})|)\\&=\mathcal{O}_{C^{r}}(t^{\star},  \hat  t^{\star},\sqrt{\rho^{w}})=\mathcal{O}_{C^{r-2}}(\sqrt{\rho}).
\end{split}
\end{equation}So, by lemma \ref{lem:rhorhow}  the error terms in  (\ref{eq:leadingorder}),  \(\mathcal{O}_{C^{r}}(\rho\epsilon,\hat \rho\epsilon,\epsilon t^{\star},\epsilon\hat  t^{\star},(\theta-\hat \theta)\epsilon)\) and \(\mathcal{O}_{C^{r}}(\epsilon,\rho^2,\rho-\hat \rho,(\rho-\hat \rho)\hat  t^{\star},\epsilon t^{\star},\epsilon\hat  t^{\star})\) all depend smoothly on \(\sqrt{\rho}\) and are of order   \(\mathcal{O}_{C^{r-2}}(\epsilon\sqrt{\rho})\) and \(\mathcal{O}_{C^{r-2}}(\epsilon,\rho^2,\epsilon\sqrt{\rho})\) respectively, so, by the definition of \(\rho\),   the error terms are as in (\ref{eq:neartanItheta}b).

\end{proof}

\subsection{Proof of Theorem \ref{thm:mainthm}} The final form of the local return map near the singularity line follows, by a symplectic change of coordinates, from the form of \(\mathcal{F}_\epsilon^{loc}\) (Eq. (\ref{eq:neartanItheta})).  Let $S(\theta ,K)=K\theta+\int^{\theta}I_{tan}^{\epsilon}(\theta')d\theta'$ denote a generating function of the second kind. Since  \(I_{tan}^{\epsilon}(\theta)\) is a \(C^{r}\) smooth function, \(S(\theta ,K)\) is \(C^{r+1}\). Then:
\begin{equation}\label{eq:Kphitransform}
\begin{split}
I=\frac{\partial S}{\partial\theta}=K+I_{tan}^{\epsilon}(\theta) \Rightarrow K(\theta,I)&=I-I_{tan}^{\epsilon}(\theta) \\
\phi(\theta,I)&=\frac{\partial S}{\partial K}=\theta
\end{split}
\end{equation}
i.e. this is a \(C^{r}\) smooth symplectic transformation for which in the new coordinates $(\phi ,K)$, $K=0$ represents tangency in the perturbed system. The image of $(\theta,I)$ under the local return map is $(\bar{\theta},\bar{I})=\mathcal{F}^{loc}_{\epsilon}(\theta,I)$ and hence $(\bar{\phi},\bar{K})=(\bar{\theta},\bar{I}-I_{tan}^{\epsilon}(\bar{\theta}))$, so, by Equation (\ref{eq:neartanItheta}), using (\ref{eq:gsindef}) and the definition of \(\Omega\) we find:
\begin{subequations}\label{eq:neartangentKphi}
\begin{equation}
K\geq 0: \begin{cases}
\bar{K}=K-I_{tan}^{\epsilon}(\bar{\phi})+\bar{I}_{tan}^{\epsilon}(\bar{\phi})+\mathcal{O}_{C^{r-1}}(\epsilon K) \\
\bar{\phi}=\phi+\Omega+ G_{s}(K)+\mathcal{O}_{C^{r-1}}(\epsilon ,K^{2})
\end{cases}
\end{equation}
\begin{equation}
K<0: \begin{cases}
\bar{K}=K-I_{tan}^{\epsilon}(\bar{\phi})+\bar{I}_{tan}^{\epsilon}(\bar{\phi})+\mathcal{O}_{C^{r-2}}(\epsilon\sqrt{-K} ) \\
\begin{aligned}
\bar{\phi}=\phi+\Omega&+ G_{s}(K)+\mathcal{O}_{C^{r-2}}(\epsilon,K^{2},\epsilon\sqrt{-K} )
\end{aligned}
\end{cases}
\end{equation}
\end{subequations}
As,  $\epsilon f(\bar{\phi})=\bar{I}^{\epsilon}_{tan}(\bar{\phi})-I^{\epsilon}_{tan}(\bar{\phi})$ and \(G_{s}(K)\) is of order \(K\) for positive \(K \) and of order \( \sqrt{-K} \) for negative \(K\), the form (\ref{eq:pert-return}) is established.
\qed

\medskip

\section{Stable dynamics near tangency}\label{sec:stable}
The local return map of section \ref{sec:returnmap} is valid for $I-I_{tan}(E)=o(1)$. For $I-I_{tan}(E)=\mathcal{O}(1)$, it is known that stable, near integrable behavior arises \cite{pnueli2018near}, namely,  away from the singularity line  orbits are bounded away from being tangent and remain, forever (for 2 d.o.f. systems), in either the impacting regime or the non-impacting regime.  The intermediate boundary layer, at  which the local return map (\ref{eq:pert-return}) has  stable, near integrable behavior is identified next: we show that  for sufficiently small $\gamma>0$ and bounded \(k_s\), initial conditions with  $K=k_s\epsilon^\gamma$ do  not cross the singularity line at $K=0$ and their dynamics are  near integrable. The bounds on the values of $\gamma$ differ between the two sides of the tangency curve; For any \((\alpha>0,\nu>0,\sigma\in\{+1,-1\})\),  set\begin{equation}
\gamma_{c}(\Omega,\alpha,\nu ,\sigma)=\begin{cases}1 &  \Omega \text{ is \((\alpha,\nu)\)-Diophantine, } \sigma =\pm1 \\
\frac{1}{2} &  \Omega \text{ is near resonant (not  \((\alpha,\nu)\)-Diophantine),  } \sigma =+1  \\
\frac{2}{3} &  \Omega \text{ is near resonant  (not  \((\alpha,\nu)\)-Diophantine),  } \sigma =-1
\end{cases}
\end{equation}
 where we recall that \((\alpha,\nu)\)-Diophantine means  that  for all integers \((k_1,k_2)\) with  \(k_{2}\neq 0\), \begin{equation}
|\frac{\Omega}{2\pi}-\frac{k_{1}}{ k_{2}}|\geqslant\frac{\alpha}{|k_{2}|^{1+\nu}((\frac{k_{1}}{ k_{2}})^{2}+1)^{\nu/2}}.
\end{equation}
\begin{thm}\label{thm:stability} For any  \(\Omega,\alpha,\nu>0,k_s\ne0,\gamma\in(0,\gamma_{c}(\Omega,\alpha,\nu ,sign(k_{s}))\),   there exists $\epsilon_c=\epsilon_c(k_s,\gamma,\Omega,\alpha,\nu)>0$ such that for any  $\epsilon\in(0,\epsilon_c)$,   the return map (\ref{eq:pert-return}) has an invariant KAM circle which is at a distance     $k_s \epsilon^\gamma$  from the singularity line at  \(K=0\).
\end{thm}

\begin{proof}
Let  \(\tilde k_s=2\tau k_s, \tilde{\epsilon}=\epsilon \) for \(k_{s}>0\) and  \(\tilde k_s=k_{s},\tilde{\epsilon}=\epsilon^{1-\gamma/2} \) for $k_s<0$.
We show below that   when  the near integrable smooth twist map:
\begin{equation}\label{eq:zaslavsky}
\begin{cases}
\bar{I}=I+\tilde{\epsilon} \tilde{f}(\bar{\phi}) \\
\bar{\phi}=\phi+\Omega+I
\end{cases}
\end{equation}
with a \(C^{r }\) smooth \(2\pi\)-periodic forcing function  $\tilde{f}(\bar{\phi})=\tau f(\bar{\phi})$   has KAM rotational invariant curves that are at a distance     $\tilde k_s \tilde \epsilon^{ \gamma}$  from \(I=0\) for all \( \gamma\in(0, \gamma_{c})\) and \(\tilde{\epsilon}<\tilde \epsilon_c(\tilde k_s,\Omega,\alpha,\nu)\), then the return map (\ref{eq:pert-return}) has also invariant KAM curves that are at a distance   $K=k_s \epsilon^\gamma$ from the singularity line (meaning that both the minimal and maximal distances of the curve from \(I=0\) are bounded from above and below by some constants time \(k_s \epsilon^\gamma\)).

By  standard KAM theory, the near identity smooth twist  map  (\ref{eq:zaslavsky}) has such invariant curves with \( \gamma\in(0,1) \) for Diophantine \(\Omega\) and  \( \gamma\in(0,\frac{1}{2}) \)  for non-Diophantine \(\Omega\); More precisely,
for \(\)a Diophantine  \(\Omega\) and \(\tilde \epsilon<\tilde \epsilon_{c}\),  there exist a positive measure  set \(\Lambda_{\epsilon}\) of  \(I_{\epsilon}\) values in \([-1,1]\),  so that the KAM curves \((I_{\Omega+I_\epsilon}^\epsilon(\phi),\phi)\) are invariant curves of the map with rotation numbers \(\Omega+I_\epsilon\), and there exists \(M=M(\Omega,\tilde \epsilon_{c})\) such that \(|I_{\Omega+I_\epsilon}^\epsilon(\phi)-I_{\epsilon}|\leqslant \tilde \epsilon M\). In particular, for all \({ \gamma}\in(0,1)\) and \(\tilde  k_{s}>0\), the upper set \([ \tilde k_{s}\tilde \epsilon^{ \gamma}+\tilde{\epsilon}M,2\tilde  k_{s}\tilde \epsilon^{ \gamma}+\tilde{\epsilon}M]\cap\Lambda_{\epsilon}\) and the lower set \([- \tilde k_{s}\tilde \epsilon^{ \gamma}-\tilde{\epsilon}M,-2\tilde  k_{s}\tilde \epsilon^{ \gamma}-\tilde{\epsilon}M]\cap\Lambda_{\epsilon}\) include many invariant curves, and the corresponding invariant curves remain      \(\pm\tilde  k_{s}\tilde \epsilon^{ \gamma}\) away from \(I=0\) and at a maximal distance of \(\pm2\tilde  k_{s}\tilde \epsilon^{\gamma}\pm\tilde{\epsilon}M\) from the  line \(I=0\).
For a non-Diophantine  \(\Omega\), averaging of the twist  map  (\ref{eq:zaslavsky})   brings it to its resonant normal form, and the resulting averaged map has rotational KAM tori surrounding a resonance band of oscillatory orbits of width \(C(\Omega)\sqrt{\tilde{\epsilon}}\) around the origin. Rotational tori that are at a distance larger than  \(\pm\tilde  k_{s}\tilde \epsilon^{ \gamma}\) with some  \( \gamma<\frac{1}{2}\)   from \(I=0\) belong to the Diophantine set of rotational invariant curves of the normal form, and thus, as above, there exists \(M_{R}=M_{R}(\Omega,\tilde \epsilon_{c},\tilde \gamma)\) such that \(|I_{\Omega+I_\epsilon}^\epsilon(\phi)-I_{\epsilon}|\leqslant \tilde \epsilon M_{R}\) for all  \(I_{\epsilon}\in[\pm\tilde  k_{s}\tilde \epsilon^{ \gamma}\pm\tilde{\epsilon}M_{R},\pm 2\tilde  k_{s}\tilde \epsilon^{ \gamma}\pm\tilde{\epsilon}M_{R}]\cap\Lambda_{\epsilon}\), namely, these tori remain at a distance  \( \tilde  k_{s}\tilde \epsilon^{ \gamma}\) from the origin.

Since the  return map (\ref{eq:pert-return}) is \(C^{r-2}\) smooth for positive \(K\) values, for positive \(k_{s}\) the above observations suffice to prove the results; Indeed, recall that for \(K>0\), the function \(G_s(K)\) of Eq. (\ref{eq:gsindef}) is \(C^{r}\) smooth and vanishes at the origin. Thus, rescaling by \(\frac{1}{\tau}\) and shifting  \(K\) by \(O(\epsilon)\) terms brings the map of Eq. (\ref{eq:pert-return})\  to a form which is \(C^{r-2}\) -close to the truncated map  (\ref{eq:zaslavsky}). Since \(r>6\),  if the truncated map has invariant curves so does this map, and, moreover, the invariant curves of these maps remain \(\mathcal{O}_{C^{r-2}}(\epsilon I,\epsilon^{2})\)-close to each other. In particular, for sufficiently small \(\epsilon\), for any \(\gamma>0\), the shifted rescaled map  (\ref{eq:pert-return}) has invariant curves with \(I\) values that remain at a distance \(2\tau k_s \epsilon^\gamma\) away from  \(I=0\), thus, provided \(\gamma<1\), for  sufficiently small \(\epsilon\),  \(K\) remains  at a distance of at least \( k_s \epsilon^\gamma\) from the singularity line (so indeed all iterations remain in the smooth regime).

For negative \(k_{s}\), let
$K(k)=k_s \epsilon^\gamma(1+\frac{\epsilon^{\beta}}{\tau k_{s}}k)$ so \(k=\epsilon^{-(\gamma+\beta)}\tau(K-k_s \epsilon^\gamma)\) with \(0<\gamma<1,\beta>0\). We first derive the map for \((k,\phi)\) and establish  that for \(k_{s}<0\) the resulting truncated map is \(C^{r-2}\) close to the full map, so that if the truncated map has KAM tori so does the full map. We then rescale the truncated map to obtain the required form (\ref{eq:zaslavsky}) with  $\tilde{\epsilon}=\epsilon^{1-\gamma/2}$, and the existence of KAM curves of  (\ref{eq:zaslavsky}) for small shifted \(I\) values completes the proof.
In the resonant case, the requirement that the shifted \(I\) value is below the \(\Omega\) resonance band leads to the additional restriction on \(\gamma\).

Plugging  \(K(k)\) in   Eq. (\ref{eq:pert-return})\  leads to:
 \begin{equation}
\bar{k}=k+\epsilon^{1-(\gamma+\beta)}\tau f(\bar{\phi})+\epsilon^{-(\gamma+\beta)}\tau \breve g(\epsilon G_{s}(K(k)))
\end{equation}where \( \breve g\) denotes a \(C^{r-2}\) function  that  vanishes at the origin.  We next show that the last term is \(C^{r-2}\)-small (as a function of \(k\)). Recall that  \( G_{s}(K)\) (see Eq. (\ref{eq:gsindef})) has a square root singularity for negative \(K\), so \(\epsilon^{-(\gamma+\beta)}\tau \breve g(\epsilon G_{s}(K(k)))=O(\epsilon^{1-(\gamma+\beta)} \sqrt{-K(k)})=O(\epsilon^{1-\frac{1}{2}\gamma-\beta})\), and is indeed smaller by a factor of  \(\epsilon^{\frac{1}{2}\gamma}\) from the leading order term. Moreover, while   \(\frac{d^{j}}{dK^{j}}G_s(K)\) at \(K=K(k)\) is large, of order \(\epsilon^{\gamma(\frac{1}{2}-j)}\), since  \(K'(k)=\frac{\epsilon^{\gamma+\beta}}{\tau }\), its derivatives w.r.t. \(k\) is small, of order \(\epsilon^{\frac{1}{2}\gamma+ j\beta}\), and  the \(j\)th order derivatives of \(\epsilon^{-(\gamma+\beta)}\tau \breve g(\epsilon G_{s}(K(k)))\) is, for  \(1\leqslant j\leqslant r-2\),  small,  of order \(\epsilon^{j-\frac{1}{2}\gamma+ (j-1)\beta}\). Namely, we established  \begin{equation}
\bar{k}=k+\epsilon^{1-(\gamma+\beta)}\tau f(\bar{\phi})+O_{C^{r-2}}(\epsilon^{1-\frac{1}{2}\gamma-\beta}).
\end{equation}

Next, to find the leading order behavior for \(\bar{\phi}\),  we expand Eq.  (\ref{eq:gsindef})
  for negative \(k_{s}\) and \(k=O(1)\):
\begin{equation}
\begin{split}
G_s(k_s \epsilon^\gamma(1+\frac{\epsilon^{\beta}}{\tau k_{s}} k))
&=-\epsilon^{\gamma/2}
\frac{\sqrt{|k_s| }(2\omega_2(I_{tan}(E)))^{3/2}}{|V_1'(q_1^w)|}\sqrt{1+\frac{\epsilon^{\beta}}{\tau k_{s}}k} (1+O_{C^{r-2}}(K))\\
 &+ \tau k_s \epsilon^\gamma(1+\frac{\epsilon^{\beta}}{\tau k_{s}} k)+O_{C^{r-2}}(K^{2})\\
&=-\epsilon^{\gamma/2}
\frac{\sqrt{|k_s| }(2\omega_2(I_{tan}(E)))^{3/2}}{|V_1'(q_1^w)|}\cdot
(1+\frac{1}{2}\frac{\epsilon^{\beta}}{\tau k_{s}}k +O_{C^{r-2}} (\epsilon^{2\beta},\epsilon^{\gamma},k\epsilon^{\gamma+\beta}))
\\
&+ \tau k_s \epsilon^\gamma(1+\frac{\epsilon^{\beta}}{\tau k_{s}} k)+O_{C^{r-2}}(\epsilon^{2\gamma},k\epsilon^{2\gamma+\beta})\\
&=-\epsilon^{\gamma/2}\frac{\sqrt{|k_s| }(2\omega_2(I_{tan}(E)))^{3/2}}{|V_1'(q_1^w)|}
+\tau k_s \epsilon^\gamma+O_{C^{r-2}}(\epsilon^{3\gamma/2}) \\
& +k(\epsilon^{\gamma/2+\beta}
\frac{ \sqrt{|k_s| }(2\omega_2(I_{tan}(E)))^{3/2}}{-2\tau k_{s}|V_1'(q_1^w)|}
+ \epsilon^{\gamma+\beta}+
 O_{C^{r-2}}( \epsilon^{ 3\gamma /2 +\beta,2\gamma+\beta }))
\end{split}
\end{equation}
  Hence, we get  \begin{equation}
\bar{\phi}=\phi+\hat \Omega_{\epsilon}+k(A\epsilon^{\gamma/2+\beta} +\epsilon^{(\gamma+\beta)}+O_{C^{r-2}}(\epsilon^{3\gamma/2}))+O_{C^{r-2}}(\epsilon^{3\gamma/2})
\end{equation}where \(\Delta\Omega_\epsilon:=\hat \Omega_{\epsilon}-\Omega=-\epsilon^{\gamma/2}\frac{\sqrt{|k_s| }(2\omega_2(I_{tan}(E)))^{3/2}}{|V_1'(q_1^w)|}+O(\epsilon^{\gamma})\). Summarizing, the  map near \(k_s \epsilon^\gamma\) is of the form:
\begin{equation}
\label{eq:kexpandmap}\begin{cases}
\bar{k}=k+\epsilon^{1-\gamma-\beta}\tau f(\bar{\phi})+O_{C^{r-2}}(\epsilon^{1-\frac{1}{2}\gamma-\beta}) \\
\bar{\phi}=\phi+\hat \Omega_{\epsilon}+ kA\epsilon^{\gamma/2+\beta}+O_{C^{r-2}}(\epsilon^{\frac{3}{2}\gamma},k\epsilon^{\gamma+\beta})
\end{cases}
\end{equation}
where \(A=\frac{(2\omega_2(I_{tan}(E)))^{3/2}}{2\tau \sqrt{|k_s|}|V_1'(q_1^w)|}\). Hence, if the truncated map (the above map without the \(O_{C^{r-2}}\) terms) has KAM tori, then the  map (\ref{eq:kexpandmap}) has KAM curves which are \(O_{C^{r-2}}(\epsilon^{\min(1-\frac{1}{2}\gamma-\beta,\frac{3}{2}\gamma,\gamma+\beta)})\) close to them.
   Rescaling the truncated map by setting \(J=kA\epsilon^{\gamma/2+\beta}\)
and shifting \(J\) to  \(I=J+\Delta\Omega_\epsilon\) we obtain the map (\ref{eq:zaslavsky}) with $\tilde{\epsilon}=\epsilon^{1-\frac{\gamma}{2}}$ (namely  $\epsilon=\tilde{\epsilon}^{\frac{2}{2-\gamma}}$, and      \(\Delta\Omega_{\epsilon}=O(\epsilon^{\frac{\gamma}{2}})=O(\tilde{\epsilon}^{\frac{\gamma}{2-\gamma}})\)).

For the Diophantine \(\Omega\) case, for sufficiently small \(\tilde{\epsilon}\),    (\ref{eq:zaslavsky}) has a positive measure set of rotational  KAM curves. In particular, taking   \(\gamma\in(\frac{1}{2},1)\), we obtain  \(\Delta\Omega_{\epsilon}<C\sqrt{\tilde{\epsilon}}\), so for sufficiently small \(\tilde \epsilon\) there is a positive measure set of rotational  KAM tori near  \(I_{0}=\Delta\Omega_{\epsilon}\),  and the corresponding shifted \(J\) values satisfy \(|J_{KAM}(\phi,J_{0})-J_{0}|<C_{1}\tilde{\epsilon}=C_{1}\epsilon^{1-\gamma/2}\). Taking  \(J_{0}\) values of order \(\epsilon^{\gamma/2+\beta}=\tilde{\epsilon}^{\frac{\gamma+2\beta}{2-\gamma}}\ll|\Delta\Omega_{\epsilon}\)\textbar, the corresponding \(k\) values satisfy \(|k^{trun}_{KAM}(\phi,k_{0})-k_{0}|<C_{2}\epsilon^{1-\gamma-\beta}\) for the truncated map. The KAM curves of  (\ref{eq:kexpandmap}), \(k_{KAM}(\phi,k_{0})\),  are  \(O_{C^{r-2}}(\epsilon^{\min(1-\frac{1}{2}\gamma-\beta,\frac{3}{2}\gamma,\gamma+\beta)})\)   -close to the invariant curves of the truncated map (even though they may be supported on a different set of \(k_{0}\) values), so  \(|\epsilon ^{\beta}k_{KAM}(\phi,k_{0})-\epsilon ^{\beta}k_{0}|<C_{2}\epsilon^{\min(1-\gamma,\frac{3}{2}\gamma+\beta,\gamma+2\beta)}\ll1\) (recall that  \(\gamma<1,\beta>0\)), hence, there is a positive measure set of order one \(k_{0}\) values such that \(K_{KAM}(\phi,k_{0})=k_s \epsilon^\gamma(1+\epsilon ^{\beta}k_{KAM}(\phi,k_{0}))\)    are invariant curves which are strictly below the singularity line: \begin{equation}\label{eq:KAMks}
|K_{KAM}(\phi,k_{0};k_{s},\gamma,\beta,\Omega_{D})-k_s \epsilon^\gamma|<k_s \epsilon ^{\gamma+\beta}(k_{0}+\epsilon^{\min(1-\frac{1}{2}\gamma-\beta,\frac{3}{2}\gamma,\gamma+\beta)}).
\end{equation}
We conclude that these curves  correspond to rotational KAM curves of  (\ref{eq:pert-return}) that bound the tangency band from below at a distance \(k_s \epsilon^\gamma\) from the singularity line for \(\gamma\) arbitrary close to the critical exponent \(\gamma_{c}(\Omega_{Diophantine},\alpha,\nu ,-1)=1\), which also implies that other rotational invariant curves with smaller \(\gamma\) values remain below that tangency band.

For the resonant \(\Omega\) case,  for  \(\gamma<\frac{2}{3}=\gamma_{c}(\Omega_{Resonance},\alpha,\nu ,-1)\),
  the shift in the action,     \(|\Delta\Omega_{\epsilon}|=O(\tilde{\epsilon}^{\frac{\gamma}{2-\gamma}})>C\sqrt{\tilde \epsilon}\), and thus, for sufficiently small \(\tilde \epsilon\) there is again a positive measure set of rotational  KAM tori near  \(I_{0}=\Delta\Omega_{\epsilon}\). Following exactly the same steps as in the Diophantine case with negative \(k_{s}\),  we conclude that the map  (\ref{eq:pert-return}) has invariant curves of the form (\ref{eq:KAMks}) for all \(\gamma<\frac{2}{3}\).
\end{proof}

Theorem \ref{thm:stability} implies that the  tangency band boundaries, \(I^\epsilon_{\pm}(\theta)=K^\epsilon_\pm(\theta)+I_{tan}^{\epsilon}(\theta)\) (see Def. \ref{def:tangencyband}) are at most a distance      $k_s \epsilon^\gamma$  from \(I_{tan}^{\epsilon}(\theta)\) which is order \(\epsilon\) close to  \(I_{tan}\), thus implies Theorem \ref{thm:mainthm3}. From the form of the map (\ref{eq:pert-return}) it is clear that for $K=o(\epsilon)$ small or even $K=0$, the next iteration in the return map is of $\mathcal{O}(\epsilon)$, and hence a lower bound on the width of the tangency band is  $\mathcal{O}(\epsilon)$.

\section{The truncated tangency map\label{sec:truncatedmap} }
To demonstrate the stability results of Theorem \ref{thm:stability} and to support claims \ref{conj:bandwidth} -\ref{conj:resislands}  we  derive first some basic properties of the truncated tangency map, \(F_{TTM}\)  of Eq. (\ref{eq:truncatedreturn}). This map corresponds to taking  the leading order terms of \(G_{s}(K)\)  in the   map (\ref{eq:pert-return}). Notice that in (\ref{eq:pert-return})  \(\alpha>0\) and, by rescaling \(K\), we can take \(\tau\in\{\pm1\}\); rescaling \(K\) by \(|\tau(I_{tan}(E))|\)   shows that the parameters \((\alpha,\epsilon,\tau)\) in  (\ref{eq:truncatedreturn})  are related to those of (\ref{eq:pert-return}) by : \(\alpha\leftarrow \frac{(2\omega_2(I_{tan}(E)))^{3/2}}{\sqrt{|\tau(I_{tan}(E))|}|V_1'(q_1^w)|}>0\),  \( \epsilon\leftarrow\epsilon|\tau(I_{tan}(E))|,\tau\leftarrow\frac{\tau(I_{tan}(E))}{|\tau(I_{tan}(E))|}\ \in\{\pm1\}\) and \(K\leftarrow|\tau(I_{tan}(E))| K\) (recall that we study here the behavior near a regular torus which is tangent to the wall, so \(V_1'(q_1^w),\omega_2(I_{tan}(E)),\tau(I_{tan}(E))\neq0\)).  The symmetry property of Lemma \ref{lem:symmetry} shows that $f(\bar{\phi})=\bar{I}_{tan}^{\epsilon}(\bar{\phi})-I_{tan}^{\epsilon}(\bar{\phi})$ is an odd function of $\bar{\phi}$.
We first establish a few properties of  \(F_{TTM}\)  and then, taking the first non-trivial Fourier mode of \(f\),  we set  \( f({\phi})=\sin \phi\) and  study numerically  the resulting map, establishing claims  \ref{conj:bandwidth} -\ref{conj:nonresconnectingorbit}  for this specific realization.
\subsection{Twist and fixed points of the truncated tangency map \label{sec:basicprop}}

The twist of the unperturbed truncated tangency map   (Eq. (\ref{eq:truncatedreturn}) at \(\epsilon=0\)) is piecewise smooth:
\begin{equation}
\frac{\partial\bar{\phi}}{\partial K}=\begin{cases}
\tau & K> 0 \\
\tau+\frac{\alpha}{2\sqrt{-K}} & K<0
\end{cases}
\end{equation}
Hence, for $\tau>0$ the map  \(F_{TTM}\)  is a piecewise smooth twist map: it has a strictly positive twist for all values of $K$. For $\tau<0$, the twist changes sign and a unique non-twist circle is created at \(K<0\) (on the impacting side):
\begin{equation}
\tau+\frac{\alpha}{2\sqrt{-K_{NT}}}=0\Rightarrow K_{NT}=-\frac{\alpha^2}{4\tau^2}.
\end{equation}For  \(K\) values which are bounded away from \(0, K_{NT}\) the twist is bounded away from \(0\). For both positive and negative shear, the twist diverges when the singularity line is approached from below, at \(K\rightarrow0^{-}\). The stability results of section \ref{sec:stable} utilize this twist property for  circles which are bounded away from the singularity set and twist-less circle.

The limits of small and large  \(\alpha\)  are interesting and deserve a separate study -  recall that \(\alpha= \frac{(2\omega_2(I_{tan}(E)))^{3/2}}{\sqrt{|\tau(I_{tan}(E))|}|V_1'(q_1^w)|}\), so large \(\alpha\) appears when \(\sqrt{|\tau(I_{tan}(E))|}|V_1'(q_1^w)|\) is small - either when the wall is approaching an elliptic fixed point and then   \(\alpha\)  is large for all energies, or, when the energy is such that the tangent torus is close to a twist-less torus of the smooth system.

The fixed points of the map, as listed in Appendix B  and summarized in  the bifurcation diagram   in the \((\Omega,K)\) plane presented in Figure \ref{fig:bifomega}, appear, for each zero \(\phi^{*}\) of \(f(\phi)\), as a countable set of continuous  branches that emanate from \(K=0, \Omega=2\pi j,\:\ j\in\mathbb{Z}\). For positive shear these branches are monotone deceasing in \(\Omega\). They correspond to unstable fixed points when   \(f'(\phi^{*})>0\). When     \(f'(\phi^{*})<0\) they correspond to centers, with the exception of order \(\epsilon\) interval of negative \(K\) values at which the fixed point is a saddle with negative multipliers. For negative shear, each branch includes a  saddle node bifurcation at  \(K=K_{NT}\), so there is a range of \(\Omega\) values at which each branch contributes three fixed points, two of which with nearby \(K\) values.
\begin{figure}[!htp]
\begin{centering}
\includegraphics[scale=0.25]{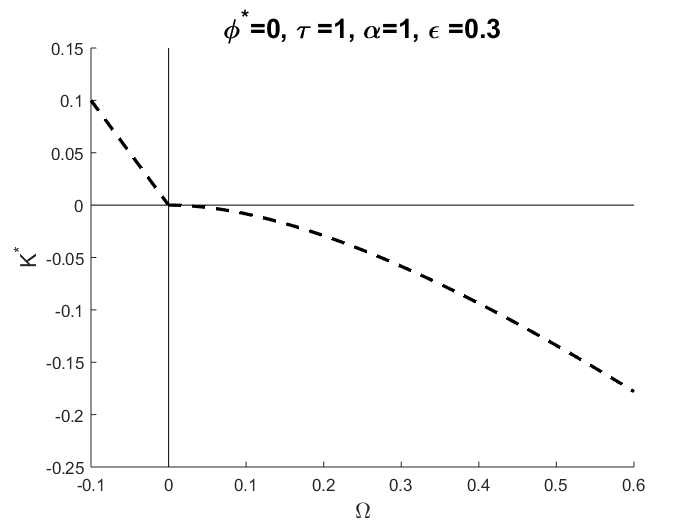}
\includegraphics[scale=0.25]{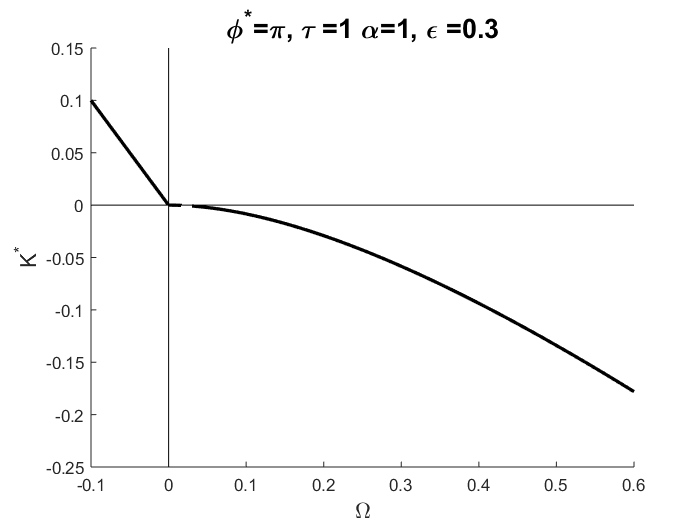}\\
\includegraphics[scale=0.25]{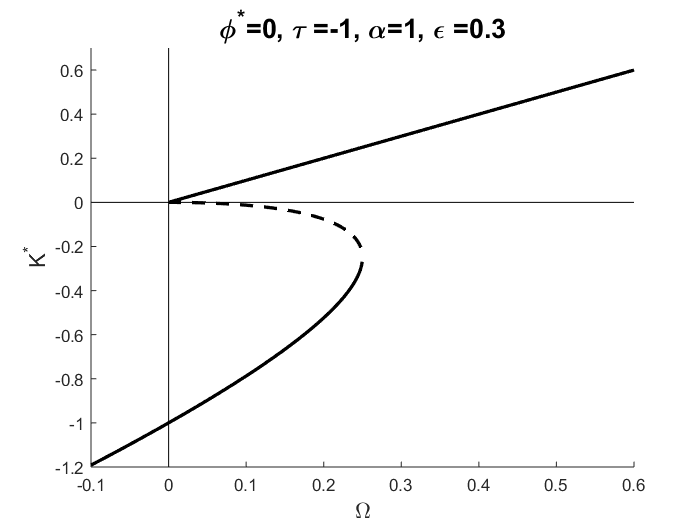}
\includegraphics[scale=0.25]{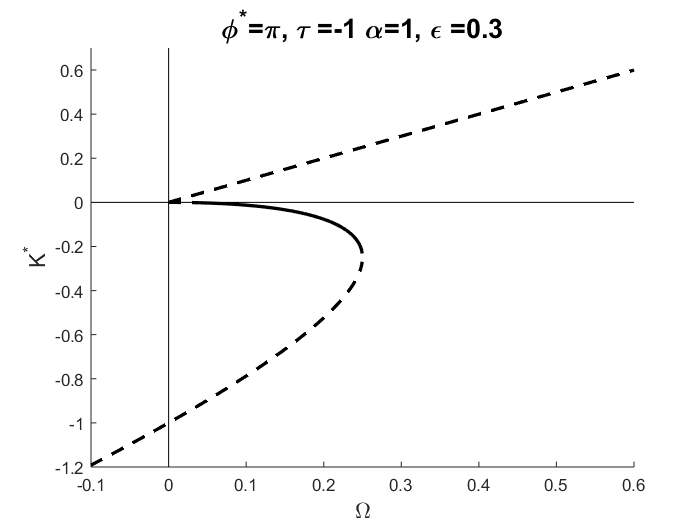}\\
\includegraphics[scale=0.25]{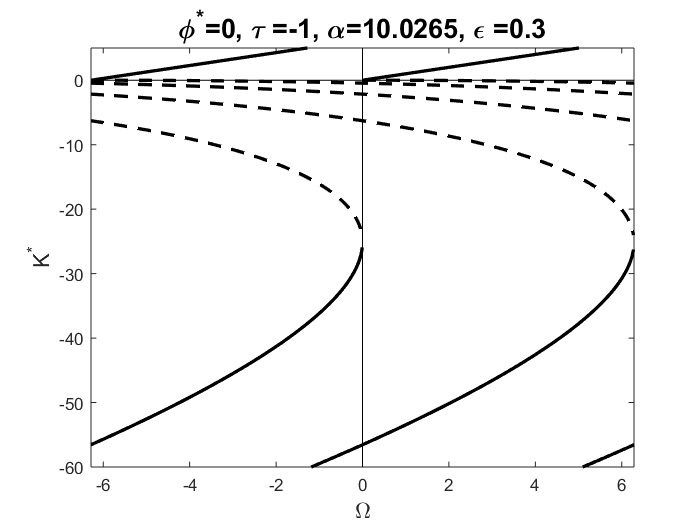}
\includegraphics[scale=0.25]{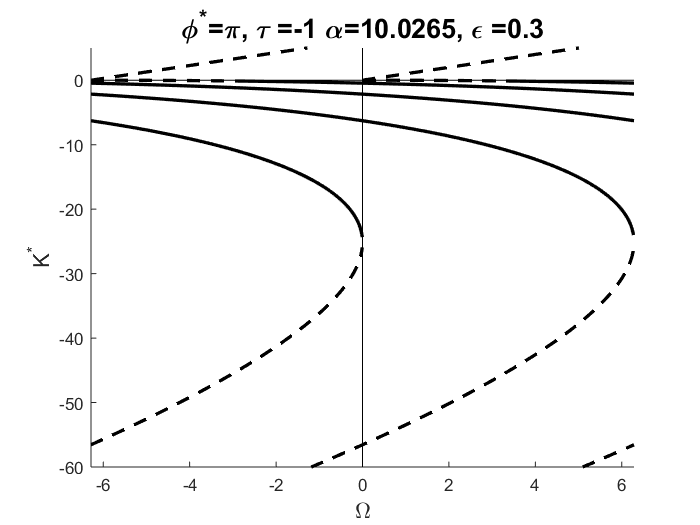}
\par\end{centering}
\protect\caption{\label{fig:bifomega} Bifurcation diagrams  for the  map (\ref{eq:truncatedreturn}).  First column  \(f'(\phi^{*})>0\), second column  \(f'(\phi^{*})<0\). First two rows  $ \alpha=1,$ last row  $ \alpha=2\sqrt{8\pi}$. First row \(\tau=1\), second and third rows \(\tau=-1\). In all figures  $\epsilon=0.3$.   Notice that the branches of centers at negative \(K^{* }\) values undergo period-doubling bifurcation at small \(\Omega-2\pi j\)  values (second column).}
\end{figure}

When the fixed points spacing in \(K\) becomes small, we expect chaotization of the map  by three main mechanisms: overlap of resonance regions around stable fixed points,  heteroclinic tangles of nearby saddles and sequences of period doubling bifurcations.  Such proximity occurs in two situations: first, for \(\tau=-1\), at  \(\Omega=2\pi j+\delta\), the positive and negative  branches of the fixed points are \(\delta\) close (a singular saddle-center bifurcation, the positive branch is at  \(K=\delta\) whereas the negative branch is close to
\(-\frac{\delta^{2}}{\alpha^2}\)), so near the bifurcation point complex dynamics is always observed. Second,    in the large \(\alpha\) limit the negative branches are close to each other and to the singularity line, and undergo, as \(\epsilon\) is increased, sequences of period doubling bifurcations. For    \(\alpha\gg\sqrt{8\pi}\approx5.01\) and small \(l\ll \left\lfloor\frac{\alpha^{2}}{8\pi}\right\rfloor\), we observe that the spacing between the first \(l\) branches is of order \(\frac{1}{\alpha^{2}}\). Overlap between \(\sqrt{\epsilon}\) resonances around stable points is thus expected when  \(\epsilon> \frac{C}{\alpha^{4}}\).  Moreover, these branches become  unstable  when \(\epsilon>C\frac{2\pi l+\Omega\text{ mod }2\pi}{\alpha^2}\), for some constant \(C\), so for such values of \(\epsilon\) sequences of period doubling bifurcations occur close to the singularity line. These observations may explain the fast mixing which is observed at high \(\alpha\) values (see Figures  .

\subsection{\label{sec:tangwidth} The dynamics in the tangency band  }

 Figures \ref{fig:stability1} demonstrates  the main claim of Theorem \ref{thm:stability}:  for small \(\epsilon\), KAM curves with strictly positive \(K\) values and KAM curves with strictly negative \(K \) values are observed for both Diophantine and resonant rotation numbers \(\Omega\). The tangency band, \(\mathcal{B}_{\epsilon}\) is bounded in between these curves. On the other hand, as is demonstrated in Figures \ref{fig:bandwidthres}-\ref{fig:tangency-res1-alphadependence},  which show zoom-ins and longer trajectory segments in the tangency band, trajectories that pass close the singularity line do not lie on rotational invariant curves.

\noindent\textbf{Tangency band width (Claim \ref{conj:bandwidth}):}  Figure \ref{fig:bandwidthres} shows the dependence of the maximal upper and lower widths of the tangency band on \(\epsilon\). We iterate  100 tangent initial conditions   for \(10^{5}\) iterates,  plotting, for each \(\epsilon\),   $\max(K):=\max_{\phi_{0},j}{K_{j}} \leqslant\ W^{+},\min(K)=\min_{{\phi_{0},j}}{K_j}\geqslant -W^{-}$ where \(j=1,...,10^{5}\),  \(K_{0}=0\) and   $\phi_{0}\in[-\pi,\pi]$ is evenly spaced (in the next subsections we study the  dynamics and the width of individual trajectories in the band and provide a tighter numerical bound on the band widths). Figure \ref{fig:bandwidthres}a,b summarises the findings: for regular \(\Omega\), \(W^{\pm}\) are linear in \(\epsilon\) and are essentially symmetric. Figure \ref{fig:bandwidthres}c shows that for $\Omega= 2\pi$ the \(\sqrt{\epsilon}\) scaling of \(W^{+}\) and the \(\epsilon^{2/3}\) of \(W^{-}\) suggested  from Theorem \ref{thm:stability} are realized. On the other hand, Figure \ref{fig:bandwidthres}a,b demonstrate that the resonant $\Omega$ values which are not strongly resonant produce \(W^{\pm}\)  which appear to be linear in  \(\epsilon\), with widths close to those  of the nearby  irrational \(\Omega\).\\ The above scalings are valid for sufficiently small \(\epsilon\): as $\epsilon$ is increased,  the tangency band overlaps with other resonances, reminiscent of Chirikov's resonance-overlap phenomenon, and the  dependence on $\epsilon$ changes: see the width curve in Figure \ref{fig:bandwidthres}a  for the case \(\Omega=\pi/3,\tau=1\)   and Figure \ref{fig:bandwidthres}d,e,f showing the corresponding phase space plots. Similarly, the width curve of Figure \ref{fig:bandwidthres}b for  \(\Omega=2\pi,\tau=-1\) indicates  an overlap transition around \(\epsilon=0.06\) due to the growth of the 1-1 resonance band centered at \(K=-1\).

\begin{figure}[hp]
\begin{centering}
% this is from phd11
\includegraphics[scale=0.2]{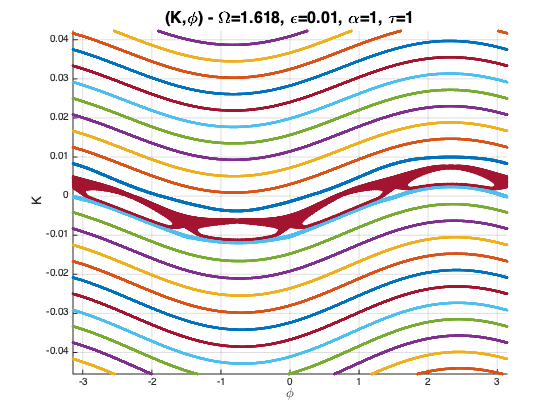}
\includegraphics[scale=0.2]{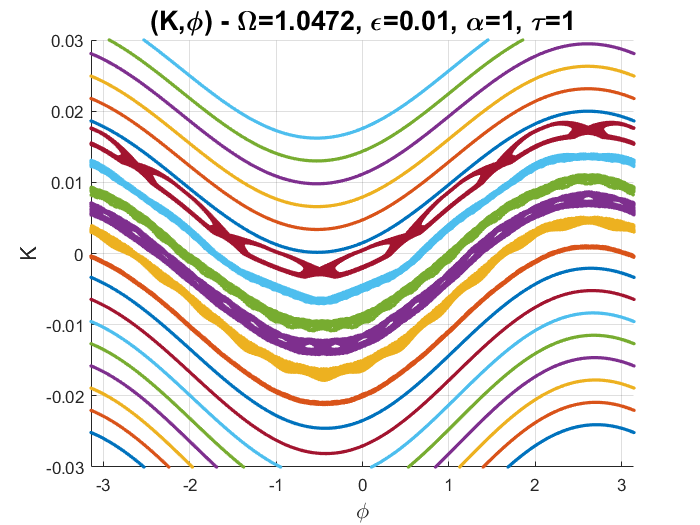}
% this is from phd39
\includegraphics[scale=0.2]{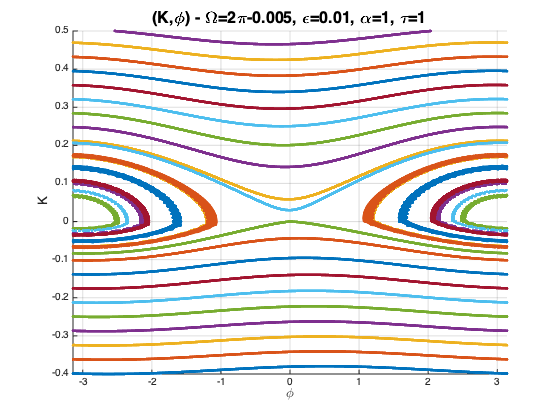} \\
% this is from phd41
\includegraphics[scale=0.2]{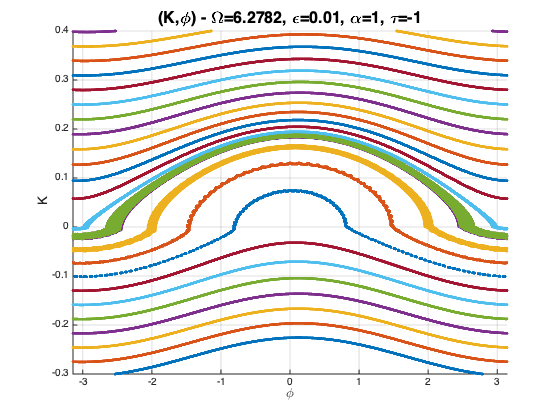}
% from
\includegraphics[scale=0.25]{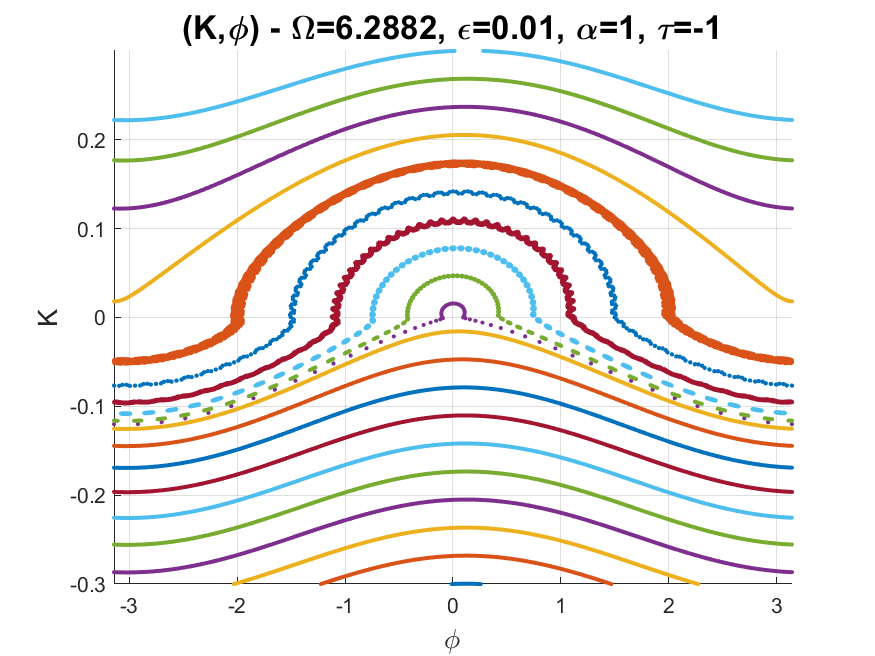}
\par\end{centering}
\protect\caption{\label{fig:stability1}The tangency band for Diophantine, resonant and strong resonant rotation numbers. (a) Diophantine case: $\Omega=\frac{\sqrt{5}+1}{2}\approx1.618,\ \tau=1$ (b) Resonant case: $\Omega=\frac{\pi}{3}\approx1.047,\ \tau=1$ (c) Strong resonance case with positive shear: $\Omega=2\pi-0.005$, $\tau=1$ (d) Strong resonance case with negative shear, sub-resonance case: $\Omega=2\pi-0.005$, $\tau=-1$. (e) Strong resonance case with negative shear, super-resonance case: $\Omega=2\pi+0.005$, $\tau=-1$ (4 fixed points with small \(K\) values co-exist). For all simulations  $\epsilon=0.01$, $\alpha=1$ and $10^5$ iterations of $20$ i.c. starting at \(\phi_{0}=0\) and at equally spaced \(K \) values  are plotted. Notice that  in the strong resonance cases (c,d,e) the tangency band is at least ten fold larger than the other cases (a,b) and that its extent for positive \(K\) values is much larger than its extent for negative \(K\) values.}
\end{figure}

\begin{figure}[!htp]
\begin{centering}
(a)\includegraphics[scale=0.2]{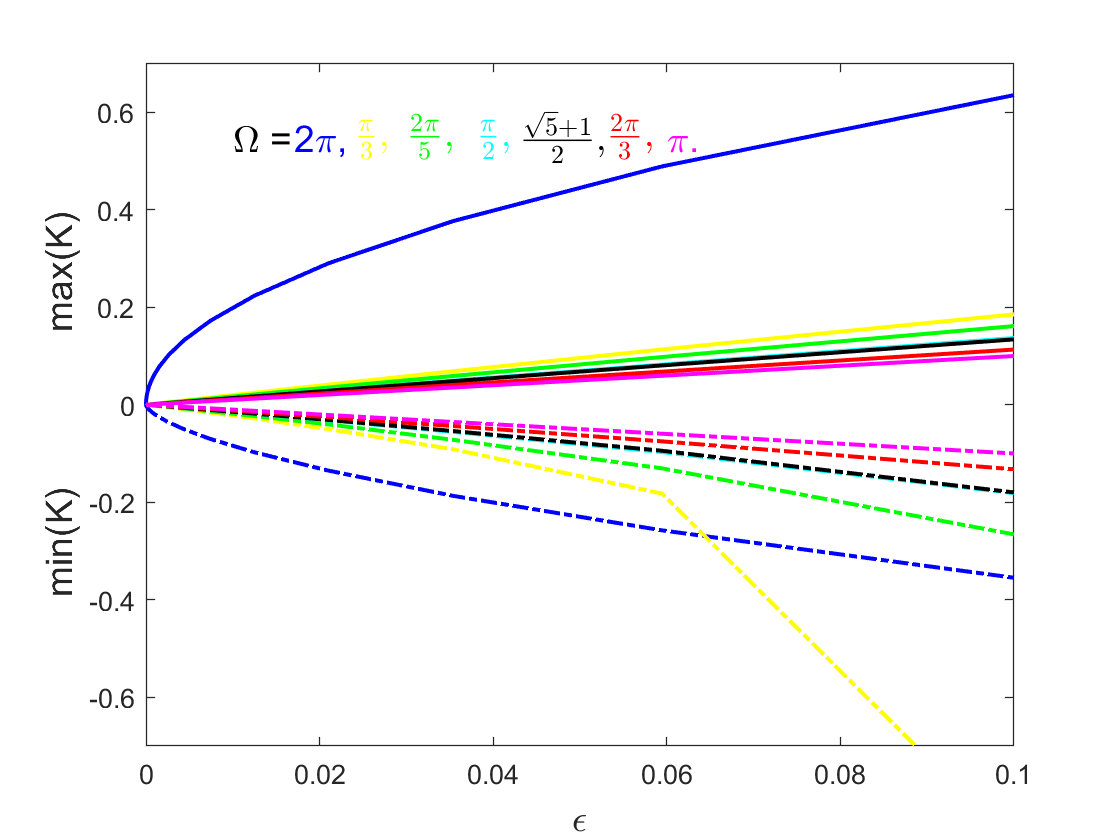}
(b)\includegraphics[scale=0.2]{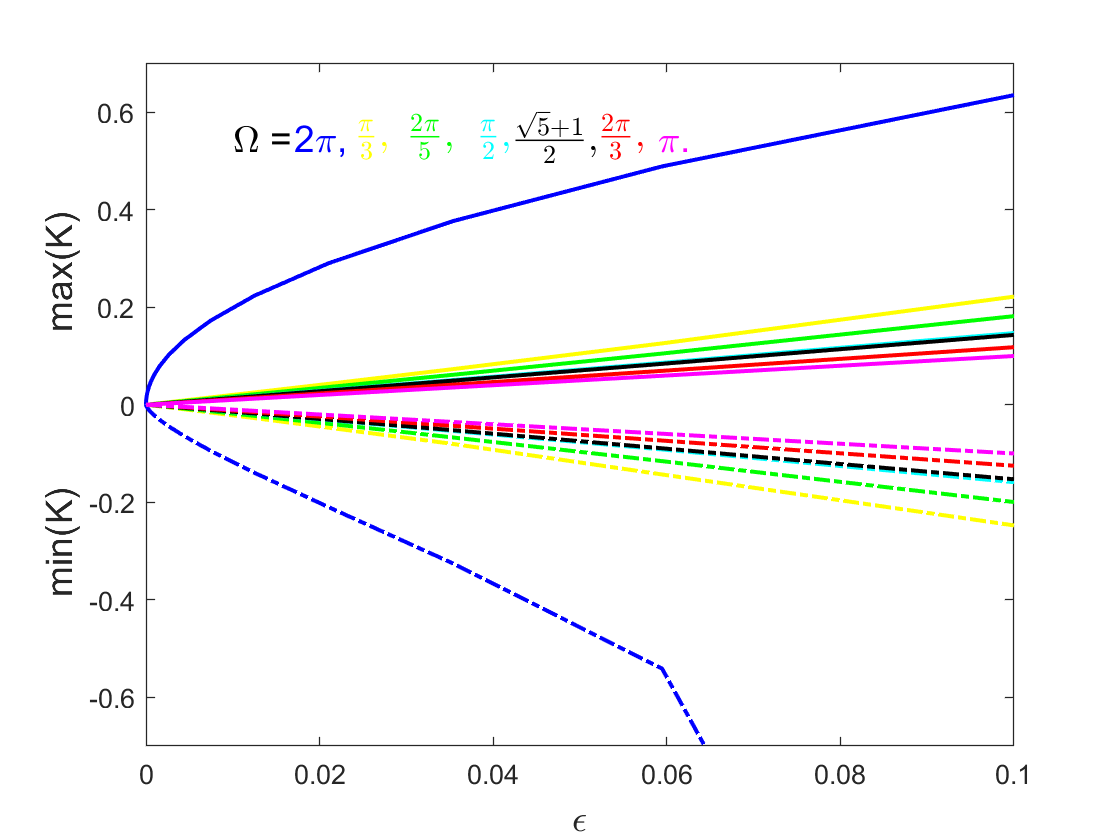}
(c) \includegraphics[scale=0.2]{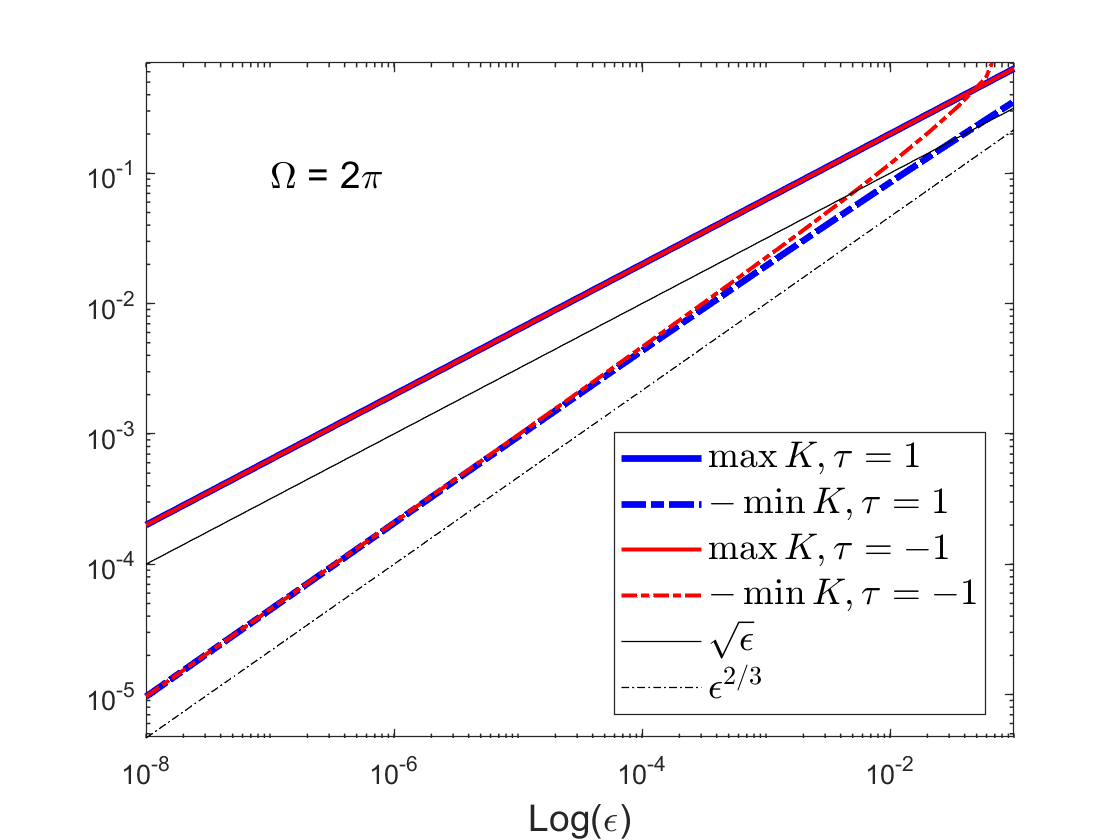}\\
(d)\includegraphics[scale=0.2]{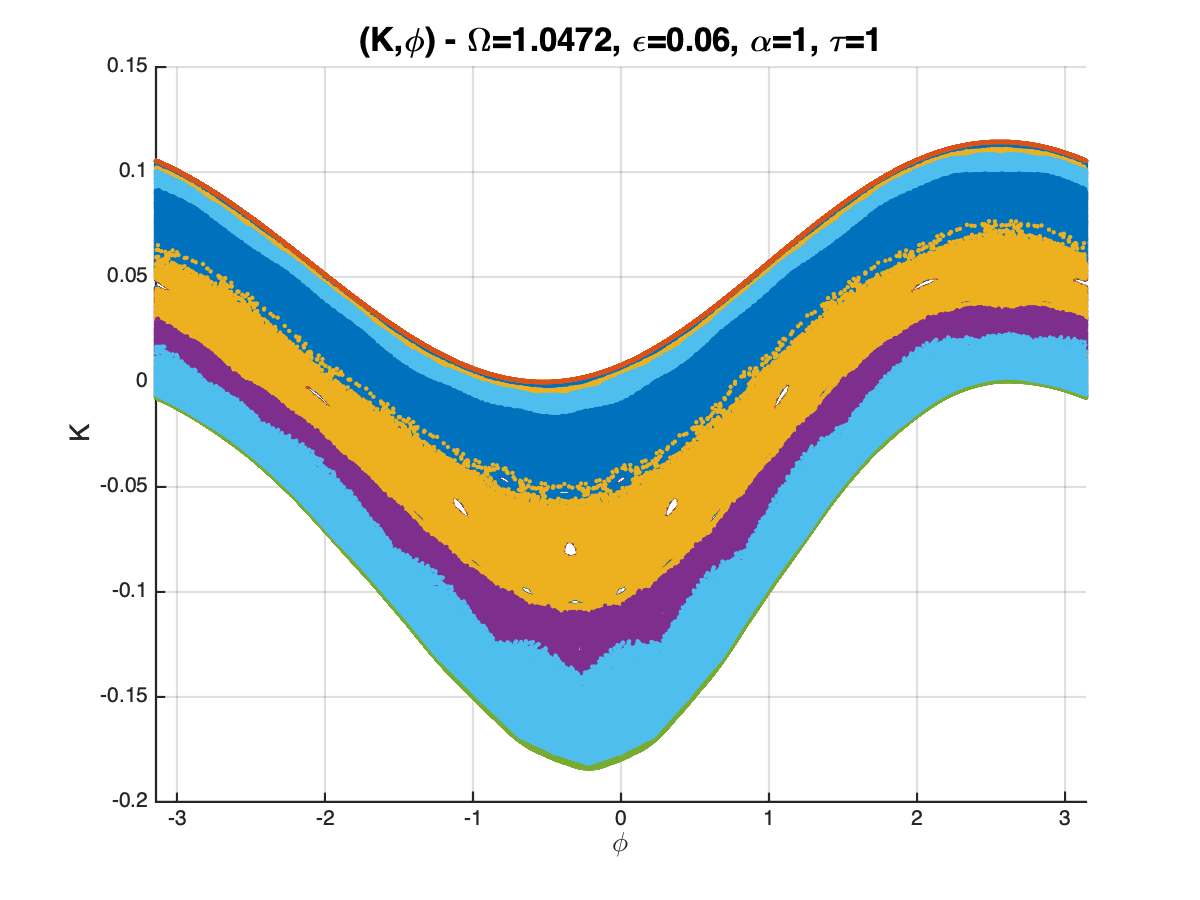}
(e)\includegraphics[scale=0.2]{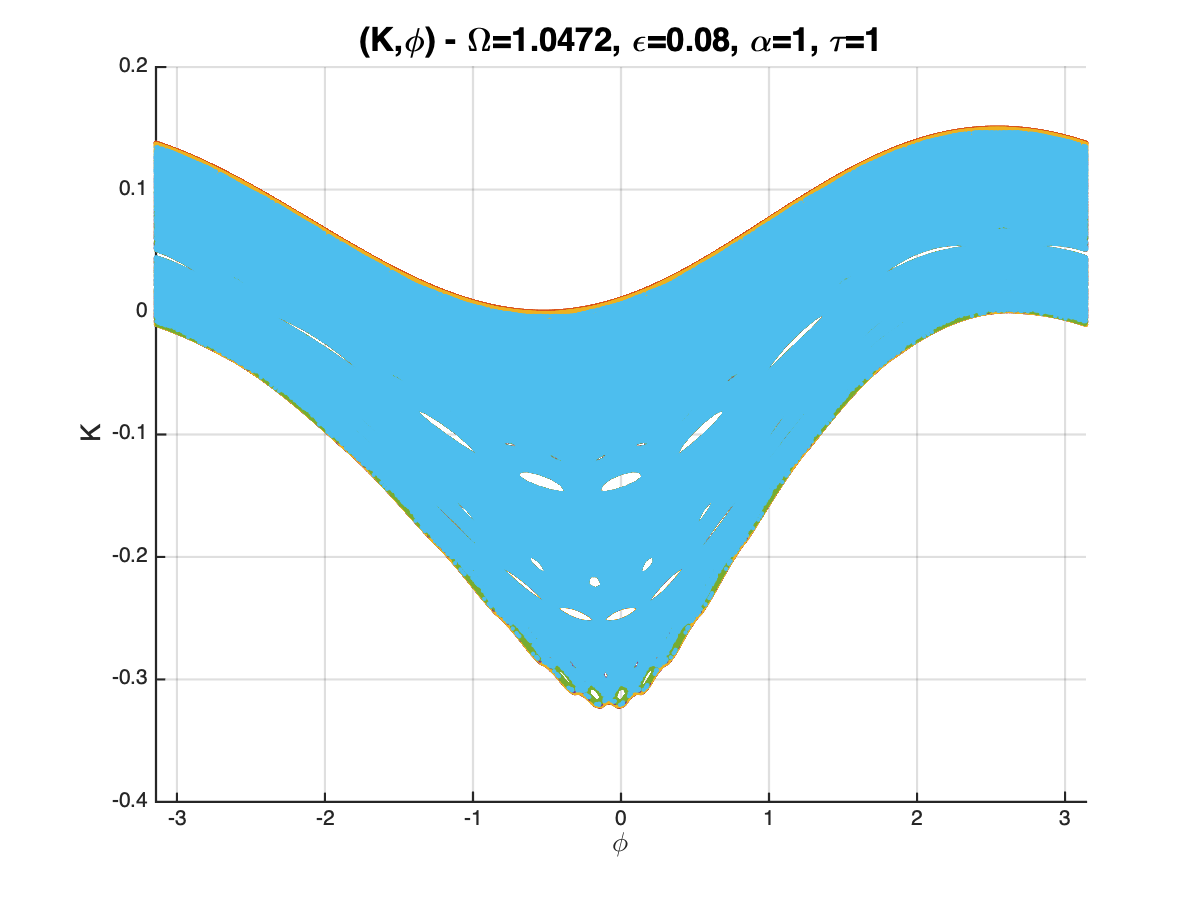}
(f)\includegraphics[scale=0.2]{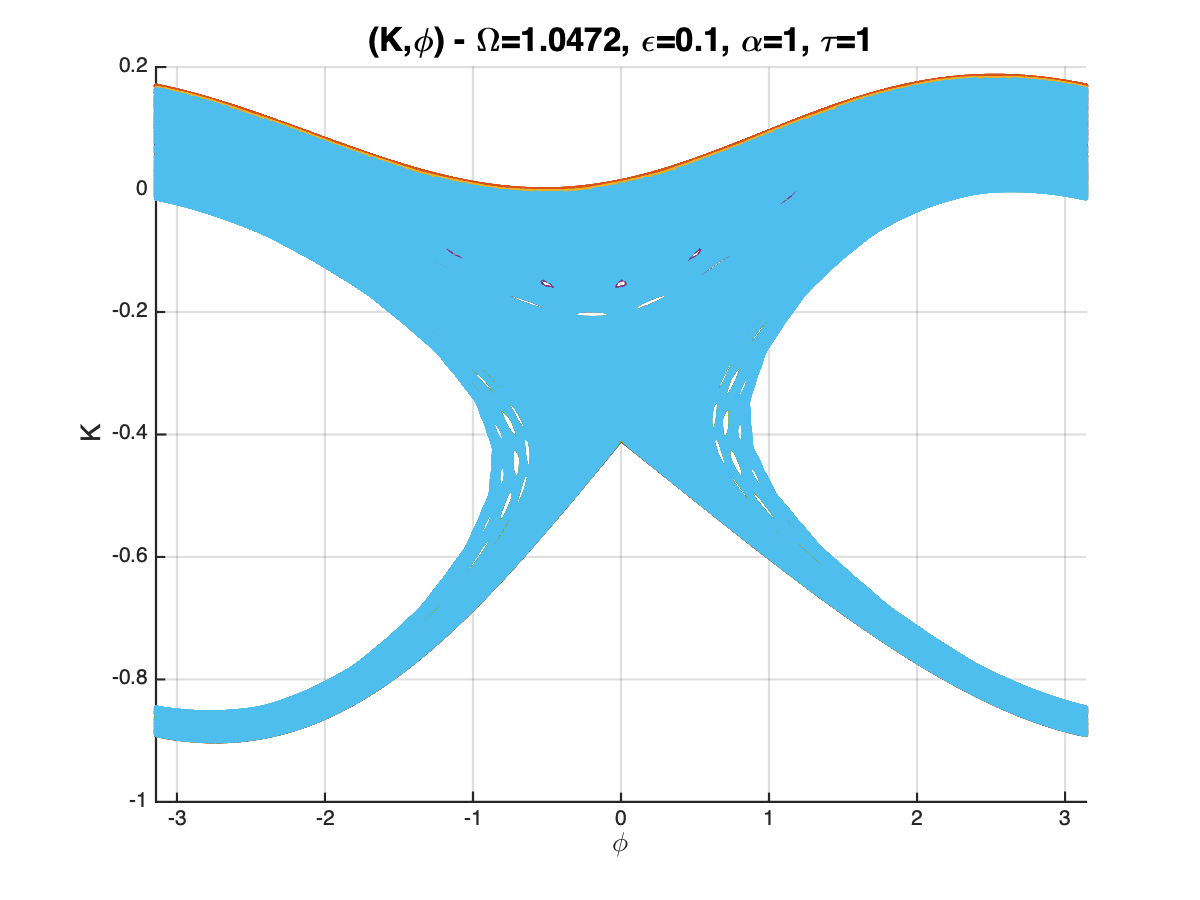} \\
\par\end{centering}
\protect\caption{\label{fig:bandwidthres}The  upper and lower tangency band widths.   (a) Positive shear, \(\tau =1\),(b) Negative shear, \(\tau =-1 \), (c) Strong resonance scaling:  \(\Omega=2\pi\),  \(\tau =\pm1\).  At regular \(\Omega\) values, for sufficiently small   \(\epsilon\), the upper and lower total widths are linear in \(\epsilon\),  are essentially symmetric, and  nearby resonant and Diophantine rotations behave similarly  (e.g. the cyan curve of \(\Omega=\pi/2\approx1.57 \)  and the black curve of  \(\Omega=(\sqrt{5}+1)/2\approx1.618\) almost overlap). At the strong resonance, the upper width exhibits square-root dependence on $\epsilon$ whereas the lower width scales as \(\epsilon^{2/3}\). Panels (d-f) demonstrate that the width jump in slope for    \(\Omega=\pi/3\approx1.0472\) around $\epsilon=0.06$ is a finite \(\epsilon\) effect associated with the overlap of the tangency band  and  the 1-1 resonance band centered around \(K=-0.4.\)  }
\end{figure}

\noindent
\textbf{Singularity band islands (Claim \ref{conj:resislands}):} stable, smooth quasi-periodic motion on a finite collection of invariant non-dividing circles   which do not intersect the singularity line and have  components with both positive and negative \(K\) values are shown in Figures  \ref{fig:tan11long}, \ref{fig:tangency1}(b), \ref{fig:tangency-res1}(c), \ref{fig:tangency-res1-zoom}(c--h).
These structures persist even in the long term simulations and correspond to the holes that appear in the singularity set, namely these belong to the set \(\mathcal{B}^{\epsilon}\setminus\mathcal{S}^\epsilon\). The  singularity-band islands are stable structures with no proper unperturbed limit, similar to the tangle islands that appear near tangent homoclinic bifurcations \cite{rom1999islands}.

\

\noindent\textbf{Long transients and connecting orbits  (Claim \ref{conj:nonresconnectingorbit}):}
 Figure \ref{fig:tan11long} shows the long term dynamics at small \(\epsilon\) value for a Diophantine frequency  of three tangent initial conditions (\(K_{0}=0\)). It demonstrates that the transient dynamics  may exhibit seemingly separate ergodic components (e.g. the blue trajectory appears, up to \(N\approx2\cdot10^{9}\) as a  very narrow invariant band) and  overlapping regions that do not coincide (e.g. the green and maroon trajectories have large overlapping regions, yet, they do not fully coincide up to  \(N\approx2\cdot10^{9}\)). Here, as simulation length grows, all the tangent trajectories seem to cover  uniformly the same region. Indeed, Figure  \ref{fig:tan11long}f shows the width of each tangent trajectory as a function of the iteration number:    \begin{equation}\label{eq:trajwidthdef}
W(n;\phi_0,\bigtriangleup):=\max_{i}\left(\max_{j=1,..n}K_{j}|_{\phi_{j}\in (\theta_i-\bigtriangleup,\theta_i+\bigtriangleup)} \ -\min _{j=1,..n}K_{j}|_{\phi_{j}\in (\theta_i-\bigtriangleup,\theta_i+\bigtriangleup)} \ \right)\end{equation}with windows of size   \(\bigtriangleup=0.005\) set at some random phases \(\theta_{i}\).  The total width of the tangency band is estimated as the maximal width of all trajectories in each of the chosen random windows at the simulation end: \[W\approx\max_{i}\left (\max_{j=1,..N;\phi_0}K_{j}|_{\phi_{j}\in (\theta_i-\bigtriangleup,\theta_i+\bigtriangleup)}-\min_{j=1,..N;\phi_0} K_{j}|_{\phi_{j}\in (\theta_i-\bigtriangleup,\theta_i+\bigtriangleup)}
\right) \] (dashed black line in Figure \ref{fig:tan11long}f). Observe that the width of the individual trajectories has long plateaus (on logarithmic scale) and sudden growth sprees until finally, after almost \(10^{10}\) iterates, it coincides with the band width. This suggests that here the singularity set,   \(\bar{\mathcal{S}^\epsilon}\), is ergodic and that the tangency band consists of  the union of this single ergodic component and  other invariant sets (like the singularity band islands) which do not include any tangent trajectory. In particular, this means that the singularity set has a connecting orbit which visits arbitrarily closely the upper and lower boundary of the tangency band.
Similar behavior is observed for other \(\Omega,\alpha\) values, even when \(\Omega\) is rational, as long as \(\Omega\)  is not strongly resonant.
\\ \\ The strongly resonant cases with positive shear  are shown in Figure \ref{fig:tan39long} ( $\Omega=2\pi-0.005$, $\tau=1$) and Figures \ref{fig:tan41long}-\ref{fig:tan41longzoom} depict the sub-strong resonance negative shear case ($\Omega=2\pi-0.005,\tau=-1$). In the positive shear case, due to the 1-1 resonance structure typical trajectories cover a circular band (see  Figure \ref{fig:tan39long}), so we refine the definition of a trajectory width by restricting the right hand side in Eq. (\ref{eq:trajwidthdef}) to non-negative \(K\).  \\ Figure \ref{fig:tan39long} suggests that the strong resonance positive shear singularity set, which covers most of the 1-1 resonance band, is composed of one ergodic component as well; even though  after \(2\cdot10^{10}\) iterations the resonance band is not uniformly covered, all trajectories have some overlapping regions and it is suspected that this is sufficient to establish the existence of connecting orbit. On the other hand,  Figures \ref{fig:tan41long}-\ref{fig:tan41longzoom} suggest that for some cases of negative shear and subcritical \(\Omega\), the tangency band has several ergodic components - the zoom-ins of Figure \ref{fig:tan41longzoom} show that  the borders between two nearby trajectories is sharp even  after \(2\cdot10^{10}\) iterations, suggesting (yet not proving..) that this division persists. %\ref{fig:tan40longzoom} ,

%\subsubsection{Diverging in \(\epsilon,\alpha\) time scales for  connecting orbits }
 Our findings on the dependence of the transient time on \(\epsilon\) and \(\alpha\) is presented in Figure \ref{fig:widthepsdep-long}. Since the width of the band and the slope of the limiting KAM tori are of order \(\epsilon\), to get a meaningful width definition in (\ref{eq:trajwidthdef}), the window size needs to be decreased as \(\epsilon\) is decreased, so we choose \(\bigtriangleup=\frac{\epsilon }{2}\). For a Diophantine \(\Omega\), the return time to a window of size \(2\bigtriangleup\) is approximately \(N_\bigtriangleup=\frac{2\pi}{2\bigtriangleup}\), namely \(N_\frac{\epsilon }{2}=\frac{2\pi}{\epsilon}\). Figure \ref{fig:widthepsdep-long}b,d show that  \(\frac{1}{\alpha^{2}\epsilon}\frac{max_{\phi_0}W(N;\phi_0,\frac{\epsilon }{2})}{W(\epsilon,\alpha)}\approx 10^{-1.25}N^{2/5}\)  so, \(N_{c}(\epsilon,\alpha,\nu)\), the minimal number of iterations needed for obtaining a coverage of  a fraction \(\nu\) of the layer  in an \(\epsilon\)-window, is  \(N_{c}(\epsilon,\alpha,\nu)\approx10^{25/8}\frac{\nu^{5/2}}{\alpha^{5}\epsilon^{5/2}}\), so \(\frac{N_{c}(\epsilon,\alpha,\nu)}{N_\frac{\epsilon }{2}}\approx\frac{10^{3.125}}{2\pi}\frac{\nu^{5/2}}{\alpha^{5}\epsilon^{3/2}}\)  diverges with small \(\epsilon\) and with small \(\alpha\), whereas for  \(\alpha\gtrsim \alpha_c(\epsilon,\nu)= \frac{\sqrt{\nu}}{\epsilon^{0.3}} \frac{10^{5/8}}{(2\pi)^{1/5}}\), we expect to see a \(\nu\) coverage of the layer at  order   \(\frac{1}{1-\nu}N_\frac{\epsilon }{2}\) iterates. Interestingly, for   \(\alpha\approx\alpha_c(0.001,.9)\approx23\) the tangency band and the 1-resonance layer merge, hence, for larger values of \(\alpha \) the structure of the band becomes more complex.

Defining the distance from the boundary as:    \begin{equation}\label{eq:trajdistance}
d(n;\phi_0,\bigtriangleup):=\max\left(\min_{i}\frac{K_{i}^{max}-\max_{j=1,..n}K_{j}|_{\phi_{j}\in (\theta_i-\bigtriangleup,\theta_i+\bigtriangleup)}}{K_{i}^{max}-K_{i}^{min}}, \min_{i}\frac{\min _{j=1,..n}K_{j}|_{\phi_{j}\in (\theta_i-\bigtriangleup,\theta_i+\bigtriangleup)} -K_{i}^{min}}{K_{i}^{max}-K_{i}^{min}} \ \right)\end{equation}where  \(K_{i}^{min,max}\) are the minimal and maximal \(K \) values of the band at the \(ith\) window, a connecting orbit corresponds to \(d(n;\phi_0,\bigtriangleup)\rightarrow0\) with \(O(\frac{\pi }{\bigtriangleup})\) windows. Thus, minimizing \(d\) over \(\phi_{0}\) provides the rate at which a connection is approached, and this minimum is \(1/2\) when the curves are roughly invariant.  For a fixed number of windows, the joint probability to visit the upper most and lowest most slots is  our 5 windows numerical experiment is expected to be of order \(N^{2}_\frac{\epsilon }{2}\).  We find numerically  that \(min_{\phi_{0}}d(N;\phi_0,\frac{\epsilon }{2})\approx0.5-log(N\alpha ^{4}/N^{2}_\frac{\epsilon }{2})^{-1/7}\) , namely,   to get an orbit which visits \(1-\nu\) close to both boundaries of the tangency band we need \(N\approx(\alpha ^{-4}N^{2}_\frac{\epsilon }{2})10^{-3.5+7\nu}\) iterates. In conclusion, both measures indicate that the connection times diverges as a power law in   \(\epsilon\alpha^{2}\) and that for a fixed small \(\epsilon\) and large \(\alpha\) fast mixing is achieved (as long as there is no merge with resonances, which increase the size of the tangency band and thus slows down the connection time). Notably, the dependence of the tangency band width on \(\alpha\)  is rather mild for this range of \(\alpha\) values (for \(\epsilon=0.001\) we find \(W\approx\frac{\epsilon}{0.73-0.014\alpha}\)).  The source for these power law dependencies is unclear and is left for future studies.
\begin{figure}[ht]
\begin{centering}
\includegraphics[scale=0.2]{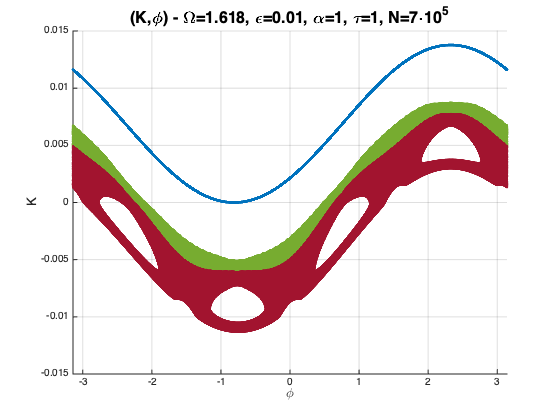}
\includegraphics[scale=0.2]{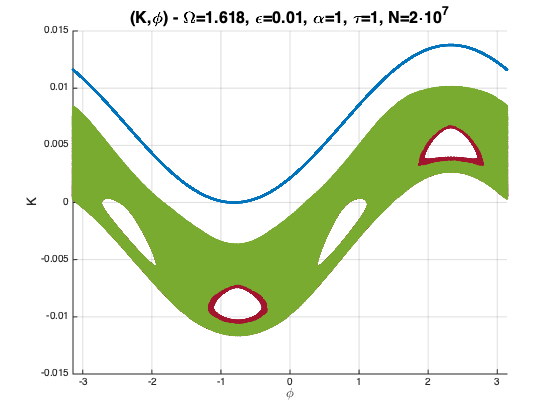}
\includegraphics[scale=0.2]{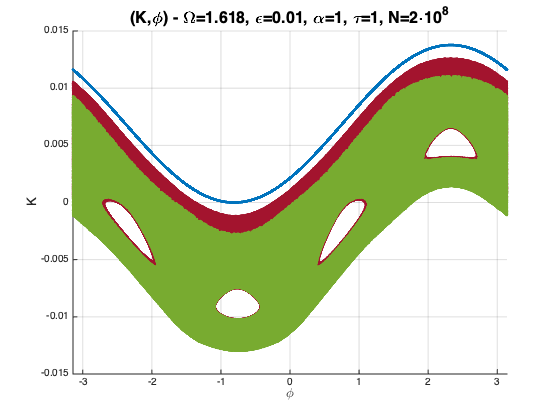}
\includegraphics[scale=0.2]{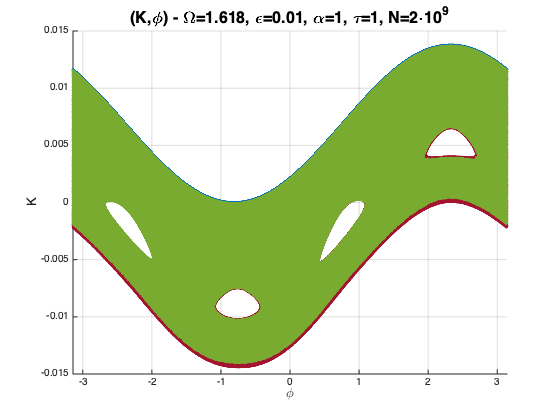}
\includegraphics[scale=0.2]{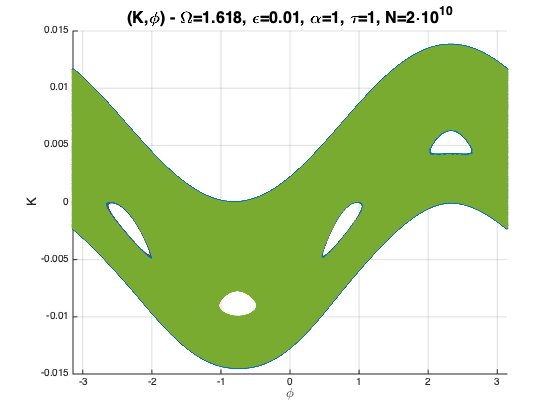}
\includegraphics[scale=0.2]{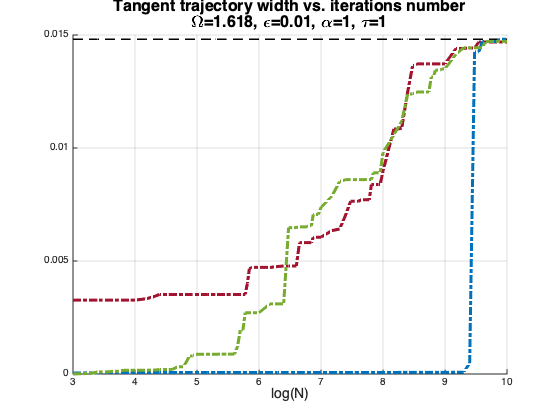}
\par\end{centering}
\protect\caption{\label{fig:tan11long}Long term simulations for a Diophantine rotation number.  (a-e) Three trajectories  marked by blue, green and maroon, shown for an increased number of iterations:  $N=7\cdot 10^5,2\cdot 10^7,2\cdot 10^8,2\cdot 10^9,2\cdot 10^{10} $ (f) The trajectories width as a function of \(\log N\),  with the dashed black line indicating the numerical tangent band width estimate.  Only at   $2\cdot 10^{10}$ iterations all three trajectories cover the entire tangency band. Here \(\Omega=\frac{\sqrt{5}+1}{2}, \epsilon=0.01, \alpha=1,\tau=1\), i.c. are \(K_{0}=0, \phi_0=-2.8109,   -0.8267    ,0.4960. \) Hereafter, points in the figures are diluted by a factor of $1, 10, 100, 1000$ and $10000$ respectively, due to limits of the graphic representation.}
\end{figure}

% res1ell
\begin{figure}[ht]
\begin{centering}
\includegraphics[scale=0.15]{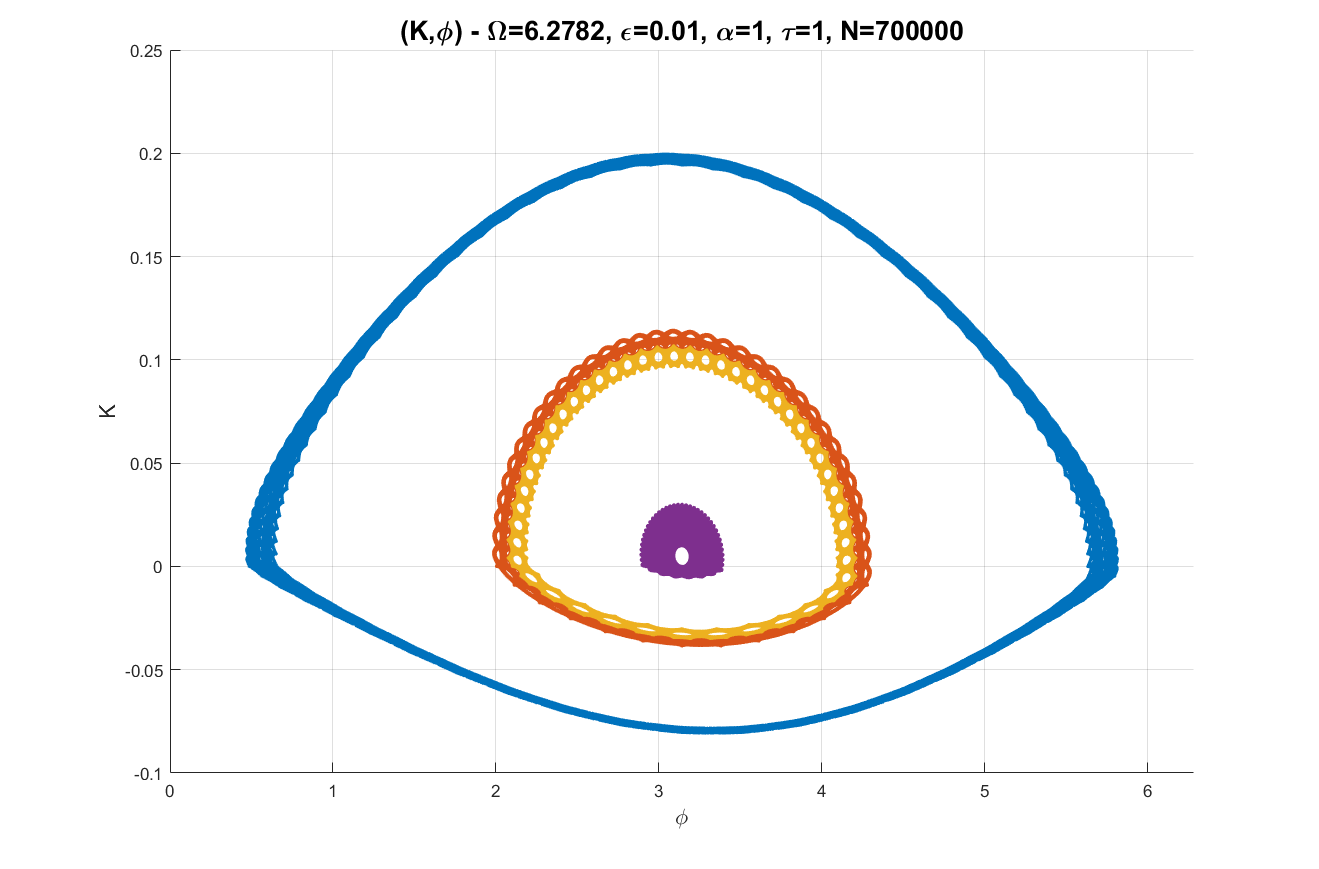}
\includegraphics[scale=0.25]{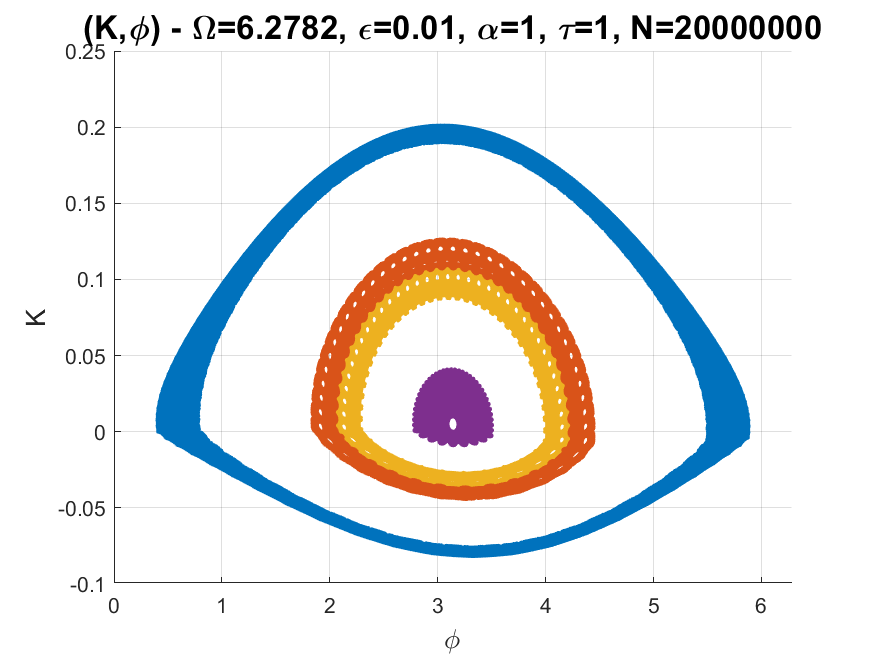}
\includegraphics[scale=0.25]{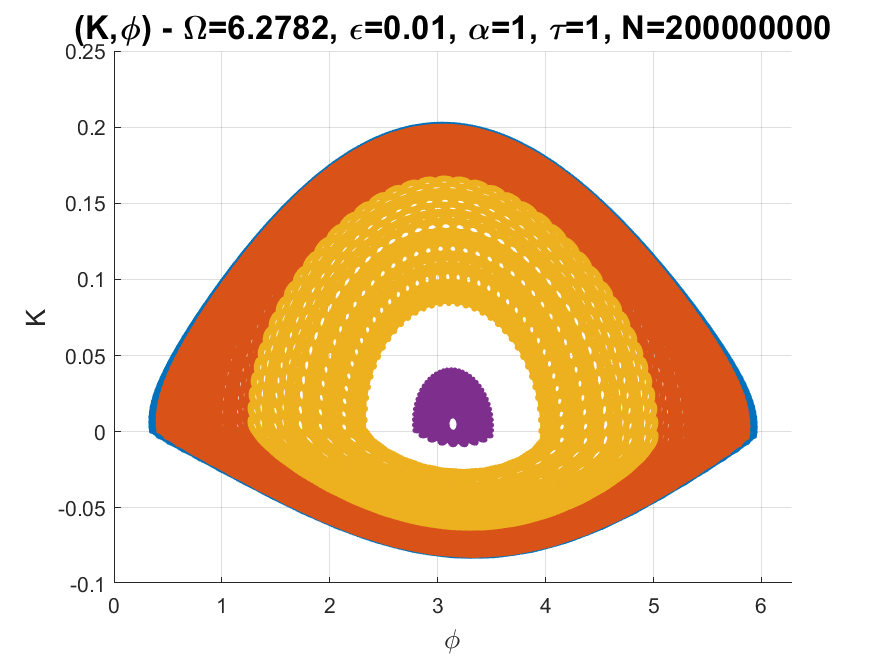}
\includegraphics[scale=0.25]{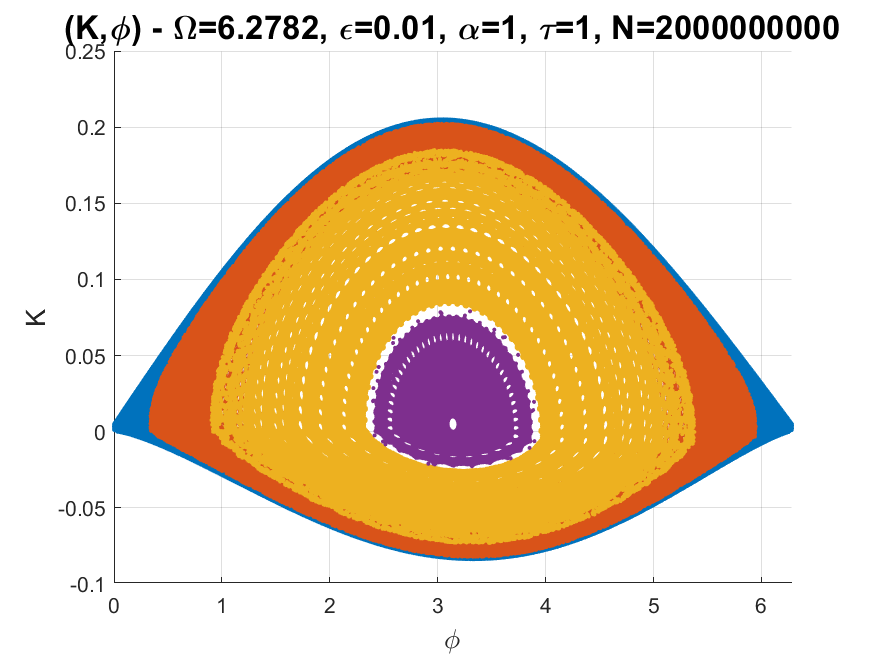}
%includegraphics[scale=0.42]{phenomlong5}
\includegraphics[scale=0.25]{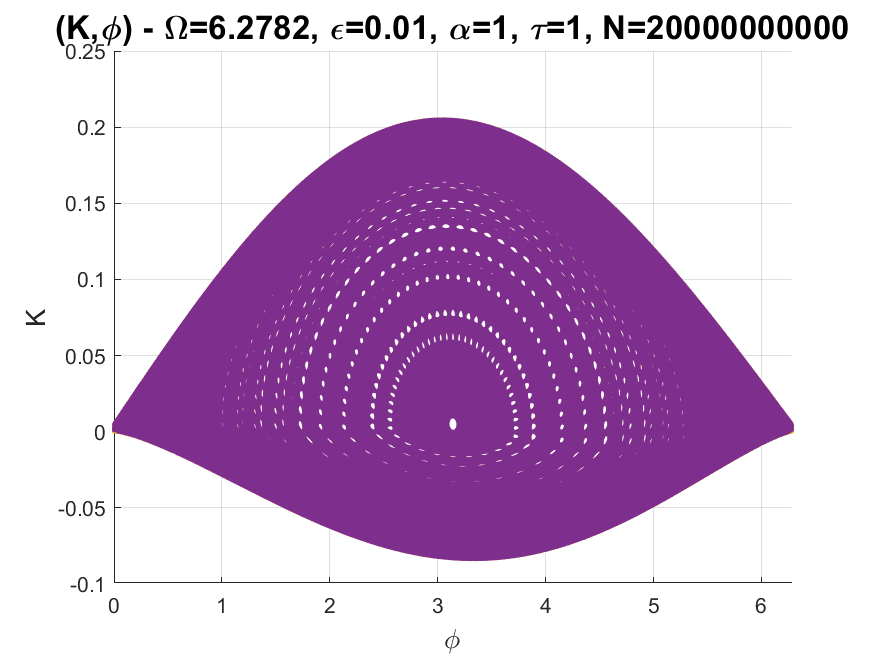}
\includegraphics[scale=0.2]{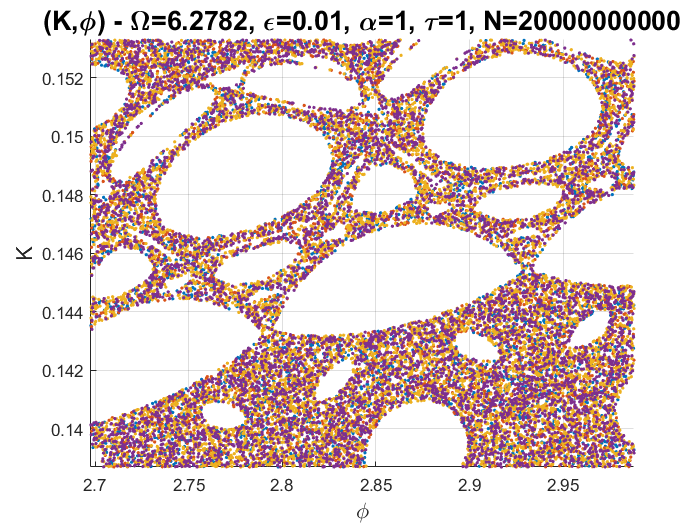}
\includegraphics[scale=0.15]{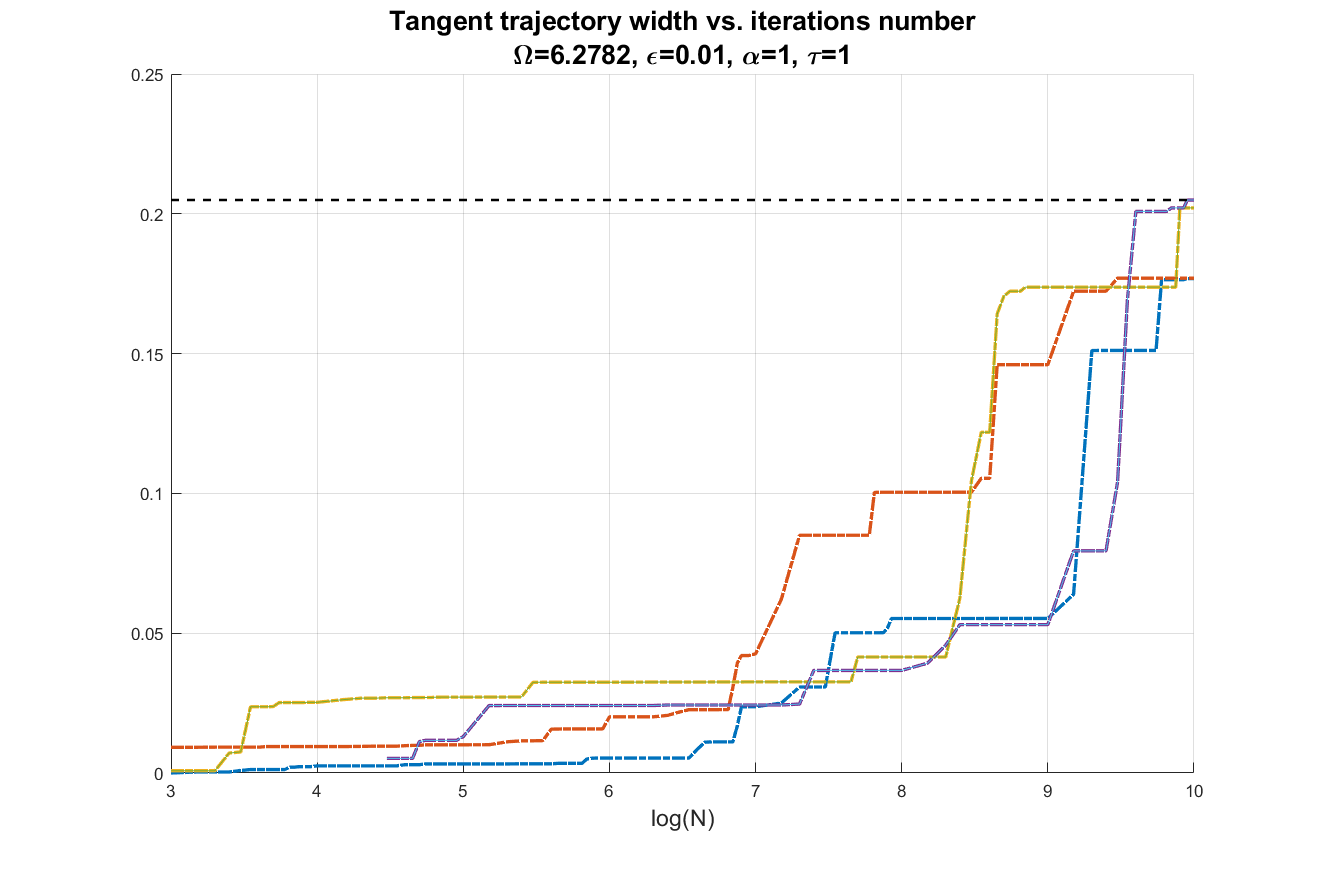}
\par\end{centering}
\protect\caption{\label{fig:tan39long}Long term dynamics for strong resonant and positive shear case.   Separate transient layers  overlap only at large \(N\), and all trajectories  tend to cover the entire singularity set. Nonetheless, even at $2 \cdot 10^{10}$ iterations, uniform covering has not been achieved (here the trajectory and band widths are defined for positive \(K \) values).}
\end{figure}

%res1hyp
\begin{figure}[ht]
\begin{centering}
\includegraphics[scale=0.2]{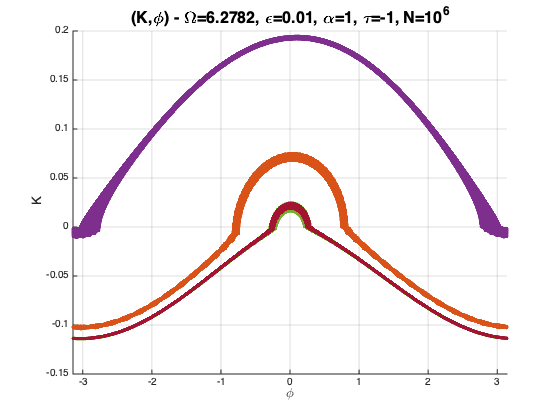}
\includegraphics[scale=0.2]{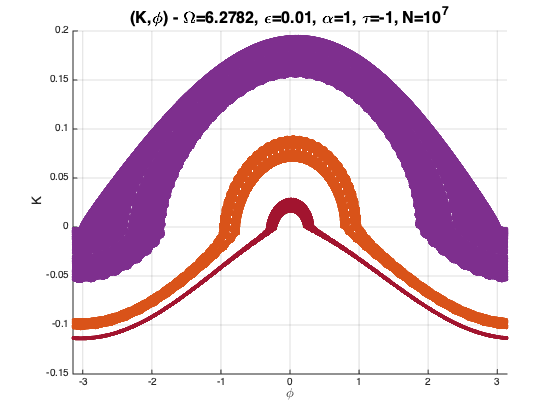}
\includegraphics[scale=0.2]{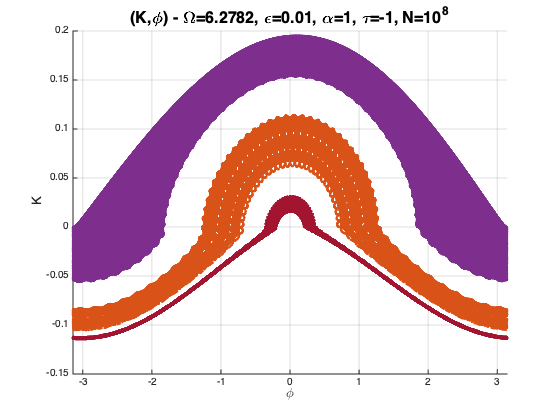}
\includegraphics[scale=0.2]{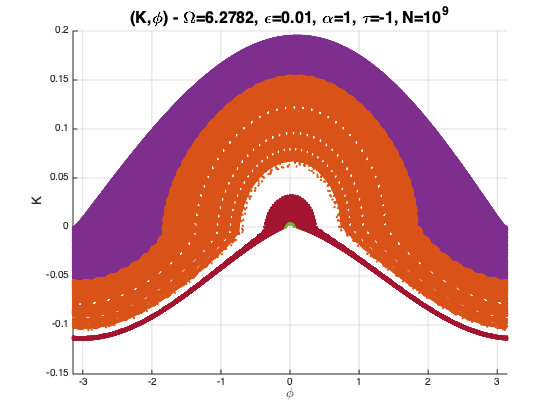}
\includegraphics[scale=0.2]{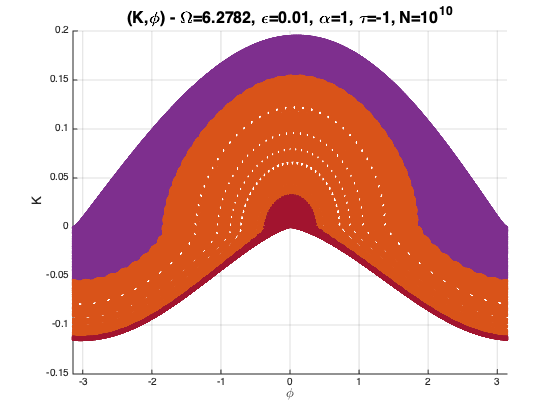}
\includegraphics[scale=0.2]{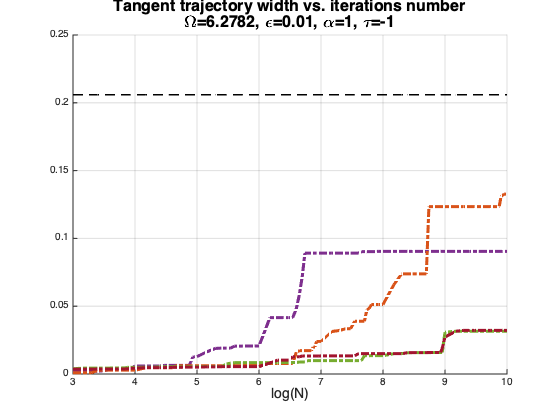}
\par\end{centering}
\protect\caption{\label{fig:tan41long}Long term dynamics for strong resonance with negative shear. The lack of overlapping regions (see zoom-in in Figure \ref{fig:tan41longzoom}) suggests that the tangency zone is divided to three distinct ergodic components.}
\end{figure}

\begin{figure}[ht]
\begin{centering}
\includegraphics[scale=0.2]{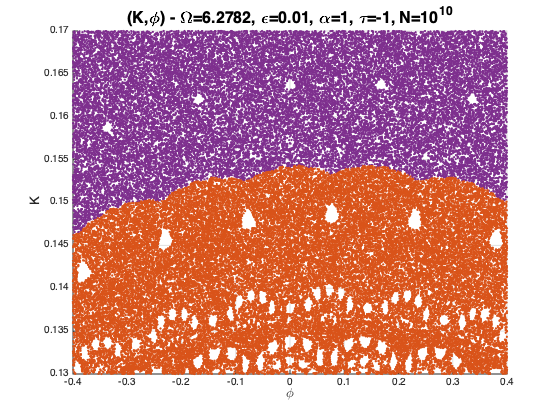}
\includegraphics[scale=0.2]{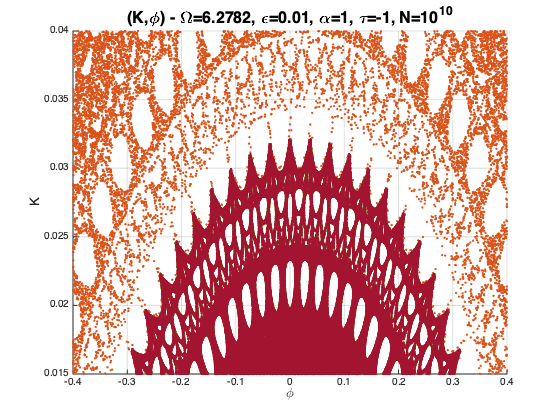}
\includegraphics[scale=0.2]{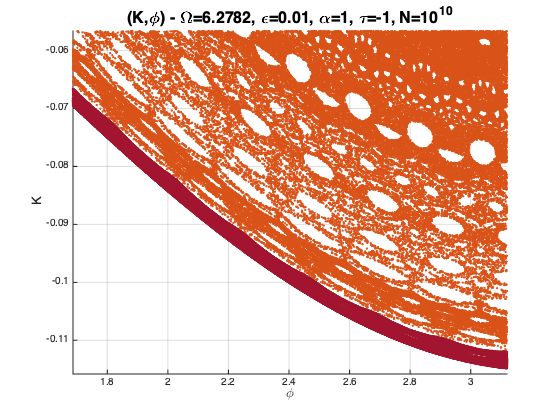}
\includegraphics[scale=0.2]{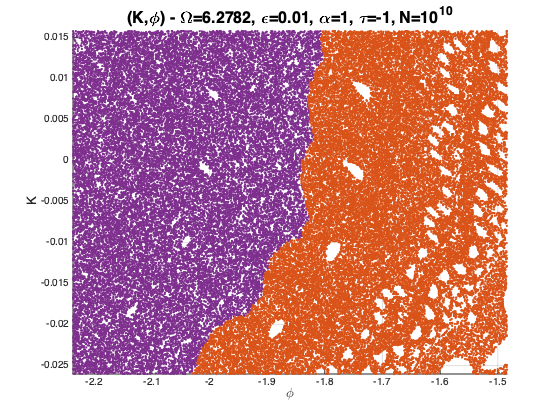}
\par\end{centering}
\protect\caption{\label{fig:tan41longzoom}Separate ergodic components in the tangency zone. Zoom-ins of Figure \ref{fig:tan41long}e demonstrate that there is no overlap between the maroon, purple and orange trajectories (up to \(10^{10}\) iterates).}
\end{figure}

% epsilon dependence
\begin{figure}[ht]
\begin{centering}
\includegraphics[scale=0.2]{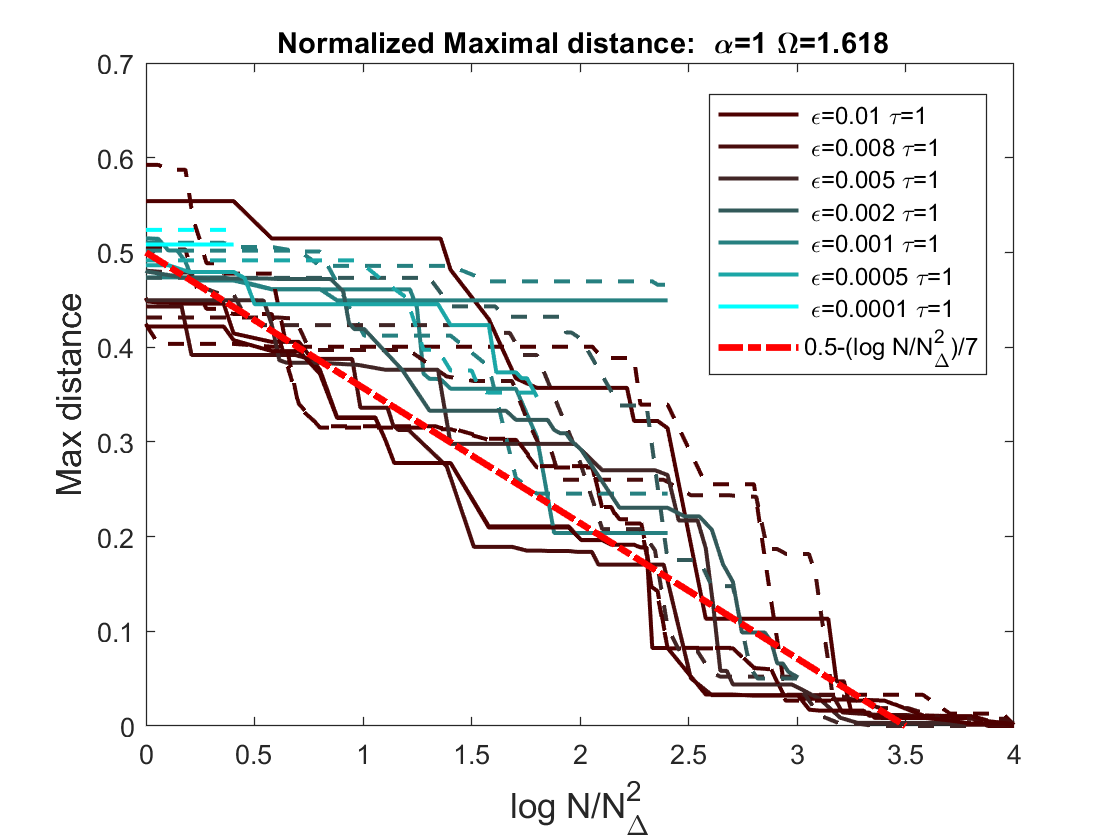} %matlab code
\includegraphics[scale=0.2]{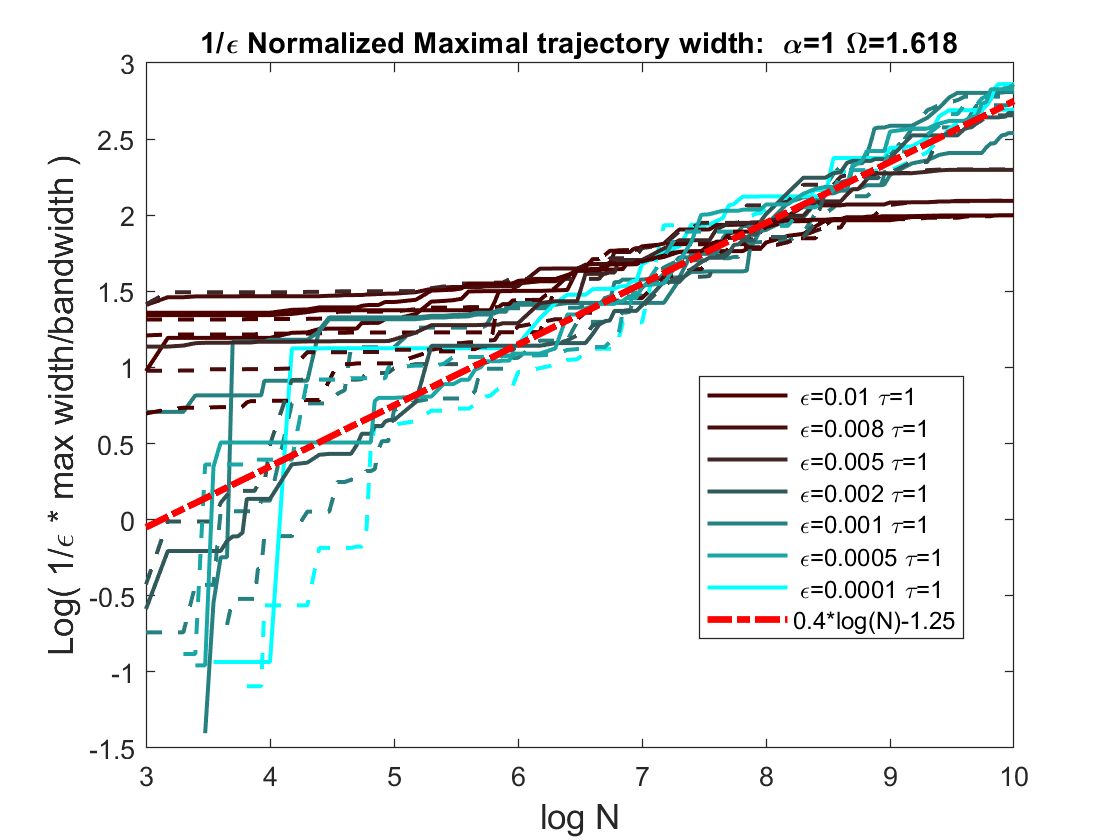} \\
\includegraphics[scale=0.2]{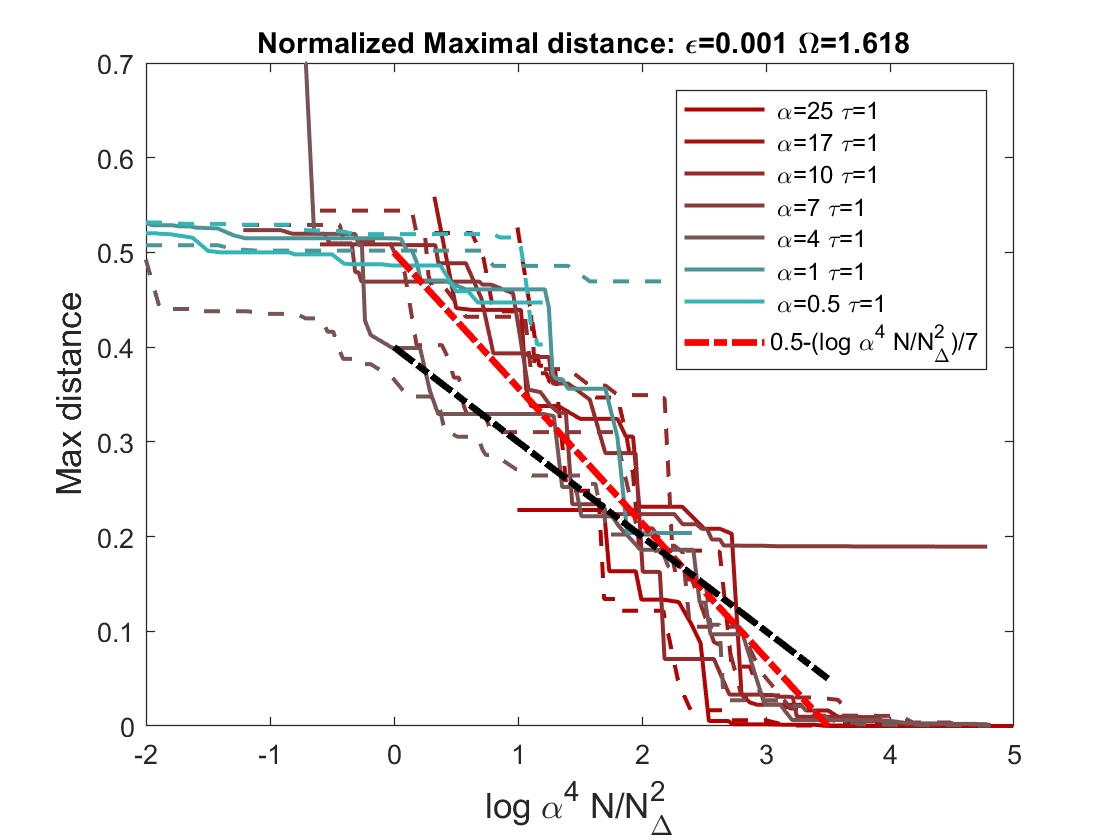}
\includegraphics[scale=0.2]{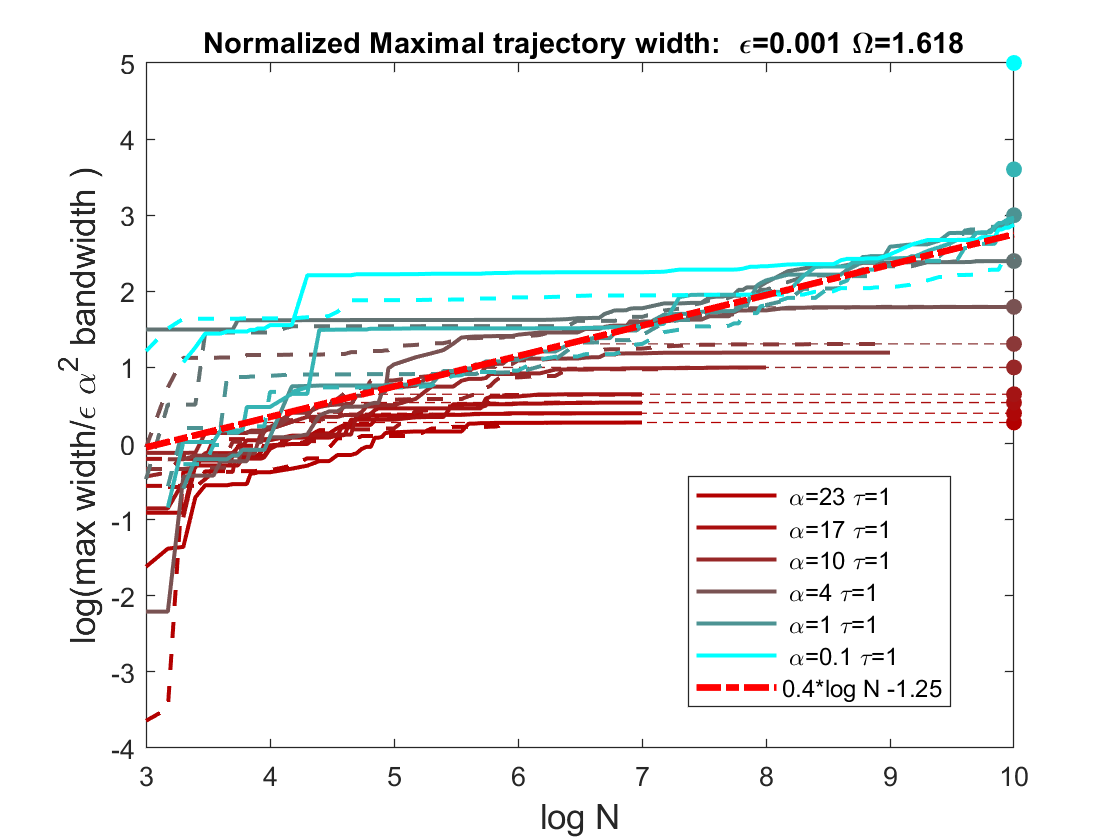}
\par\end{centering}
\protect\caption{\label{fig:widthepsdep-long}  Width and distance from boundary growth - dependence on    $\epsilon$ (a,b) and on \(\alpha\) (c,d). It is observed that that  \(\frac{1}{\alpha^{2}\epsilon}\frac{max_{\phi_0}W(N;\phi_0,\frac{\epsilon }{2})}{W(\epsilon,\alpha)}\approx 10^{-1.25}N^{2/5}\) and that   that \(min_{\phi_{0}}d(N;\phi_0,\frac{\epsilon }{2})\approx0.5-log(N\alpha ^{4}/N^{2}_\frac{\epsilon }{2})^{-1/7}\) ,  }
\end{figure}

\noindent\textbf{Overlaps of transiently-invariant regions}
    The  plateaus of the trajectories width function (over logarithmic time scale) in Figures \ref{fig:tan11long}f,\ref{fig:tan39long}g and \ref{fig:tan41long}f demonstrate the existence of transiently-invariant phenomena (see also additional figures in \cite{PnueliPhd}).  In particular, we observe the following structures:
\begin{itemize}
\item Transiently-invariant \textit{bands}: bands  limited between   dividing circles that cross the singularity line in a close to tangent manner, and appear to provide a partial barrier to the motion: the bands  appear to be  densely covered on intermediate time scales until  the partial boundaries are crossed (e.g.  the blue band in Figure  \ref{fig:tan11long}). \item  Motion in narrow  transiently-invariant \textit{rings} which are densely covered and are limited between  smooth non-dividing circles which cross the singularity line in a close to tangent manner, surrounding singularity band islands (e.g. the blue rings in Figures  \ref{fig:tangency1}(b,e,f) and   \ref{fig:tangency-res1-zoom}(c-e)).

\item \textit{Chaotic bands} that surround  resonant islands and cross the singularity line  for a significant range of angles. Their boundary is typically non-smooth and in some cases appears to be fractal, see Figures \ref{fig:tangency1}(b,e,f), \ref{fig:tangency-res1}(b-e), \ref{fig:tangency-res1-zoom}(a,b), \ref{fig:tangency-res1-alphadependence}(d). The width of these bands is enhanced by overlap of resonances  (similar phenomenon was observed  in 1.5 d.o.f impact systems \cite{nordmark1992effects,lamba1995chaotic}). Figure \ref{fig:tangency-res1}(c) demonstrates that within the limiting non-smooth island, smooth, ordered behavior does occur. This has also been observed in other piecewise smooth systems (e.g. \cite{lee2014dynamics,cao2008limit,altmann2018intermittent}). \item Overlapping of the three transient structures (bands, rings and chaotic bands): there are long transient periods at which orbits covering  densely different transient sets overlap (Figures \ref{fig:tangency1}(e-f), \ref{fig:tangency-res1}(d-e), \ref{fig:tangency-res1-zoom}(b), \ref{fig:tangency-res1-alphadependence}(d)). The resulting transient density is non-uniform.\end{itemize}

  These phenomenon may be possibly explained by adiabatic theory with chaotic jumps at the singularity line, the details of which are left for future studies (briefly, in each half plane the motion follows approximately level sets of an averaged system with jumps that occur only at the singularity line).

% this is from phd11
\begin{figure}[ht]
\begin{centering}
\includegraphics[scale=0.2]{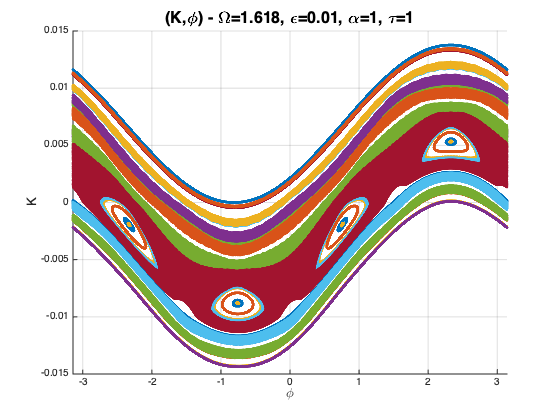}
\includegraphics[scale=0.2]{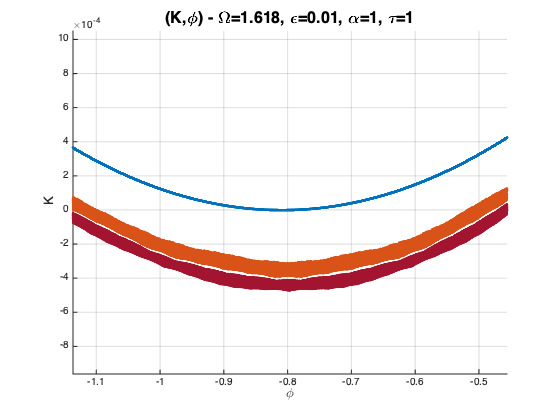}
\includegraphics[scale=0.2]{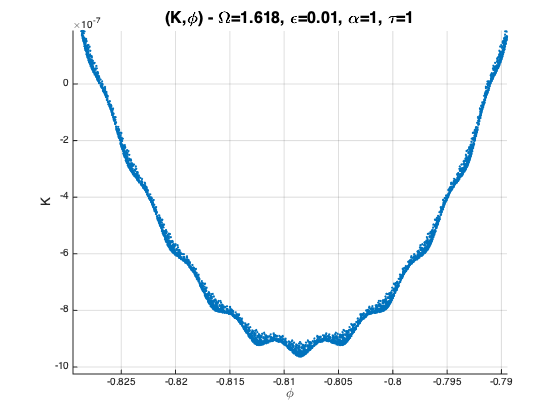}
\includegraphics[scale=0.2]{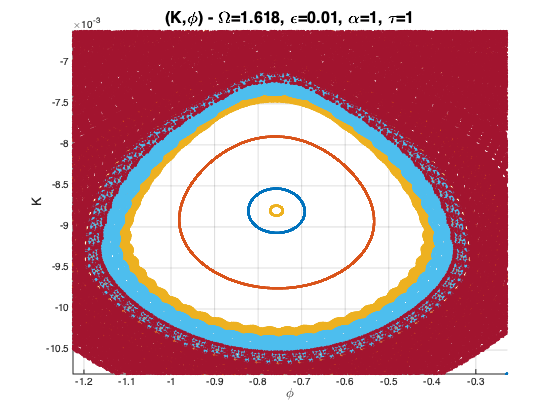}
\includegraphics[scale=0.2]{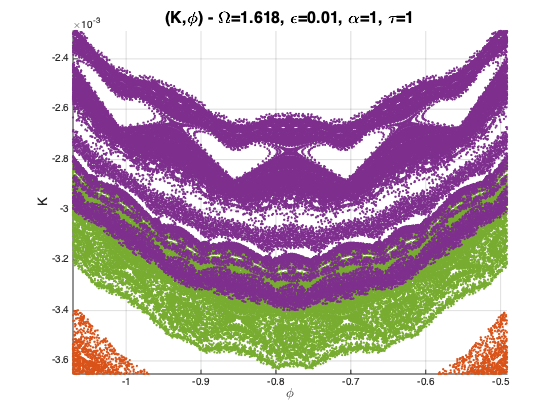}
\includegraphics[scale=0.2]{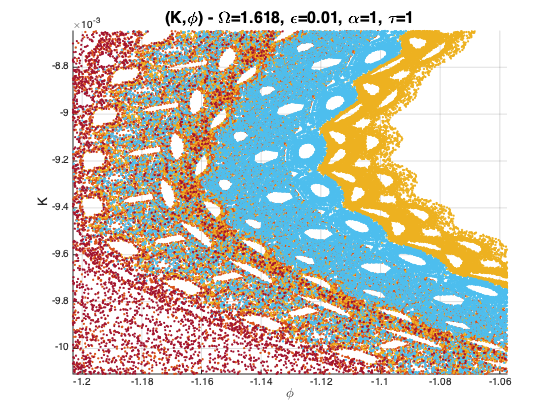}
\par\end{centering}
\protect\caption{\label{fig:tangency1}The singularity set transient structures at a general \(\Omega\). (a) The  band of near tangent trajectories includes transient bands and islands.  (b) Transient bands persist. (c) Zoom-in on the uppermost thinnest band which has a very small negative \(K\) region. (d) The interior of the singularity island has smooth invariant curves that remain bounded away from the singularity line. (e) Nonsmooth islands in various shapes appear. (f) Trajectories covering different regions with an overlap are typical. }
\end{figure}

% phd39
\begin{figure}[ht]
\begin{centering}
\includegraphics[scale=0.2]{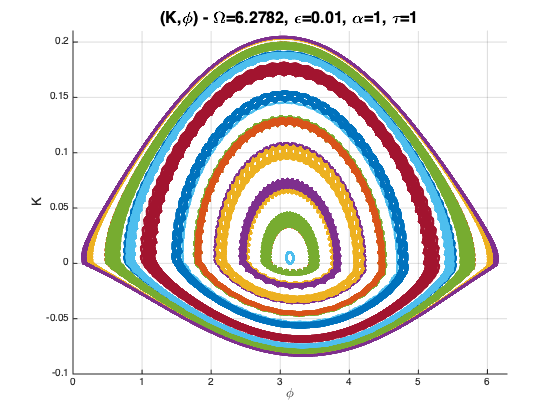}
\linebreak
\includegraphics[scale=0.2]{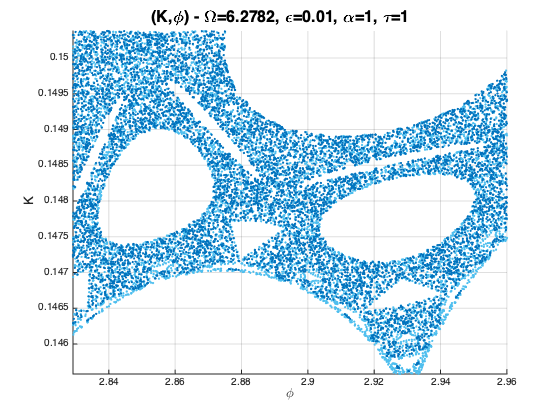}
\includegraphics[scale=0.2]{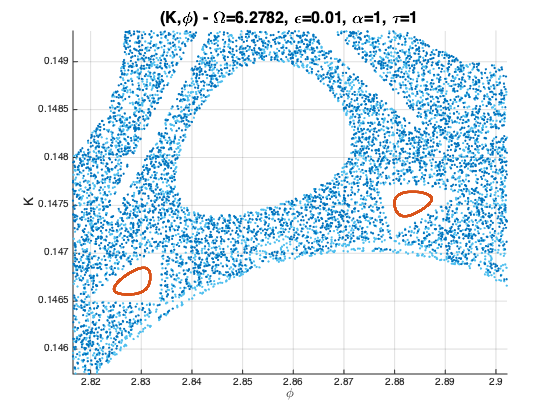}
\includegraphics[scale=0.2]{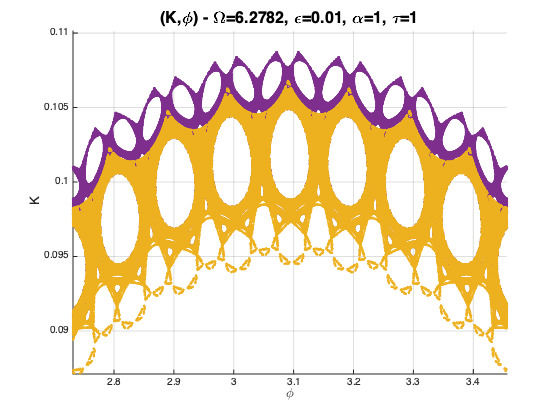}
\includegraphics[scale=0.2]{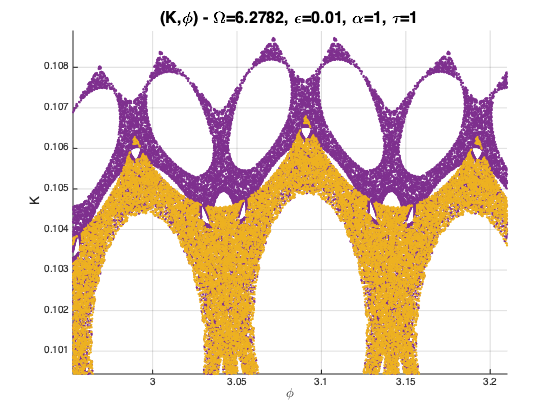}
\par\end{centering}
\protect\caption{\label{fig:tangency-res1}Small islands near tangent dynamics  at strong resonance.  The island-like structures around the main resonance appear to have a positive measure. (b) Smooth and non-smooth island shapes within a chaotic region. (c) Smooth singularity islands appear within the non-smooth islands' holes. (d) Secondary resonances in a chaotic zone. (e) Zoom in on (d) - two chaotic trajectories sharing part of the chaotic region, but not its entirety. Long term simulations (see Figures \ref{fig:tan39long}) show that such trajectories completely mix after \(10^{10}\) iterates.}
\end{figure}

\begin{figure}[ht]
\begin{centering}
\includegraphics[scale=0.2]{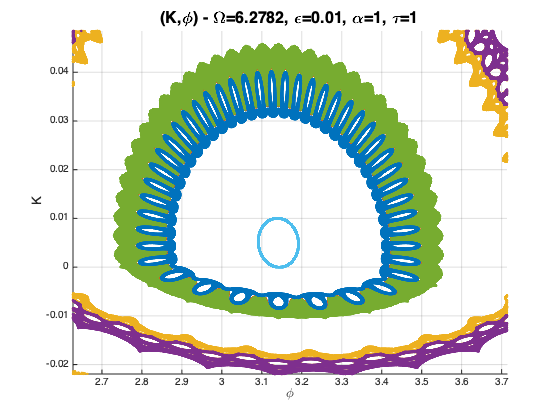}
\includegraphics[scale=0.2]{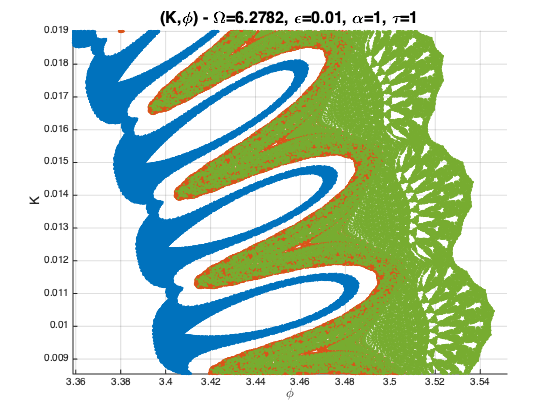}
\linebreak
\includegraphics[scale=0.2]{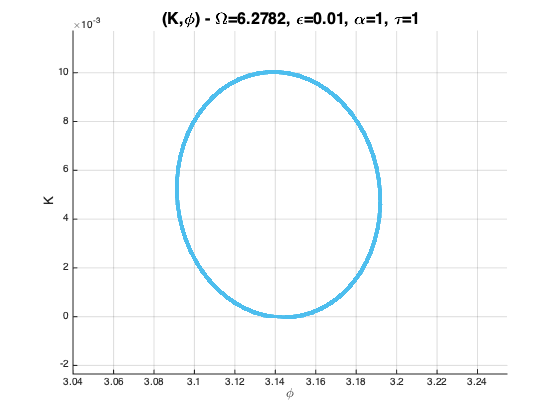}
\includegraphics[scale=0.2]{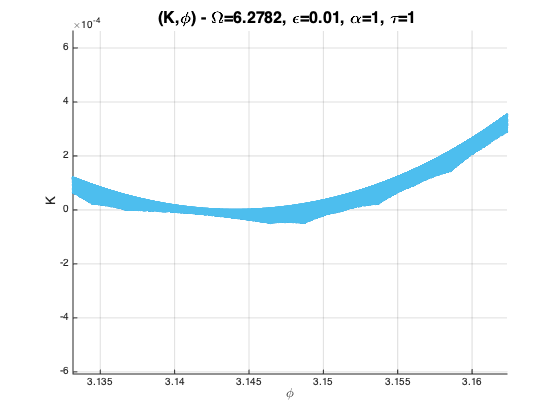}
\includegraphics[scale=0.2]{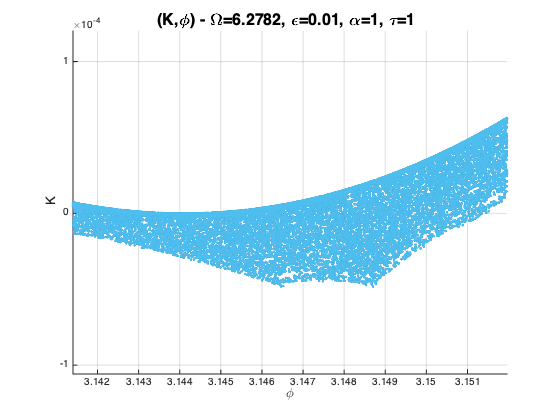}
\includegraphics[scale=0.2]{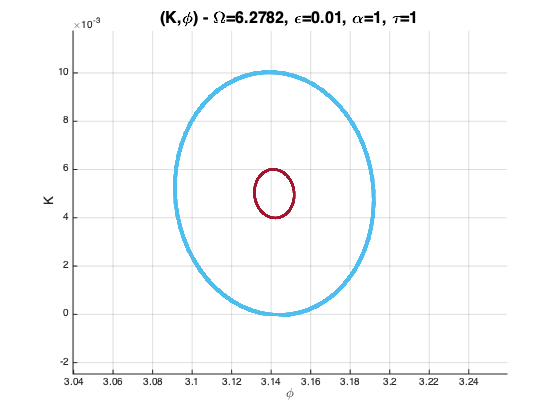}
\includegraphics[scale=0.2]{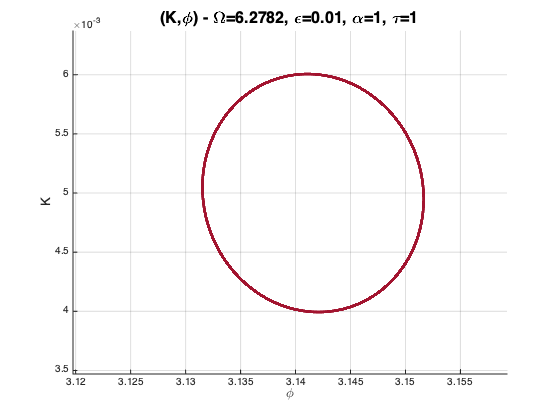}
\includegraphics[scale=0.2]{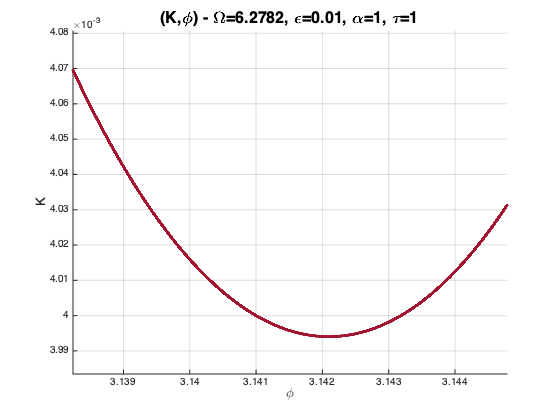}
\par\end{centering}
\protect\caption{\label{fig:tangency-res1-zoom}Floating and grazing boundaries. Zoom-in on parts of Figure \ref{fig:tangency-res1}a show the difference between curves that are bounded away from the singularity curve and those which graze it;  (a) Chaotic layered islands around the elliptic resonant point. (b) Zoom in on chaotic layer - the appearance of multiple ergodic components observed, as well as non-smooth island shapes in the chaotic layer. (c-e) Similar to the boundary layers in Figure \ref{fig:tangency1}(c,d), this island has less intersection with the singularity line and hence smaller width - however it is still chaotic as can be seen zooming in (d-e). (f-h) Additional simulation of a non-impacting, smooth island within the non-smooth island of (c-e).}
\end{figure}

\begin{figure}[ht]
\begin{centering}
\includegraphics[scale=0.2]{tangency-res1}
\includegraphics[scale=0.2]{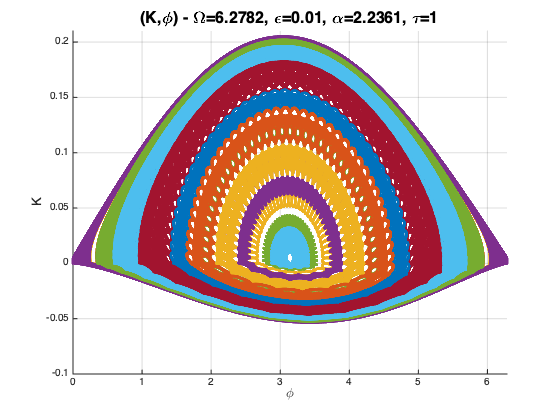}
\includegraphics[scale=0.2]{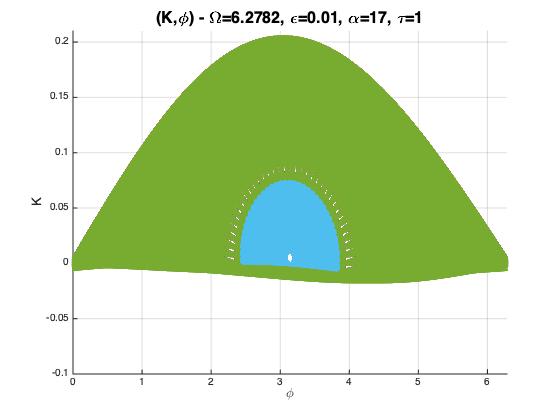}
\includegraphics[scale=0.2]{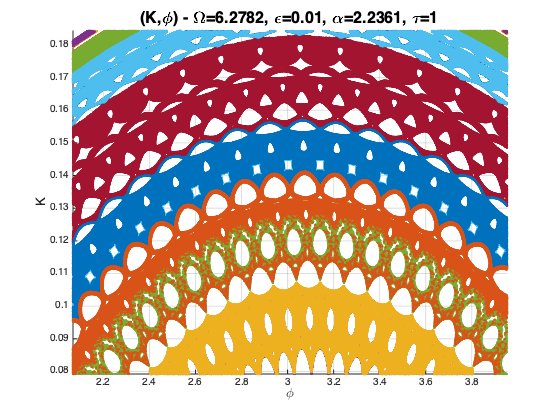}
\par\end{centering}
\protect\caption{\label{fig:tangency-res1-alphadependence}Dependence on $\alpha$ value for the case of elliptic resonance. Parameters are as in Figure \ref{fig:tangency-res1}, except for $\alpha$. (a) $\alpha=1$ (as in Figure \ref{fig:tangency-res1}a). (b) $\alpha=\sqrt{5}$. (c) $\alpha=17$. (d) $\alpha=\sqrt{5}$ - zoom in on the structure of secondary resonances. For convenience of presentation, the angles are wrapped around $[0,2\pi]$}
\end{figure}

% tanmetalpha
\begin{figure}[ht]

\end{figure}

\subsection{Simulations of a near-integrable impact system  }

The perturbed return map (\ref{eq:pert-return}) found by integrating numerically  a  near-integrable Hamiltonian impact flow is shown in Figure \ref{fig:neartangencyflow2}, demonstrating that similar transient structures appear. The separable potential of the system is the Duffing-Center potential:
\(V(q_1,q_2)=-\frac{\lambda}{2}\cdot(q_{1}-q_{1s})^{2}+
\frac{1}{4}\cdot(q_{1}-q_{1s})^{4}+\frac{\omega^{2}}{2}\cdot(q_{2}-q_{2c})^{2}
\), and
the perturbation is a combination of the small, smooth term $\epsilon_r V_c=\epsilon_r(q_1-q_{1s})\cdot(q_2-q_{2c})$ and a small tilt to the perpendicular wall - $q_1^w=\epsilon_w\cdot q_2$, $q_2\in\mathbb{R}$, see  \cite{pnueli2018near} for the analysis of this system away from the tangency region.
The return map to the section \(\Sigma_{E}\)   is computed numerically, and since the potential is quadratic in \(q_2\) the action angle coordinates are simply \( (\theta,I)=(\frac{1}{2\omega}p^{2}_2+\frac{\omega}{2}\cdot(q_{2}-q_{2c})^{2}, \text{arctan}(\frac{\omega\cdot(q_{2}-q_{2c})}{p_2})) \). Due to limitations on the integration accuracy, the feasible number of iterations of the return map is of few thousands, well below  the truncated map intermediate time scale. Indeed, we observe similar qualitative behavior to the one observed in the truncated return map transient time simulations:  chaotic dynamics with non-smooth boundaries surrounding islands, narrow chaotic bands and elliptic islands, and overlapping transiently invariant regions (see Figure \ref{fig:neartangencyflow2}).

% this is same as in nearint paper
\begin{figure}[ht]
\begin{centering}
\includegraphics[scale=0.2]{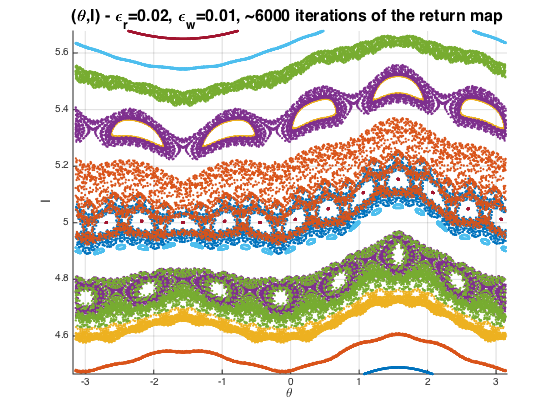}
\includegraphics[scale=0.2]{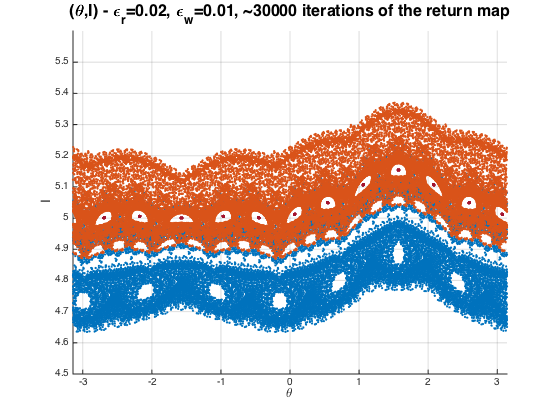}
\par\end{centering}
\protect\caption{\label{fig:neartangencyflow2}Return map of $(\theta,I)$ for near tangent initial conditions. Chaos, non-smooth resonant islands, narrow chaotic bands and overlapping transiently invariant regions are observed.  Here $H=10$, $\lambda=\sqrt{5}-1$, $\omega=1$, $\epsilon_w=0.01,\epsilon_r=0.02, q_{1s}=2.5$, $q_{2c}=0$.   a) $ 6000$ returns b)  $ N=30000$  returns.   }
\end{figure}

\section{Discussion}\label{sec:discussion}
We derived the near tangent return map for a class of near-integrable Hamiltonian impact systems introduced in \cite{pnueli2018near,pnueli2019structure}. To this aim, we defined the tangency curve - the singular curve on the Poincar\'{e} section which separates between impacting and non-impacting trajectories, and compute the perturbed behavior on its two sides using Menikov type integrals.  The near integrable structure implies that this curve is a perturbation of a tangent level set of the system - a torus of tangent trajectories (see  \cite{chillingworth2013periodic,di2001normal,lamba1995chaotic,turaev1998elliptic} for  analysis  of the return map to a neighborhood of  the singularity line/discontinuity surface  near isolated tangent periodic orbits). The resulting map is a piecewise smooth near-integrable symplectic map, exhibiting a square-root singularity along one circle, which is invariant  at the integrable limit.

The  construction of the map allows to refine previous stability results regarding the width of the tangency zone  \cite{pnueli2018near} and to study analytically the behavior near the main resonances of the return map.
 Moreover, it allows to perform  long simulations which reveal the non-trivial asymptotic behavior;  Chaos near tangency always appears as expected, yet,  in some cases  fast chaotization of the full tangency region (up to small stability islands), similar to the behavior in slow-fast systems  \cite{neishtadt2008jump}, is observed, whereas in other cases the number of iterations required to identify how many visible ergodic components  are included in the tangency zone is large. In particular, we observe that the number of iterations needed to cover the tangency band diverges with \(\epsilon\), the perturbation parameter and  with \(\alpha\), the singularity strength coefficient.
In the intermediate time scales, a myriad of chaotic transient phenomena are observed:  thin chaotic bands, non-smooth chaotic resonant islands, and overlapping thick chaotic zones persist for a significant number of iterations. The mechanism for these long transients appears to be different from the standard stickiness to partial barriers which appears in the smooth near-integrable setting  \cite{mackay1984transport,meiss1994transient}, and may be related to   partial barriers in an adiabatic approximation of the motion.
 For example, for the main resonances, at which \(\Omega=2\pi j+\tau\cdot\delta(\epsilon)\) and \(K \) is small, rescaling \(K\) transforms the truncated map   to a  near identity map. The proper scaling for positive \(K\) and negative  \(K\)  is different, so the dynamics may be described as adiabatic in each half plane, with a transition and an adiabatic jump at the singularity line.
Comparison of the truncated map with ODE simulations of near tangent trajectories of a perturbed HIS flow confirm that similar  intermediate time structures emerge, and suggest that for some systems the intermediate time-scale dynamics may be most relevant. The proper mathematical formulation of  intermediate time-scale behavior remains, as usual,   an open  challenge.

\appendix
\section*{Appendices}

\section{The angle variable of \((q,\pm p)\) }\label{appendix:aapsign}

Consider a 1 DOF integrable mechanical Hamiltonian $H(q,p)$ with bounded energy level sets. Let $(\theta,I)$ denote  action-angle variables associated with the system.  Define the \(\pm\) generating functions of the second kind:  \[F(q,I,\pm)=\pm\int_{q_{max}(I)}^{q}|p(x,H(I))|dx\]  where \(|p(x,H(I))|=\sqrt{2(H-V(q))}\). Then, the transformation \(S(q,p)=(\theta,I)\) is produced by \(F(q,I,sign(p))\); first, notice that \(p_{\pm}=\frac{\partial F(q,I,\pm)}{\partial q}=\pm|p(x,H(I))|\)  and \(\theta(q,p_{\pm})=\frac{\partial F(q,I,\pm)}{\partial I}=\mp\omega(I)t(q)\)    where  \(t(q)\in [0,\frac{\pi }{\omega(I)}]\) is the monotone increasing function of \(q\in[q_{min}(I),q_{max}(I)]\) given by \(t(q)=\int^{q_{max}(I)}_{q}\frac{dx}{|p(x,H(I))|}\) (see e.g. \cite{Arnold2007CelestialMechanics}). %\begin{equation}
%\theta=\frac{dH}{dI} \int_{q_{max}(I)}^{q}\frac{1}{p(q)}dq=\omega(I)\int_{q_{max}(I)}^{q}\frac{dq}{p(q)}=\omega(I)t.
%\end{equation}
%\begin{equation}
%\theta=\frac{\partial F}{\partial I}=-q_{min}'(I)p(q_{min}(I),H(I))+\frac{dH}{dI} %\int_{q_{min}(I)}^{q}\frac{dx}{p(x,H)}
%\end{equation}and  \begin{equation}
%p=\frac{\partial F}{\partial q}=p(q,H(I))
%\end{equation}
% $S(q,p)=(\theta,I)$, and $\omega(I)=\frac{2\pi}{T(I)}$. Let \((\theta_{\pm},I_{\pm})=S(q,\pm p)\), where, hereafter \(p>0\) and, b
   We see that \(\theta(q,p_{+})=-\omega(I)t(q)=-\theta(q,p_{-})\) and that \(\theta(q,p_{-})\in[0,\pi]\). Since \(\theta\) is  cyclic, the transformation \(S\) is smooth, and we can rewrite this relation as \(\theta(q,p)=2\pi-\theta(q,-p)\) for all \(\theta\) defined on the \(2\pi\) circle and any \(p\in\mathbb{R}\).

\section{Fixed points of the truncated tangency map.}

Define the following functions and bifurcation values:\begin{equation}
\begin{split}K_{+}^*(j,\Omega,\tau)&:=\frac{2\pi j-\Omega}{\tau}, \\ K^{*}_{-,\pm}(j,\Omega,\tau)&:=
-\frac{\alpha^{2}}{4\tau^{2}}\left(-1\mp\sqrt{1+\frac{4\tau(\Omega-2\pi j)}{\alpha^2}}\right)^2, \\
\Omega_{c,j}&:=2\pi j,  \\
\Omega^{\pm}_{pd,j}(\phi^{*})&:=\Omega_{c,j}+\epsilon\frac{\alpha^2|f'(\phi^{*})|}{8}\frac{ 1\mp\epsilon| f'(\phi^{*})|/8}{(1\mp\epsilon| f'(\phi^{*})|/4)^{2}} \\
\omega_{\alpha}&=\frac{\alpha^2}{4} \text{ mod }2\pi\\
\Omega_{sn,j}&:=\begin{cases}\emptyset\ & \tau=1 \\
2\pi j+\frac{\alpha^2}{4}=2\pi (j+\left\lfloor\frac{\alpha^{2}}{8\pi}\right\rfloor)+\omega_{\alpha} & \tau=-1 \\
\end{cases}
 \end{split}
\end{equation}
Notice that \(K_{+}^*(j,\Omega_{c,j},\tau)=K_{-,-}^*(j,\Omega_{c,j},\tau)=0 \) and \(K^*_{-,-}(j,\Omega_{sn,j},-1)=K^*_{-,+}(j,\Omega_{sn,j},-1)=-\frac{\alpha^{2}}{4}\).
Bifurcation diagrams for fixed points of the   map  \(F_{TTM}\)  as a function of \(\Omega\) are plotted in Fig. \ref{fig:bifomega} and their properties are summarized below:

\begin{lem}
For each zero \(\phi^{*}\)  of \(f(\phi)\) and \(j\in\mathbb{Z}\) the   map \(F_{TTM}\)   has a unique continuous branch of fixed points, \((\phi^{*},\mathbf{K}_{j}=\mathbf{K}_{j}(\Omega,\tau))_{\Omega\in\mathbb{R}}\) emanating from \(\Omega_{c,j}\) (so \(\mathbf{K}_{j}(\Omega_{c,j},\tau)=0\)).  The branch \(\mathbf{K}_{j}(\Omega,1)\) is a singled valued function which is monotonically decreasing in \(\Omega\): \begin{equation}
\mathbf{K}_{j}(\Omega,1)=\{K_{+}^*(j,\Omega,1)\}_{\Omega\leqslant\Omega_{c,j}}\cup
   \{K^*_{-,-}(j,\Omega,1)\}_{\Omega\geqslant\Omega_{c,j}}, \quad j\in\mathbb{Z}.
\end{equation} The branch \(\mathbf{K}_{j}(\Omega,-1)\)  is multi-valued in the range \(\Omega\in[\Omega_{c,j},\Omega_{sn,j}]\) and is composed of 3 connecting, monotonic in \(\Omega\) parts (\(K_{+}^*(j,\Omega,-1)\) and \(K_{-,+}^*(j,\Omega,-1)\) are increasing and \(K^*_{-,-}(j,\Omega,-1)\) is decreasing):\begin{equation}
\mathbf{K}_{j}(\Omega,-1)=\{K_{+}^*(j,\Omega,-1)\}_{\Omega\geqslant\Omega_{c,j}}\cup
   \{K^*_{-,-}(j,\Omega,-1)\}_{\Omega\in[\Omega_{c,j},\Omega_{sn,j}]}\cup
   \{K^*_{-,+}(j,\Omega,-1)\}_{\Omega\leqslant\Omega_{sn,j}} \quad
\end{equation}For sufficiently small \(\epsilon\), the \(K_{+}^*\) and \(K_{-,+}^*\) parts of \(\mathbf{K}_{j}(\Omega,\tau)|_{\Omega\neq\Omega_{sn,j}}\) correspond to  centers for negative \(\tau \cdot f'({\phi^{*}})\) and saddles for  positive \(\tau \cdot f'({\phi^{*}})\) and  \((\phi^{*},K_{-,+}^*(j,\Omega_{sn,j},-1)) \) is parabolic. The  \(K_{-,-}^*\) part corresponds to a saddle for positive \(f'({\phi^{*}})\). For  negative \(f'({\phi^{*}})\),   \(K_{-,-}^*\) corresponds to  a center away from the singularity line, undergoing a period doubling bifurcation at \(\Omega^{\text{sign}(\tau)}_{pd,j}(\phi^{*})\), so,   close to the singularity line, when \(\Omega\in(\Omega_{c,j},\Omega^{\text{sign}(\tau)}_{pd,j}(\phi^{*}))\), it is a  saddle. \end{lem}

\begin{proof}  Calculating the solutions to \(\Omega+\tau\cdot K^{*}=2\pi j\) for positive \(K^*\), and to  \(\Omega+\tau\cdot K^{*}-\alpha\sqrt{-K^{*}}=2\pi j\) for negative \(K^*\) we find the dependence of \(K^{*}\) on \(\Omega ,j\) and the stability of \((\phi^{*},K^*(\Omega,j;\tau))\). Figure \ref{fig:bifomega} shows the corresponding bifurcation diagram for the map (\ref{eq:truncatedreturn}) for the solutions with the smallest \(|K^{*}|\).

For \(\tau>0\), each zero of \(f\) and any integer \(j\) produce a branch of positive fixed points at  \((\phi^{*},K_{+}^*(j,\Omega,\tau))=\frac{2\pi j-\Omega}{\tau}\) (the positive branches), which emanates from the bifurcation frequencies  \(\Omega_{c,j}=2\pi j\), and exist for all \(\Omega\leqslant\Omega_{c,j}\).  Fixed points with negative \(K^{*}\)  appear when the solutions, \(\sqrt{-K^{*}_{-,\pm}(j,\Omega,\tau)}\), to the quadratic equation    \(-\tau\cdot( \sqrt{-K^{*}})^{2}-\alpha\sqrt{-K^{*}}+\Omega-2\pi j=0\) are non-negative.  Formally, the quadratic equation for  \(\sqrt{-K^{*}}\), has, for all \(\Omega>\Omega_{sn,j}=2\pi j-\frac{\alpha^2}{4\tau}\)  solutions  \(\sqrt{-K^{*}_{-,\pm}(j,\Omega,\tau)}
 =\frac{\alpha}{2\tau}\left(-1\mp\sqrt{1+\frac{4\tau(\Omega-2\pi j)}{\alpha^2}}\right)\), yet, since   \(\sqrt{-K^{*}}\) must be positive, we conclude that the  \(K^{*}_{-,+}(j,\Omega,\tau)\) branch is not a valid solution and that the   \(K^{*}_{-,-}(j,\Omega,\tau)\) branch is a valid solution only for \(\Omega\geqslant2\pi j\) with \(K^{*}_{-,-}(j,\Omega,\tau)=
-\frac{\alpha^{2}}{4\tau^{2}}\left(-1+\sqrt{1+\frac{4\tau(\Omega-2\pi j)}{\alpha^2}}\right)^2\) monotonically decreasing in \(\Omega\) and \(K^{*}_{-,-}(j,\Omega_{c,j},\tau)=0\).

For \(\tau<0\), the positive branch of positive fixed points at  \((\phi^{*},K_{+}^*(j,\Omega,\tau))=\frac{\Omega-2\pi j}{|\tau|}\)  exists for all \(\Omega\geqslant\Omega_{c,j}\), whereas two negative branches exist for a range of \(\Omega\) values;  The quadratic equation for  \(\sqrt{-K^{*}}\), has, for all \(\Omega<\Omega_{sn,j}=2\pi j+\frac{\alpha^2}{4|\tau|}\)  solutions
\(\sqrt{-K^{*}_{-,\pm}(j,\Omega,\tau)} = \frac{\alpha}{2|\tau|}\left(1\pm\sqrt{1+\frac{4|\tau|(2\pi j-\Omega)}{\alpha^2}}\right) \). Since, again,    \(\sqrt{-K^{*}}\) must be positive,  the  \(K^{*}_{-,-}(j,\Omega,\tau)\) branch is  valid only for \(\Omega \in [\Omega_{c,j},\Omega_{sn,j}]\) with    \(K^{*}_{-,-}(j,\Omega,\tau)\in[0,-\frac{\alpha^{2}}{4\tau^{2}}]\) monotonically decreasing in this interval and the  other branch is valid solution for all  \(\Omega<\Omega_{sn,j}\) with \(K^{*}_{-,+}(j,\Omega,\tau)\in(-\infty,-\frac{\alpha^{2}}{4\tau^{2}})\)  monotonically increasing in \(\Omega\) with \(K^{*}_{-,+}(j,\Omega_{sn,j},\tau)=-\frac{\alpha^{2}}{4\tau^{2}}\), see Fig.  \ref{fig:bifomega}b.

For all the fixed points that are not on the singularity line,  stability is determined by the trace of the map: \begin{equation}\label{appendix:eq:traceJ}
trDF_{TTM}|_{(\phi^{*},K^*)}=2+\epsilon f'({\phi^{*}})\cdot(\tau+\frac{\alpha}{2\sqrt{-K^{*}}}Heav(-K^{*})).
\end{equation}
and, as usual  for symplectic maps, when \(|trDF_{TTM}|_{(\phi^{*},K^*)}|<2\) the fixed point \({(\phi^{*},K^*)}\) is a center, undergoing a center-saddle bifurcation when \(trDF_{TTM}|_{(\phi^{*},K^*)}=2\) and a period doubling bifurcation when \(trDF_{TTM}|_{(\phi^{*},K^*)}=-2\). Since \(f'\) is bounded and \(|\tau|=1\), for the positive branch, \(K_{+}^*(j,\Omega,\tau)  \), and for the negative-plus branch, for which \(K_{-,+}^*(j,\Omega,\tau)\leqslant-\left(\frac{\alpha}{2\tau}\right)^{2}\  \) so \(\frac{\alpha}{2\sqrt{-K_{-,+}^{*}}}\in(0,1]\), we see that  \(trDF_{TTM}|_{(\phi^{*},K^*)}=2+O(\epsilon)\). In particular, for sufficiently small \(\epsilon\),  the stability of the positive branch, \(K_{+}^*(j,\Omega,\tau)  \) and of  the negative-plus branch, \(K_{-,+}^*(j,\Omega,\tau)\) (which exists only for \(\tau=-1\)), is determined by \(sign( f'({\phi^{*}})\cdot\tau)\): the fixed point \((\phi^{*},K_{+}^*(j,\Omega,\tau))\) is a center if \( f'({\phi^{*}})\cdot\tau<0\) and is a saddle if  \( f'({\phi^{*}})\cdot\tau<0\) and, for all \(\Omega\neq\Omega_{sn,j}\) the same statement applies to the fixed point \((\phi^{*},K_{-,+}^*(j,\Omega,\tau)) \) ( whereas  \((\phi^{*},K_{-,+}^*(j,\Omega_{sn,j},\tau)) \) is parabolic.

For the negative-minus branch, for \(\tau=1\), away from the singularity, when  \(\frac{\alpha}{2\sqrt{-K_{-,-}^{*}(j,\Omega,1)}}\in(0,\frac{4}{\epsilon| f'(\phi^{*})|}-1]\) the stability is  determined, as above, by the sign of \(f'({\phi^{*}})\), so for positive \(f'({\phi^{*}})\) the branch is a saddle and for negative  \(f'({\phi^{*}})\) it is a center ( Fig. \ref{fig:bifomega}a,b). The unstable branch remains a saddle also close to the singularity, whereas for negative  \(f'({\phi^{*}})\)  the center changes its stability to a saddle via a period doubling bifurcation at \(\Omega^{+}_{pd,j}=\Omega_{c,j}+\delta^{+}\) when
\(
\frac{\alpha\epsilon| f'(\phi^{*})|}{2(4-\epsilon| f'(\phi^{*})|)}=\sqrt{-K_{-,-}^{*}(j,\Omega^{+}_{pd,j},1)}
=\frac{\alpha}{2}\left(-1+\sqrt{1+\frac{4\delta}{\alpha^2}}\right) \)
i.e. when \(\delta^{+}=\frac{\alpha^2}{8}\frac{\epsilon |f'(\phi^{*})|-(\epsilon f'(\phi^{*}))^{2}/8}{(1-\epsilon| f'(\phi^{*})|/4)^{2}}=\frac{\alpha^2}{8}(\epsilon |f'(\phi^{*})|+O(\epsilon^2))\).
\\    \\ For \(\tau=-1\), recall that the negative branch satisfies     \(K^{*}_{-,-}(j,\Omega,\tau)\in[0,-\frac{\alpha^{2}}{4}]\), so,    \(1<\frac{\alpha}{2\sqrt{-K_{-,-}^{*}(j,\Omega,\tau)}}\)  and thus, as above, positive   \( f'({\phi^{*}})\) produce a saddle branch whereas negative    \( f'({\phi^{*}})\) produce a center branch which changes its stability via period doubling bifurcation when        \(\frac{\alpha}{2\sqrt{-K_{-,-}^{*}(j,\Omega_{pd,j},\tau)}}=\frac{4}{\epsilon| f'(\phi^{*})|}+1\), i.e., setting  \(\Omega^{-}_{pd,j}=\Omega_{c,j}+\delta^{-}\), we get
\(
\frac{\alpha\epsilon| f'(\phi^{*})|}{2(4+\epsilon| f'(\phi^{*})|)}=\sqrt{-K_{-,-}^{*}(j,\Omega^{-}_{pd,j},1)}
=\frac{\alpha}{2}\left(1-\sqrt{1-\frac{4\delta}{\alpha^2}}\right) \), so
\(
\delta^{-}=\frac{\alpha^2}{8}\frac{\epsilon| f'(\phi^{*})|+\epsilon^{2}| f'(\phi^{*})|^{2}/8}{(1+\epsilon| f'(\phi^{*})|/4)^{2}}=\frac{\alpha^2}{8}(\epsilon |f'(\phi^{*})|+O(\epsilon^2))\), i.e. we get up to second order in \(\epsilon\) the same bifurcation points as for \(\tau=1\).
\end{proof}

\begin{cor}
Each simple zero \(\phi^{*}\) of \(f\) produces, for sufficiently small \(\epsilon\),   a countable number of fixed points \( Fix|_{\phi^{*}}=\{(\phi^{*},K^{*})|K^{*}\in \mathcal{K}(\Omega,\tau)\}\) of  \(F_{TTM}\)   :\begin{description}
\item[Positive shear: ]
each branch \(\mathbf{K}_{j},j\in\mathbb{Z}\) contributes a single fixed point to \(\mathcal{K}\):  \[
\mathcal{K}(\Omega,1)=
   \bigcup_{j\leqslant\left\lfloor\frac{\Omega}{2\pi}\right\rfloor}K^*_{-,-}(j,\Omega,1)\cup\bigcup_{j>\left\lfloor\frac{\Omega}{2\pi}\right\rfloor}K_{+}^*(j,\Omega,1)\]
   \item[ Negative shear, super-resonance case (\(\Omega<\frac{\alpha^2}{4}(\text{mod }2\pi)\)):] \((\left\lfloor\frac{\alpha^{2}}{8\pi}\right\rfloor+1)\) branches of \(\mathbf{K}_{j}\) contribute 3 fixed points to \(\mathcal{K}\):  \[\begin{split}\mathcal{K}(\Omega,-1)|_{\Omega<\frac{\alpha^2}{4}(\text{mod }2\pi)}&=
   \bigcup_{j\leqslant\left\lfloor\frac{\Omega}{2\pi}\right\rfloor}K_{+}^*(j,\Omega,-1))\cup
   \bigcup_{j>\left\lfloor\frac{\Omega}{2\pi}\right\rfloor}K^*_{-,+}(j,\Omega,-1))
   \\&\qquad\ \cup
   \bigcup_{l=0,..,\left\lfloor\frac{\alpha^{2}}{8\pi}\right\rfloor,\sigma=\pm }K^*_{-,\sigma}(\left\lfloor\frac{\Omega}{2\pi}\right\rfloor-l,\Omega,-1)\end{split}\]     \item[ Negative shear, sub-resonance case (\(\Omega>\frac{\alpha^2}{4}(\text{mod }2\pi)\)):]   \(\left\lfloor\frac{\alpha^{2}}{8\pi}\right\rfloor\)  branches of \(\mathbf{K}_{j}\)  contribute 3 fixed points to \(\mathcal{K}\): \[\begin{split}\mathcal{K}(\Omega,-1)|_{\Omega<\frac{\alpha^2}{4}(\text{mod }2\pi)}&=
   \bigcup_{j\leqslant\left\lfloor\frac{\Omega}{2\pi}\right\rfloor}K_{+}^*(j,\Omega,-1))\cup
   \bigcup_{j>\left\lfloor\frac{\Omega}{2\pi}\right\rfloor}K^*_{-,+}(j,\Omega,-1))
   \\& \qquad \cup
   \bigcup_{l=0,..,\left\lfloor\frac{\alpha^{2}}{8\pi}\right\rfloor -1,\sigma=\pm}K^*_{-,\sigma}(\left\lfloor\frac{\Omega}{2\pi}\right\rfloor-l,\Omega,-1)
   \end{split}\]     where the last set  is empty if \(\alpha^{2}<8\pi\).\end{description}     \end{cor}

\section{Notation}

\begin{itemize}
\item  \( g(x)=\mathcal{O}_{C^{r}}(f(x))\)  means that  \( g(x)=G(f(x),x)\)  where \(G\) is \(C^{r}\) smooth in its arguments and \(G(0,x)=0\).  \item    \(\mathcal{R}_{1}z^{w}=\mathcal{R}_{1}(q_{1}^{w},p^{w}_1,q^{w}_2,p^{w}_2)=(q_{1}^{w},-p^{w}_1,q^{w}_2,p^{w}_2)\) the elastic reflection from the wall (and, more generally, by \(\mathcal{R}_{i}\) the reflection of the \(i\) th momentum). \item $(\theta,I)=S_2(q_2,p_2)$ the unperturbed smooth transformation to action-angle coordinates
\item \(\Pi\) - projection to  the $(q_2,p_2)$ plane.
 \item \(\sigma^{\epsilon}=\{(\theta,I_{tan}^{\epsilon}(\theta)),\theta\in[-\pi,\pi]\}\) -the singularity curve, \(I_{tan}^{0}(\theta)=I_{tan}(E)=H_{2}^{-1}(E-V_{1}(q_1^w))\).
\item \( \sigma_{wall}^{\epsilon}(E)=\{(\theta ,I_{tan,\Sigma^*_E}^{\epsilon}(\theta))|\theta\in[-\pi,\pi]\}\)- The projection to the \(q_2,p_2)\) plane of the singularity curve at the wall:   \(\sigma_{wall}^{\epsilon}(E)=\Pi \Xi_{wall}^{\epsilon}(E)\) \item   \(\mathcal{ S}^{\epsilon}\) - the singularity set\item   \(\mathcal{B}^{\epsilon}\) - the tangency band,  \(W\)-its width and \(W^{\pm}\) its lower and upper widths.\item
 \(\Phi^{\epsilon,im}_{t}\)  the wall system flow,  \(\Phi^{\epsilon}_{t}\)  the auxiliary smooth flow.
 \item   \(\mathcal{F}^{loc}_\epsilon\)    - the local return map
(Eq. \ref{eq:pert-return})\item \(F_{TTM}\) - the truncated tangency map (Eq. \ref{eq:truncatedreturn})
\item \(z^{\theta,\rho,\epsilon}\) - an initial condition on \(\Sigma_{E}\) which is at an action distance $\rho$ from the singularity line, so \(\Pi z^{\theta,\rho,\epsilon}=(\theta,I^{\epsilon}_{tan}(\theta))\).
\item $z_{tan}(t,\theta; E)=\Phi_t^0 z^{\theta,0,0}$ the unperturbed tangent trajectory with phase \(\theta\) on $\Sigma_E$.
\item  \(z_{im}^{\theta,\rho,\epsilon}(t)=\Phi_t^{\epsilon,im} z^{\theta,\rho,\epsilon}\) the wall system trajectory, \(z_{sm}^{\theta,\rho,\epsilon}(t)=\Phi_t^\epsilon z^{\theta,\rho,\epsilon}\) - the perturbed smooth trajectory  (without impacts).

\item Return times and crossing times: $\tilde {T}_1^{\theta,\rho,\epsilon}$ (respectively ${T}_1^{\theta,\rho,\epsilon}$ ) the return time of $z^{\theta,\rho,\epsilon}$ under the wall flow (respectively under the smooth flow) to $\Sigma_E$. The unperturbed return times are independent of $\theta$, and are denoted by  $T^{tan}_1(E)={T}_1^{\theta,0,0}, T_1^{\rho}=T_1^{\theta,\rho,0}=T_1(J(I_{tan}(E)-\rho,E))$ and for the impact unperturbed flow: $\tilde T_1^{\rho}=\tilde T_1^{\theta,\rho,0}= T_1^{\rho}-\Delta t_{travel}(E-H_2(I_{tan}(E)-\rho))$. .
\end{itemize}

%\bibliography{tangencies-bib}

\begin{thebibliography}{10}

\bibitem{altmann2018intermittent}
{\sc E.~Altmann}, {\em Intermittent chaos in Hamiltonian dynamical systems},
  PhD thesis, Universitat Wuppertal, Fakultat fur Mathematik und
  Naturwissenschaften, 2018.

\bibitem{Arnold2007CelestialMechanics}
{\sc V.~Arnold, V.~Kozlov, and A.~Neishtadt}, {\em Mathematical Aspects of
  Classical and Celestial Mechanics}, vol.~3, Springer Science \& Business
  Media, 2007.

\bibitem{cao2008limit}
{\sc Q.~Cao, M.~Wiercigroch, E.~Pavlovskaia, C.~Grebogi, and J.~Thompson}, {\em
  The limit case response of the archetypal oscillator for smooth and
  discontinuous dynamics}, International Journal of Non-Linear Mechanics, 43
  (2008), pp.~462--473.

\bibitem{chillingworth2013periodic}
{\sc D.~Chillingworth and A.~Nordmark}, {\em Periodic orbits close to grazing
  for an impact oscillator}, in Recent Trends in Dynamical Systems, Springer,
  2013, pp.~25--37.

\bibitem{di2001grazing}
{\sc M.~Di~Bernardo, C.~Budd, and A.~Champneys}, {\em Grazing and
  border-collision in piecewise-smooth systems: A unified analytical
  framework}, Physical Review Letters, 86 (2001), p.~2553.

\bibitem{di2001normal}
{\sc M.~Di~Bernardo, C.~Budd, and A.~Champneys}, {\em Normal form maps for
  grazing bifurcations in n-dimensional piecewise-smooth dynamical systems},
  Physica D: Nonlinear Phenomena, 160 (2001), pp.~222--254.

\bibitem{neishtadt2008jump}
{\sc I.~Gorelyshev and A.~Neishtadt}, {\em Jump in adiabatic invariant at a
  transition between modes of motion for systems with impacts}, Nonlinearity,
  21 (2008), p.~661.

\bibitem{granados2012melnikov}
{\sc A.~Granados, S.~Hogan, and T.~Seara}, {\em The {Melnikov} method and
  subharmonic orbits in a piecewise-smooth system}, SIAM Journal on Applied
  Dynamical Systems, 11 (2012), pp.~801--830.

\bibitem{ivanov1994impact}
{\sc A.~Ivanov}, {\em Impact oscillations: linear theory of stability and
  bifurcations}, Journal of Sound and Vibration, 178 (1994), pp.~361--378.

\bibitem{ivanov1996bifurcations}
{\sc A.~Ivanov}, {\em Bifurcations in impact systems}, Chaos, Solitons \&
  Fractals, 7 (1996), pp.~1615--1634.

\bibitem{RK2014smooth}
{\sc M.~Kloc and V.~Rom-Kedar}, {\em Smooth {Hamiltonian} systems with soft
  impacts}, SIAM Journal on Applied Dynamical Systems, 13 (2014),
  pp.~1033--1059.

\bibitem{kozlov1991billiards}
{\sc V.~Kozlov and D.~Treshch{\"e}v}, {\em Billiards: a genetic introduction to
  the dynamics of systems with impacts}, vol.~89, American Mathematical Soc.,
  1991.

\bibitem{kunze2001non}
{\sc M.~Kunze and T.~K{\"u}pper}, {\em Non-smooth dynamical systems: an
  overview}, Ergodic theory, analysis, and efficient simulation of dynamical
  systems,  (2001), pp.~431--452.

\bibitem{kunze1997application}
{\sc M.~Kunze, T.~K{\"u}pper, and J.~You}, {\em On the application of kam
  theory to discontinuous dynamical systems}, journal of differential
  equations, 139 (1997), pp.~1--21.

\bibitem{lamba1995chaotic}
{\sc H.~Lamba}, {\em Chaotic, regular and unbounded behaviour in the elastic
  impact oscillator}, Physica D: Nonlinear Phenomena, 82 (1995), pp.~117--135.

\bibitem{lee2014dynamics}
{\sc D.~Lee, D.~Lee, H.~Choi, and S.~Jo}, {\em Dynamics on an invariant set of
  a two-dimensional area-preserving piecewise linear map}, East Asian
  mathematical journal, 30 (2014), pp.~583--597.

\bibitem{mackay1984transport}
{\sc R.~MacKay, J.~Meiss, and I.~Percival}, {\em Transport in {Hamiltonian}
  systems}, Physica. D, Nonlinear phenomena, 13 (1984), pp.~55--81.

\bibitem{meiss1994transient}
{\sc J.~Meiss}, {\em Transient measures in the standard map}, Physica D:
  Nonlinear Phenomena, 74 (1994), pp.~254--267.

\bibitem{nordmark1991non}
{\sc A.~Nordmark}, {\em Non-periodic motion caused by grazing incidence in an
  impact oscillator}, Journal of Sound and Vibration, 145 (1991), pp.~279--297.

\bibitem{nordmark1992effects}
{\sc A.~Nordmark}, {\em Effects due to low velocity impact in mechanical
  oscillators}, International Journal of Bifurcation and Chaos, 2 (1992),
  pp.~597--605.

\bibitem{nordmark2001existence}
{\sc A.~Nordmark}, {\em Existence of periodic orbits in grazing bifurcations of
  impacting mechanical oscillators}, Nonlinearity, 14 (2001), p.~1517.

\bibitem{PnueliPhd}
{\sc M.~Pnueli}, {\em Dynamics of Hamiltonian Impact Systems}, PhD thesis, The
  Weizmann Institute, 2020.

\bibitem{pnueli2018near}
{\sc M.~Pnueli and V.~Rom-Kedar}, {\em On near integrability of some impact
  systems}, SIAM Journal on Applied Dynamical Systems, 17 (2018),
  pp.~2707--2732.

\bibitem{pnueli2019structure}
{\sc M.~Pnueli and V.~Rom-Kedar}, {\em On the structure of {Hamiltonian} impact
  systems}, Nonlinearity, 34 (2021), pp.~2611--2658.

\bibitem{rom1999islands}
{\sc V.~Rom-Kedar and G.~Zaslavsky}, {\em Islands of accelerator modes and
  homoclinic tangles}, Chaos: An Interdisciplinary Journal of Nonlinear
  Science, 9 (1999), pp.~697--705.

\bibitem{turaev1998elliptic}
{\sc D.~Turaev and V.~Rom-Kedar}, {\em Elliptic islands appearing in
  near-ergodic flows}, Nonlinearity, 11 (1998), p.~575.

\bibitem{Wojtkowski1998}
{\sc M.~P. Wojtkowski}, {\em Hamiltonian systems with linear potential and
  elastic constraints}, vol.~157, 1998, pp.~305--341.
\newblock Dedicated to the memory of Wies\l aw Szlenk.

\bibitem{Wojtkowski1999}
{\sc M.~P. Wojtkowski}, {\em Complete hyperbolicity in {H}amiltonian systems
  with linear potential and elastic collisions}, in Proceedings of the {XXX}
  {S}ymposium on {M}athematical {P}hysics ({T}oru\'{n}, 1998), vol.~44, 1999,
  pp.~301--312.

\end{thebibliography}
%

\end{document}